\documentclass[final]{informs3} 
\usepackage{amsfonts,amssymb,amsmath} 
\usepackage{mathtools}
\usepackage{tikz}
\usepackage{url}
\usetikzlibrary{external,matrix,fit,external,arrows}
\usepackage{pgfplots}
\usepackage[utf8]{inputenc}
\usepackage[english]{babel}
\usepackage[toc,page]{appendix}

\newcommand{\hV}{\mathcal{V}}
\newcommand{\E}{\mathbb{E}}
\newcommand{\var}{\mathrm{Var}}
\newcommand{\hT}{\mathcal{T}}

\newcommand{\he}{\hat \sigma}

\newcommand{\U}{U}
\newcommand{\hh}{h}

\newcommand{\real}{\mathbb R}

\newcommand{\eg}{\textit{e.g. }}
\newcommand{\ie}{\textit{i.e. }}

\newcommand{\at}{a}		
\newcommand{\e}{\sigma}		
\newcommand{\p}{P}			
\newcommand{\rate}{r}
\newcommand{\st}{\tau}		
\newcommand{\hst}{\hat \tau}	
\newcommand{\x}{y}			
\newcommand{\y}{x}			

\newcommand{\uc}{c}		
\newcommand{\uind}{\bar u}
\newcommand{\off}{o}		

\newcommand{\C}{\delta}	
\newcommand{\D}{\epsilon}
\newcommand{\Q}{W}

\newcommand{\YDS}{\eqref{eq:YDS_optimization_problem}}

\newcommand{\hp}{\tilde v}	

\newcommand{\Z}{\mathbb Z_+}
\newcommand{\R}{\mathbb R}
\newcommand{\Rp}{\mathbb R_+}

\newcommand{\pint}{\Lambda}
\newcommand{\pinti}{\Lambda}
\newcommand{\fint}{\lambda}

\usepackage[linesnumbered,commentsnumbered,ruled,vlined]{algorithm2e}
\usepackage{cite}

\usepackage{mathtools}
\mathtoolsset{showonlyrefs}
\definecolor{ao}{rgb}{0.0, 0.0, 1.0}
\usepackage{subcaption}
\usepackage{tabularx}
\newcolumntype{L}[1]{>{\hsize=#1\hsize\raggedright\arraybackslash}X}%
\newcolumntype{R}[1]{>{\hsize=#1\hsize\raggedleft\arraybackslash}X}%
\newcolumntype{C}[2]{>{\hsize=#1\hsize\centering\arraybackslash}X}%

\SingleSpacedXI 

\TheoremsNumberedThrough     
\ECRepeatTheorems

\EquationsNumberedThrough    

\usepackage{changes}
\definechangesauthor[color=red]{andres}
\definechangesauthor[color=blue]{yorie}

\usepackage{color}
\newcommand{\re}[1]{\textcolor{black}{#1}}

\begin{document}

\RUNAUTHOR{Nakahira, Ferragut, and Wierman}

\RUNTITLE{Minimal-Variance Distributed Deadline Scheduling}

\TITLE{Generalized Exact Scheduling: a Minimal-Variance Distributed Deadline Scheduler}

\ARTICLEAUTHORS{%
\AUTHOR{Yorie Nakahira}
\AFF{Carnegie Mellon University, \EMAIL{yorie@cmu.edu}} 
\AUTHOR{Andres Ferragut}
\AFF{Universidad ORT Uruguay, \EMAIL{ferragut@ort.edu.uy}}
\AUTHOR{Adam Wierman}
\AFF{California Institute of Technology, \EMAIL{adamw@caltech.edu}}
} 

\ABSTRACT{

Many modern schedulers can dynamically adjust their service capacity to match the incoming workload. At the same time, however, unpredictability and instability in service capacity often incur operational and infrastructure costs. In this paper, we seek to characterize optimal distributed algorithms that maximize the predictability, stability, or both when scheduling jobs with deadlines. Specifically, we show that \emph{Exact Scheduling} minimizes both the stationary mean and variance of the service capacity subject to strict demand and deadline requirements. For more general settings, we characterize the minimal-variance distributed policies with soft demand requirements, soft deadline requirements, or both. The performance of the optimal distributed policies is compared to that of the optimal centralized policy by deriving closed-form bounds and by testing centralized and distributed algorithms using real data from the Caltech electrical vehicle charging facility and many pieces of synthetic data from different arrival distribution. Moreover, we derive the Pareto-optimality condition for distributed policies that balance the variance and mean square of the service capacity. Finally, we discuss a scalable partially-centralized algorithm that uses centralized information to boost performance and a method to deal with missing information on service requirements. 


}

\maketitle

\section{Introduction}
\label{sec:intro}

Traditionally, the scheduling literature has assumed a static or fixed service capacity.  However, it is increasingly common for modern applications to have the ability to dynamically adjust their service capacity in order to match the current demand. For example, power distribution networks match the energy supply demand as it changes over time and, when using cloud computing services, one can modify the total computing capacity by changing the number of computing instances and their speeds. 

\re{The ability to adapt service capacity dynamically gives rise to challenging new design questions. In particular, how to enhance the predictability and stability of service capacity is of great importance in such applications since peaks and fluctuations often come with significant costs 
\cite{stankovic1990predictability, melnik2010dremel, buttazzo2011hard}. For example, in the emerging load from electric vehicle charging stations, maintaining stable power consumption (\ie limiting the fluctuations in power consumption) is important because large peaks in power use may strain the grid infrastructure and result in a high peak charge for the station operators. The stations also prefer predictable power consumption (\ie knowing future power consumption) because purchasing power in real time is typically more expensive than purchasing in advance. Cloud content providers also prefer stable and predictable service capacity because on-demand contracts for compute instances (\eg Amazon EC2 and Microsoft Azure) are typically more expensive than long-term contracts. Additionally, significant fluctuations in service capacity induce unnecessary power consumption and infrastructure strain for computing equipment.} 

\re{Thus, in situations where service capacity can be dynamically adjusted, an important design goal is to reduce the costs associated with unpredictability and instability in the service capacity while maintaining a high quality of service, \eg meeting job deadlines and satisfying job demands. In this work, we study this problem by designing policies that minimize the variance of the service capacity in systems where jobs arrive with demand and deadline requests. Our model is motivated by power distribution networks, where the size of jobs and (active) service capacity are small compared to the total energy resources available and where contracts often depend on the mean and variance of service capacity, \eg if a charging station participates in the regulation market, then costs/payments rely explicitly on them \cite{spees2007demand,behrangrad2015review}.}

\re{Although the literature on deadline scheduling is large and varied, optimal algorithms are only known for certain niche cases. We review some of these results below in the related work section below. We emphasize however that only recently have researchers approached the task of designing algorithms that balance service quality and costs associated with variability. Much of the work on this topic has been application-driven, particularly in the areas of cloud computing and power distribution systems. As we mentioned before, in these areas service capacity is indeed elastic by design, and variability has direct cost implications. In this regard, no general optimality results have been proven so far about how to balance service quality and cost, except in some limited settings, such as deterministic worst-case settings~\cite{bansal2007speed}, single server systems~\cite{panwar1988optimal,panwar1988optimality,bhattacharya1989optimal}, and/or heavy traffic settings \cite{lehoczky1997using,gromoll2007heavy}. In heavy traffic settings, the dynamic behavior can be approximated by a continuous-state process involving Brownian motion, for which there exist established tools to optimize. On the other hand, optimizing queueing systems without continuous-state approximations remains a hard problem. Solving this problem is a challenging task due to the heterogeneity of jobs (diversity in service requests) and the size of the state and decision space (numbers of possible configurations on existing job profiles and the set of feasible control policies).} 

\re{In this paper, our goal is to derive general optimality results that hold beyond the heavy-traffic regime.  Further, we seek to design optimal \emph{distributed algorithms}, which only use local information about each job to decide the desirable service rate. Those algorithms are particularly useful for large systems such as power distribution networks and cloud computing, where implementing centralized algorithms is likely to be prohibitively slow and costly in large-scale service systems, \ie we are unlikely to be able to access global information about all jobs and servers in real time when deciding the service rate of individual jobs. Despite this constraints in information sharing, we show that, interestingly, the optimal distributed algorithms under mild assumptions have comparable performance to centralized algorithms. }


\paragraph{Contributions of this paper.} 

In this paper, we adapt tools from optimization and control theory to characterize  the optimal distributed policies in a broad range of settings without any approximation. Further, we provide a novel competitive-ratio-like bounds that describes the gap between the performance of optimal distributed policies and the performance of optimal centralized policies. 

Specifically, we identify the optimal distributed algorithms under strict demand and deadline requirements (Theorem \ref{thm:1_stationary_problem1}), soft demand requirements (Theorem \ref{thm:stationary_softdemand}), soft deadline requirements (Theorem \ref{thm:stationary_softdeadline}), and soft demand and deadline requirements (Theorem \ref{thm:stationary_softdemand+deadline}) in settings with stationary Poisson arrivals as well as non-stationary Poisson arrivals (Theorem \ref{thm:nonstationary_maxpred_hard} and Corollary~\ref{cor:nonstationary_maxpred_soft}).

Our first results focus on stationary arrivals. While a considerable amount of work has analyzed the variance of specific policies (see \cite{ferragut2017controlling} and references therein), little prior work characterizes the optimal policies. In the basic setting of strict service requirements, we show that \textit{Exact Scheduling} is the optimal distributed algorithm, \ie the distributed algorithm that minimizes the stationary service capacity variance. Exact Scheduling is a simple scalable algorithm that works by finishing jobs \textit{exactly} at their deadlines using a constant service rate~\cite{liu1973scheduling,buttazzo2011hard,ferragut2017controlling}.  Although it has received considerable attention in the existing literature, its optimality conditions have been unknown. In more general settings of soft service requirements, we propose novel generalizations of Exact Scheduling, each of which minimizes a combination of the service capacity variance, the expected penalties for unsatisfied demands, and the expected penalties for deadline extensions. These optimal algorithms all have closed-form expressions and use constant service rates with varying forms of rate and admission control. Due to these properties, they are also easy to implement and highly scalable.  


We also extend our results to the case of non-stationary Poisson job arrivals and characterize the pareto-optimal algorithm that balances service capacity variance, penalties for unsatisfied demands, and penalties for deadline extensions. Additionally, we consider a more general class of objective functions: the service capacity variance, the mean-squared service capacity, and the weighted sum of the two. The resulting optimal algorithm has a striking analogy to the YDS algorithm~\cite{yao1995scheduling}, which is an offline algorithm that minimizes service capacity peaks in a related, deterministic worst-case version of the problem. This connection suggests the opportunity to transform other deterministic offline algorithms to stochastic online algorithms in related problems.

Given our focus on \textit{distributed} algorithms, an important question is how these distributed algorithms perform compared with the optimal centralized algorithm. However, a major difficulty comes from the fact that the optimal centralized algorithms are unknown and no bounds on the optimal cost exist. Leveraging tools from optimal control, we provide closed-form formulas on the performance degradation due to using distributed algorithms (Lemma \ref{thm:lower_bound} and Corollary \ref{thm:lower_bound2}).  The resulting bounds suggest that, when sojourn times are homogeneous, Exact Scheduling attains the optimal trade-off between service capacity variance and total remaining demand variance achievable by any centralized algorithms.  Note that our proof technique (Lemma \ref{thm:lower_bound}) is novel in its use of optimal control and has the potential for providing competitive-ratio-like bounds for other scheduling policies. We also compare distributed algorithms with centralized algorithms in our motivating examples of electric vehicle charging. Our test in Caltech electric vehicle charging testbed~\cite{testbed} shows that the proposed optimal distributed algorithms also achieve comparable performance with existing centralized algorithms in practice. 

\paragraph{Related work.}

\re{There is an extensive literature that studies the design and analysis of deadline scheduling algorithms (see \cite{kleinrock1975queueing,baccelli,buttazzo2011hard,stankovic2012deadline} and references therein). 
Examples of classic scheduling algorithms include Earliest Deadline First~\cite{panwar1988optimal,towsley1989,panwar1988optimality,kruk2011heavy,moyal2013queues,bhattacharya1989optimal} and Least Laxity First~\cite{towsley1989}, among others~\cite{pinedo1983stochastic,lehoczky1997real}. 
Beyond these classic algorithms, more modern algorithms simultaneously perform admission control and service rate control to exploit the flexibility arising from soft demand or deadline requirements, \eg~\cite{plambeck2001multiclass, maglaras2005queueing, ccelik2008dynamic}. }

\re{The trade-offs between service quality and costs associated with variability have become a focus only recently~\cite{dean2013tail,boutin2014apollo,ferragut2017controlling}, motivated by applications such as cloud computing and power distribution systems. In the context of cloud computing, algorithms have been proposed to control the variability of power usage in data centers using deferrable jobs (see~\cite{zhu2010resource,adnan2012energy, adnan2014workload,lin2013dynamic, wang2008cluster, kusic2009power,mukherjee2009spatio,chen2005managing,vaquero2011dynamically, liu2011greening,gandhi2013dynamic,gandhi2009optimal} and references therein). In the context of power distribution systems, algorithms have been designed to control the variability of energy supply using deferrable loads (see~\cite{poolla_flexibility2013,poolla_scheduling2013,gan2013optimal,chen2014distributional,binetti2015scalable} and references therein). }

\re{Interesting optimality results have been obtained in some limited settings, such as deterministic worst-case settings~\cite{yao1995scheduling,bansal2007speed}, single server systems~\cite{panwar1988optimal,panwar1988optimality,bhattacharya1989optimal}, and/or heavy traffic settings \cite{lehoczky1997using,gromoll2007heavy}. For example, in heavy traffic settings, the dynamic behavior of discrete queueing systems can be approximated by a continuous-state process involving Brownian motion, for which there exist established tools to optimize~\cite{lehoczky1997using,gromoll2007heavy}. On the other hand, optimizing queueing systems without continuous-state approximations remains to be a hard problem. Particularly, the problem of designing \textit{optimal} algorithms that minimize service capacity variability while achieving high service quality has remained open. Solving this problem is a challenging task due to the heterogeneity of jobs (diversity in demands and deadlines) and the size of the state and decision space (of possible configurations on existing job profiles and the set of feasible scheduling policies).}

\re{However, the problem of designing \textit{optimal} algorithms that minimize service capacity variability while achieving high service quality has remained open.  Solving this problem is a challenging task due to the heterogeneity of jobs (diversity in service requests) and the size of the state and decision space (numbers of possible configurations on existing job profiles and the set of feasible control policies).  In particular, the only optimality results that have been obtained to this point are in niche settings such as a static single server system~\cite{panwar1988optimal,panwar1988optimality,bhattacharya1989optimal} and deterministic worst-case settings~\cite{yao1995scheduling,bansal2007speed}. }

\section{Model description}
\label{sec:optimization_problems}

The goal of this paper is to characterize the online scheduling policies that minimize service capacity variance, mean square, and both subject to service quality constraints for systems with the ability to dynamically adjust their service capacity. We use a continuous time model, where $t \in \hT = [0, T]$ denotes a point in time and $T \geq 0$ is the (potentially infinite) time horizon. Each job, indexed by $k \in \hV = \{ 1, 2, \cdots\}$, is characterized by an arrival time $\at_k$, a service demand $\e_k$, a sojourn time $\st_k $, a unit cost for unsatisfied demand $\C_k$, and a unit cost for deadline extension $\D_k$. Given the arrival time $\at_k$ and the sojourn time $\st_k$, the deadline of job $k$ is defined to be $\at_k + \st_k$. Before we formulate the scheduler design problem, we first introduce below the arrival profiles, the service profiles, and the design objectives.

\textit{Arrival profiles.}
We represent the set of jobs as a marked point process $\{(\at_k;  \e_k, \st_k, \C_k, \D_k )\}_{k \in \hV}$ in space $\hT \times S \times C$, where the arrival times $\at_k \in \hT$ are the set of points, and the service requirements $(\e_k, \st_k ) \in S$ and costs for unmet requirements $(\C_k, \D_k ) \in C$ are the set of marks.\footnote{Here, we use $(\at;  \e, \st, \C, \D)$ to denote the random variables and $(\at_k;  \e_k, \st_k. \C_k, \D_k)$ to denote one realization of them in job $k$.} We assume that the point process is an independently marked Poisson point process, which is defined by an intensity function $\tilde \pint (a)$  on $\hT$ and a mark joint density measure $f_a(\e, \st) g_a( \C ) h_a (\D) $ on $S \times C$~\cite{baccelli}. This also implies that $\{(\at_k;  \e_k, \st_k )\}_{k \in \hV}$ is a Poisson point process on $\hT \times S$ with the intensity function
$
\pint(\at,  \e, \st) =   \tilde \pint(\at)f_\at(\e, \st) .
$
Intuitively, the intensity function is the average rate at which jobs with service requirement $(\e,\st)$ arrive at time $a$. When both $\tilde \pint(\at) \equiv \tilde \pint $ and $f_\at(\e, \st) \equiv f(\e, \st)$ do not depend on $\at$, we say that the arrival distribution is stationary . For a stationary  arrival distribution, the intensity function of the Poisson point process simplifies to $\pint f(\e,\st)$.  We focus on stationary  arrival processes in Section \ref{sec:stationary_systems} and then generalize our results to non-stationary arrivals in Section \ref{sec:nonstationary}.  Throughout, we assume that the service demand $\e$ and the sojourn time $\st$ has finite first and second moments, $S$ is bounded, and $S \subset \{ (\e, \st) :  \st \geq \e \text{ and }  \e \geq 0  \}$.\footnote{\label{note0} The condition $\st \geq \e$ constrains the service demand $\e$ of a job to be no more than the amount of service that can be provided within its sojourn time $\st$.}  

\textit{Service profiles.} The service system works on each job $k \in \hV$ with a \emph{service rate} $r_k(t)$, which is an integrable function of $t$. The service rate can take any non-negative values that are smaller than the maximum rate $\bar r$, and without loss of generality, we assume that $\bar r =1$, \ie
\begin{align}
\label{eq:rate_constraints}
r_k(t) \in [0, \;1].
\end{align} 
To meet the demand requirements, the service rate must satisfy
\begin{align}
&\int_{\at_k}^{\infty} r_k (t) dt = \e_k,  & k\in \hV \label{eq:demand_constraints}. 
\end{align}
To meet the deadline requirements, it also need to satisfy
\begin{align}
&r_k (t) \leq  \mathbf{1}\{  \at_k \leq  t < \at_k+\st_k      \}  , &  k\in \hV \label{eq:deadline_constraints}. 
\end{align}
where $\mathbf{1}\{A\}$ denotes the indicator function for an event $A$. 

The service rate also determines three important quantities associated with costs: service capacity, the amount of unsatisfied demands, and the amount of deadline extensions. The service capacity is the instantaneous resource consumption of the system, given by 
\begin{equation}\label{eq:power_def}
P(t) = \sum_{k\in \hV} r_k(t).
\end{equation}
We assume that $P(t)$ has no upper bound, implying that there is always enough capacity to serve the jobs. 
The total penalty for unmet demands of jobs with deadline $t$ is 
\begin{equation}
\U(t) =  \sum_{ k  \in \hV : \at_k + \st_k = t } \C_k ( \e_k -  \he_k ),
\end{equation}
where $\he_k = \int_{\at_k}^{\infty} r_k (t) dt$ is defined to be the unsatisfied demands of job $k$. The total penalty for deadline extensions of jobs with deadline $t$ is 
\begin{equation}
\Q(t) =  \sum_{ k  \in \hV : \at_k + \st_k = t } \D_k ( \hst_k -  \st_k ).
\end{equation}
where $\hst_k = \max\{ t -\at_k : r_k(t) >0 \}$ is defined to be the actual sojourn time of job $k$.


\textit{Design objectives.} We consider designing \textit{online} scheduling algorithms, which decide the service rates in real-time without using the future job arrival information. For scalability, we restrict our attention to \textit{distributed} algorithms which only need local information about each job to decide its service rate. Examples of online distributed algorithms are \emph{Immediate Scheduling}, \emph{Delayed Schedule}, and \emph{Exact Scheduling} (see Figure~\ref{fig:charging_examples}).

\re{Recall from Section \ref{sec:intro} that the predictability and stability of service capacity are important design criteria for modern schedulers because peaks and fluctuations in service capacity strain the system infrastructure and knowing the future demand of the service capacity help reduce cost. Thus, our design objective is to reduce the variance and mean square in service capacity for the settings with strict or soft service constraints. Specifically, we consider the optimization problem
\begin{align}
& \text{minimize} \;\;\;\;\; \frac{1}{T} \int_{0}^{T} \Big( \alpha    \E [ P(t) ]^2 + \beta \var (P(t) )  \Big) dt
\end{align}
where the first term of the integrand quantifies the service capacity stability, and the second term quantifies the service capacity predictability. The coefficient $\alpha, \beta (\geq 0)$ balances the predictability and stability of $P(t)$, and the objective function reduces to the time average of $ \E[ P(t)^2 ]$ at $(\alpha , \beta ) = (1, 1)$.} 

\re{We also consider the case when the service requirements do not need to be perfectly satisfied. In such cases, we consider the optimization problem
\begin{align}
& \text{minimize} \;\;\;\;\;   \frac{1}{T} \int_{0}^{T} \Big( \var (P(t) )  + \E[\U(t) ]+ \E[\Q(t)] \Big) dt ,
\end{align}
which balances service capacity variance with the penalties for not meeting the demands and/or deadlines of some jobs.}

\begin{figure*}[h!]
	\begin{center}
		\begin{tikzpicture}[scale=0.4]
		\tikzset{font=\footnotesize}
	
		\draw [->] (0,0) -- (9,0);
		\node at (9.5,0) {$t$};
		\draw [->] (0,0) -- (0,7.9);
		\node at (0,-0.5) {$\at_k$};
		\node at (-0.5,7) {$\tau_k$};
		\node at (-0.5,4) {$\sigma_k$};
		\node at (7,-0.5) {$\at_k+\tau_k$};
		\draw [dotted,blue,->] (0,4) -- (4,0);
		\node at (4,-0.5) {$\at_k+\sigma_k$};
		\draw [black,thick,->] (0,4) -- (4,0); 
		\node at (4.5,-2.7) {$r_k(t)= \mathbf{1} \{\at_k \leq t < \at_k+ \sigma_k \}$};
		\node at (0,8.5) {$x(t)$};
		\node at (4.5,-1.7) {Immediate Scheduling};
		\end{tikzpicture}\quad \quad
		\begin{tikzpicture}[scale=0.4]
		\tikzset{font=\footnotesize}

		\draw [->] (0,0) -- (9,0);
		\draw [dotted] (0,7) -- (7,0);
		\node at (9.5,0) {$t$};
		\draw [->] (0,0) -- (0,7.9);
		\node at (0,-0.5) {$\at_k$};
		\node at (-0.5,7) {$\tau_k$};
		\node at (-0.5,4) {$\sigma_k$};
		\node at (7,-0.5) {$\at_k+\tau_k$};
		\draw [black,thick,->] (0,4) -- (3,4) -- (7,0);
		\node at (4,-0.5) {$\at_k+\sigma_k$};
		\node at (4.5,-2.7) {$r_k(t)= \mathbf{1} \{\at_k + (\tau_k - \sigma_k) \leq t < \at_k+ \tau_k \}$};
		\node at (0,8.5) {$x(t)$};
		\node at (4.5,-1.7) {Delayed Scheduling};
		\end{tikzpicture}\quad \quad
		\begin{tikzpicture}[scale=0.4]
		\tikzset{font=\footnotesize}

		\draw [->] (0,0) -- (9,0);
		\node at (9.5,0) {$t$};
		\draw [->] (0,0) -- (0,7.9);
		\node at (0,-0.5) {$\at_k$};
		\node at (-0.5,7) {$\tau_k$};
		\node at (-0.5,4) {$\sigma_k$};
		\node at (7,-0.5) {$\at_k+\tau_k$};
		\draw [black,thick,->] (0,4) -- (7,0);
		\node at (4,-0.5) {$\at_k+\sigma_k$};
		\node at (4.5,-2.7) {$r_k(t)=\frac{\sigma_k}{\tau_k}\mathbf{1}\{\at_k \leq   t < \at_k+ \st_k \}$};
		\node at (0,8.5) {$x(t)$};
		\node at (4.5,-1.7) {Exact Scheduling};
		\end{tikzpicture}\quad \quad
		\begin{tikzpicture}[scale=0.4]
		\tikzset{font=\footnotesize}

		\end{tikzpicture}
	\end{center}
	\caption{Examples of distributed scheduling algorithms. The solid black lines represent the remaining demand $x(t)$ at time $t$. Immediate Scheduling works by serving jobs at full rate upon arrival. Delayed Scheduling works by serving at full rate with a delay that is equal to its laxity $\at + \st -\e$. 
Exact Scheduling works by throttling service to a constant rate $ \sigma/\tau$ so that all jobs are completed exactly at its deadline.}
	\label{fig:charging_examples}
\end{figure*}
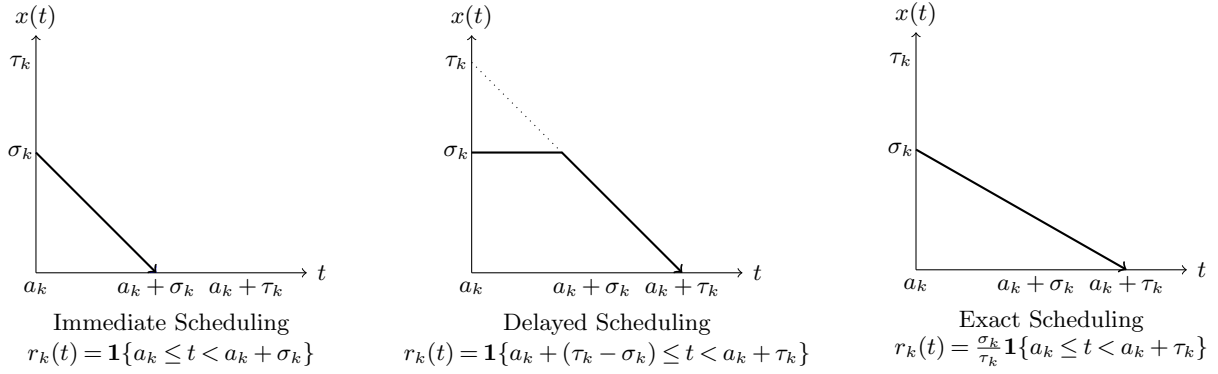

\textit{Motivating examples.} 
{\color{black}
The general model we have defined is meant to give insight into the design trade-offs that happen in applications with dynamic capacity, \eg electric vehicle charging, cloud content providers. Importantly, in this paper we are not trying to model a
specific application, rather we are exploring design trade-offs using a simple, general model. However, to highlight the connection to our motivating examples, consider first the case of electric vehicle charging. In this case, each job $k \in \hV$ corresponds to an electric vehicle with an arrival time $\at_k$, an energy demand $\sigma_k$, and a sojourn time $\tau_k$. At each time $t$, the charging station draws $P(t) = \sum_{k \in \hV} r_k(t)$ amount of power from the grid to provide each vehicle $k$ with a charging rate of $r_k(t)$. When doing so, stable resource usage is highly desirable because fluctuations and large peaks in $P(t)$ can strain the grid and results in a high peak charge for station operators. Moreover, predictable resource usage is also important when purchasing energy from the day-ahead market, whose price is lower and less volatile than that of the real-time market. Note that our model assumes $P(t)$ is unbounded and, thus corresponds to a setting where there are more charging stations
than arriving cars. 


In the case of cloud content providers, each job $k \in \hV$ corresponds to a task (requested to the cloud or data centers) with an arrival time $\at_k$, a work requirement $\sigma_k$, and an allowable waiting time $\tau_k$. The service system works on job $k$ with speed $r_k(t)$ using $P(t) = \sum_{k \in \hV} r_k(t)$ number of computers (or amount of power). Here again, a good estimate of the future resource use enables the cloud users to reserve resources through a long-term contract, whose price is lower and less volatile than that of a short-term contract, suggesting the benefit of having a predictable resource use. Note that our model considers the case
where $P(t )$ is unbounded and, thus, the data center has enough capacity to avoid congestion, \ie is in low utilization. Such periods are common, since data centers often operate at utilizations as low as 10$\%$ \cite{greenberg2008cost}. For
future work, it is important to study how to manage congested periods by considering an upper bound on $P(t )$.

In this paper, we primarily focus on the cases when the arrival times, demands, and sojourn times are all available to the scheduler upon the arrival of each job. Such cases are common in many scheduling problems and modern applications operating on an increasingly smarter infrastructure~\cite{brucker2007scheduling,testbed,lucier2013efficient,gan2013optimal,towsley1989,yao1995scheduling,bansal2007speed}. For example, in the electric vehicle charging testbed~\cite{testbed}, the users input the service request ($\sigma_k$, and $\tau_k$) through a control panel upon arrival. In cloud computing, the demands can be estimated from the past, and deadlines are determined by operational/performance requirements~\cite{lucier2013efficient}. Beyond the case of our primary focus, there are also situations when the information on demands and deadlines cannot be accessed for some or all jobs. For such cases, we discuss the algorithm to be used and its performance analysis in Section \ref{sec:emp}.

}

\section{Maximizing predictability under stationary job arrivals}
\label{sec:stationary_systems}

In this section, we characterize optimal distributed scheduling policies in a wide range of objectives when job arrivals are stationary, starting with the simplest and moving toward the most complex. To begin, we define each setting and pose the scheduler design problems as constrained functional optimizations (Section \ref{sec:stationary_model}). Then, we focus on strict service requirements and show that Exact Scheduling minimizes the stationary variance of the service capacity (Section \ref{sec:hard}). Relaxing the demand requirements, we show that a variation of Exact Scheduling minimizes the weighted sum of the stationary variance of service capacity and the penalty for unsatisfied demand (Section \ref{sec:soft_demand}). Relaxing the deadline requirements, we show that a different variation of Exact Scheduling minimizes the weighted sum of both the stationary variance of service capacity and the penalty for demand extension (Section \ref{sec:soft_demand}). Finally, we consider the case when both the demand and deadline requirements are relaxed (Section \ref{sec:soft_demand_deadline}) and show that the optimal policy can be constructed from an integration of the above optimal policies. It is interesting that all optimal algorithms admit closed-form expressions, which provide clear interpretations and insights regarding the optimal trade-offs between reducing service capacity variability, satisfying the demands, and meeting deadlines. Moreover, they are also highly scalable and easy to implement.

\subsection{Problem formulation} 
\label{sec:stationary_model}
\re{We study the settings when the arrival process is an independently marked \textit{stationary} Poisson point process. The intensity function of the process is
\begin{align}
\label{eq:arrival_intensity_stationary}
&\pint(\at,\e,\st)  = \pint  f(\e,\st)  , &\quad \at \in \hT , ( \e , \st ) \in S, 
\end{align}
where $\pint(\at,\e,\st)$ takes the same value for different $\at$ given fixed $\e,\st$.} We first consider the case when the unit cost for unsatisfied demands $\C_k$ and that for deadline extensions $\D_k$ are deterministic and homogeneous among different jobs, \ie $
\C_k = \C$, $\D_k = \D$ for any $k \in \hV$.\footnote{This assumption is relaxed in Corollary \ref{lem:stationary_softdemand+deadline}.} we consider distributed scheduling policies of the form
\begin{align}
\label{eq:u_stationary_simple}
r_k(t) = u(  \y_k(t)  , \x_k (t)  )  \geq 0
\end{align}
where $u: \Rp \times \R  \rightarrow \Rp$ is a non-negative integrable function of the remaining demand $ \y_k(t)$ and the remaining time $ \x_k(t)$ of job $k$. \re{This policy assumes that the system can access the information about the service demands and deadlines. This assumption holds, for example, in electric vehicle charging systems~\cite{testbed}.} The policy is also distributed in the sense that service rate of a job is determined using only its own information. We study policies of the form \eqref{eq:u_stationary_simple} assuming a situation where there is enough capacity available to satisfy the demand, and so the focus is on determining the optimal service rate for the jobs in a distributed manner.

\re{In the special case when Immediate Scheduling policy is used, the system becomes $M/G/\infty$ queue.\footnote{In general, the system under policy \eqref{eq:u_stationary_simple} may be different from $M/G/\infty$ queue because $M/G/\infty$ requires the service rate to be constant.} 
More generally, under any policy of the form \eqref{eq:u_stationary_simple}, the remaining job process $\{ (\y_k(t),\x_k(t)) : k \in \hV, \at_k \leq t \}$ can be represented as a point process in a two-dimensional space of remaining times and remaining demands \cite{baccelli}. As $t\rightarrow \infty$, the process $\{ (\y_k(t),\x_k(t)) : k \in \hV, \at_k \leq t \}$ converges to a stationary point process whose distribution is determined by the job profiles and scheduling policy. Moreover, it is a Poisson process in the space with mean measure $ \fint( \y  , \x )$ satisfying:}\footnote{We use $(\y,\x)$ to denote the coordinate in the two dimensional space of remaining demands and remaining times and $(\y_k(t),\x_k(t))$ to denote a point (job profile) in the space at time $t$.}
	\begin{align}
	\label{eq:conservation}
	0 = \frac{\partial }{\partial \y} (  \fint (\y,\x) u (\y,\x) )  + \frac{\partial }{\partial \x} \fint (\y,\x)  + \pint  f(\y,\x) . 
	\end{align}
\re{The above equation is also known as the continuity equation and can be derived from the movement and conservation of density of the Point Process~ \cite{pedlosky2013geophysical,zeballos2019proportional}. The movement of each point $(\y  , \x )$ has velocity $-u(\y,\x)$ in the $\y$-dimension and velocity $-1$ in the $\x$-dimension because its remaining demand is reduced by $u(-\y, \x)$ per unit time, and its remaining time is reduced by $1$ per unit time. The conservation of density states that the flow inward (of existing jobs) and new arrivals minus flow outward through the surface of a region sum up to be zero. }

\re{Because the remaining job process becomes stationary as $t\rightarrow \infty$, the distribution of $P(t)$ also becomes stationary. Moreover, its stationary mean $\E[ P ]$ is determined only by the total service provided. For example, in the special case when the demand constraints are to be strictly satisfied, we have $\E[ P ] =\pint  \E[\e]$. In a more general setting, the stationary mean is given in the following proposition. 
}
\begin{proposition}
\label{prop:poisson_P}
Consider a service system with a stationary Poisson arrivals with intensity measure $\pinti f(\e,\st) $ and a distributed scheduling policy of the form \eqref{eq:u_stationary_simple}. Let us define $\he(\e, \st)$ to be the total service a job with demand $\e$ and a sojourn time $\st$ receives.\footnote{Here, $\e, \st$ are random variables, and $\he(\e, \st)$ is the output of the function with input $(\e, \st)$. So $\he(\e, \st)$ is also a random variable.} The stationary mean of $P(t)$ is given by 
\begin{align}
\label{eq:st_meanP_10}
\E[ P(t) ] = \pint  \E[\he(\e,\st)] . 
\end{align}
\end{proposition}
We present a proof of Proposition \ref{prop:poisson_P} in Appendix \ref{sec:poisson_P_proof}. Alternatively, it can also be derived from classical queueing results such as Little's Law and the Brumelle's formula~\cite[Chapter 3, eq. (3.2.1)]{baccelliBremaud}. 

As the stationary mean of $P(t)$ does not depend on the specific form of the policy \eqref{eq:u_stationary_simple}, we consider minimizing the stationary variance of $P(t)$ under strict service constraints, soft demand constraints, soft deadline constraints, and soft demand and deadline constraints. In the case of \textit{strict demand constraints}, we consider the following optimization problem:
\begin{align}\label{eq:min_variance_st}
& \underset{u : \eqref{eq:rate_constraints} \eqref{eq:demand_constraints}\eqref{eq:deadline_constraints}\eqref{eq:u_stationary_simple}\eqref{eq:conservation}}{\text{minimize}} \;\;\;\;\;  \var(P),
\end{align}
where the optimization variable taken over the set of distributed policies \eqref{eq:u_stationary_simple} subject to the service rate constraints \eqref{eq:rate_constraints}, the demand constraints \eqref{eq:demand_constraints}, and the deadline constraints \eqref{eq:deadline_constraints}. Here, $\var( P )$ is a functional of $u$ and $ \fint( \e  , \st)$, where $ \fint( \e  , \st)$ satisfies \eqref{eq:conservation}.

In the case of \textit{soft demand constraints}, we relax the demand constraints \eqref{eq:demand_constraints} into paying penalty $\C_k = \C$ for each unit of unsatisfied demands and consider balancing the service capacity variance and the penalties due to unsatisfied demands:
\begin{align}
\label{eq:min_variance+softdemand_st}
& \underset{u : \eqref{eq:rate_constraints}\eqref{eq:deadline_constraints}\eqref{eq:u_stationary_simple}\eqref{eq:conservation}}{\text{minimize}} \;\;\;\;\;\;\;\;  \mathrm{Var}(P)+  \E[ \U ]  . 
\end{align}

In the case of \textit{soft deadline constraints}, we relax the deadline constraints \eqref{eq:deadline_constraints} into paying penality $\D$ for each unit of deadlines extensions and consider balancing the service capacity variance and the penalties due to deadline extensions:
\begin{align}
\label{eq:min_variance+softdeadline_st}
& \underset{u : \eqref{eq:rate_constraints}\eqref{eq:demand_constraints} \eqref{eq:u_stationary_simple}\eqref{eq:conservation}}{\text{minimize}} \;\;\;\;\;\;\;\;  \mathrm{Var}(P)+  \E[ \Q ]  .
\end{align}

In the case of \textit{soft demand and deadline constraints}, we relax both the demand and deadline requirements \eqref{eq:demand_constraints} and \eqref{eq:deadline_constraints} into paying $\C$ for each unit of unsatisfied demands and $\D$ for each unit of deadline extensions. We consider balancing the service capacity variance and the penalties due to unsatisfied demands and deadline extensions:
\begin{align}
\label{eq:min_variance+softdemand+softdeadline_st}
& \underset{u : \eqref{eq:rate_constraints} \eqref{eq:u_stationary_simple}\eqref{eq:conservation} }{\text{minimize}} \;\;\;\;\;\;\;\;  \mathrm{Var}(P)+   \E[ \U ] + \E[ \Q ]  .
\end{align}

Finally, we consider the most general setting, when the penalties for unsatisfied demands and deadlines are heterogeneous among jobs. To account for this heterogeneity, we consider distributed scheduling policies of the form
\begin{align}
\label{eq:u_stationary_simple_marked}
r_k(t) = \uind(  \y_k(t)  , \x_k (t) , \C_k ,\D_k  )  \geq 0 .
\end{align}
Under any policy of the form \eqref{eq:u_stationary_simple_marked}, the remaining job profiles in the system $\{ (\y_k(t),\x_k(t) ,\C_k ,\D_k  ) : k \in \hV, \at_k \leq t \}$ can be represented as a point process in the 4-dimensional space of remaining times, remaining demands, unit costs for unsatisfied demand, and unit costs for deadline extension. This point process converges to a stationary Spatial Poisson Point Process with an intensity function $\lambda(\y, \x, \C ,\D)$ satisfying  
\begin{align}
\label{eq:conservation_hete_penality}
&0 = \frac{\partial }{\partial \y} (  \fint (\y,\x,\C ,\D) \bar u (\y,\x,\C ,\D) )  + \frac{\partial }{\partial \x} \fint (\y,\x,\C ,\D)  + \pint  f(\y,\x) g(\C) h(\D).
  \end{align}
This leads to the following optimization problem: 
\begin{align}
\label{eq:min_variance+softdemand+softdeadline_st_marked}
\underset{\uind: \eqref{eq:rate_constraints}\eqref{eq:u_stationary_simple_marked}\eqref{eq:conservation_hete_penality} }{\text{minimize}}& \;\;\;\var (P(t) ) + \E\left[ \U(t) \right] + \E\left[ \Q(t) \right]  .
\end{align}

\subsection{Strict demand and deadline requirements}
\label{sec:hard}
 We first consider the case of strict service requirements and show a closed-form characterization of the optimal algorithm that minimizes the stationary variance $\var(P)$. To do so, it is worth noticing from Lemma \ref{prop:poisson} that service rates contribute to the service capacity variance in a quadratic form. Thus, having a large value in the service rate, \ie $u( \y , \x )$ taking large values for some $( \y , \x )$, results in disproportionately more service capacity variance. This observation suggests that having a flat service rate may achieve small variance. One such policy is \textit{Exact Scheduling}, 
 \begin{align}
\label{eq:exact_scheduling_st}
u( \y , \x )  = \begin{dcases} 
\frac{ \y }{ \x},  & \text{if } \x > 0 ,  \\
0 , & \text{otherwise}. 
 \end{dcases} 
\end{align}
which works by finishing each job \textit{exactly} at its deadline using a constant service rate (Figure~\ref{fig:ES}). It is also highly scalable because it is distributed, and it does not require much computation,  memory use, communication, or synchronization. Although existing literature has analyzed its performance in various settings~\cite{liu1973scheduling,lehoczky1989rate,buttazzo2011hard,ferragut2017controlling}, no work has shown its optimality conditions. In this section, we show that Exact Scheduling minimizes the stationary service capacity variance under strict demand and deadline constraints.

\begin{figure}
\begin{center}
\begin{tikzpicture}[scale=1.2]
 \draw [->] (-.2,0) -- (4.5,0);
 \node at (2.25,-.5) {Remaining demand ($\y$)};
 \draw [->] (0,-.2) -- (0,4);
 \node [rotate=90,centered] at (-.5,2) {Remaining time ($\x$)};
 \draw [blue,fill=green] (1,3.5) circle [radius=0.08];
 \draw [blue,thick,->] (1,3.5)--(.92,3.17);
  \draw [dotted,thin] (0,0)--(.92,3.17);
\draw [blue,fill=green] (1,2) circle [radius=0.08];
 \draw [blue,thick,->] (1,2)--(.85,1.75);
 \draw [dotted,thin] (0,0)--(.85,1.75);
 \draw [blue,fill=green] (3,4) circle [radius=0.08];
 \draw [blue,thick,->] (3,4)--(2.8,3.75);
 \draw [dotted,thin] (0,0)--(2.8,3.75);
 \draw [blue,fill=green] (2.5,2.8) circle [radius=0.08];
 \draw [blue,thick,->] (2.5,2.8)--(2.3,2.55);
  \draw [dotted,thin] (0,0)--(2.3,2.55);
  \draw [opacity=0, fill opacity=0.1, fill=gray] (0,0) -- (4,4) -- (4.5,4) -- (4.5,0) -- cycle;
  \draw [gray] (0,0) -- (4,4);
  \node [above left, anchor=south east, align=center] at (4,0) {Infeasible\\ region};
 \end{tikzpicture}
 \end{center}
\caption{\emph{Exact scheduling depicted in the space of remaining demand $\y$ and remaining time $\x$.}}
 \vspace{-5mm}
\label{fig:ES}
\end{figure}
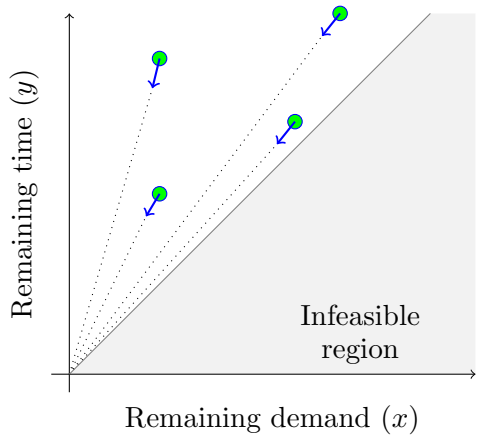

\begin{theorem}
\label{thm:1_stationary_problem1}
Exact Scheduling \eqref{eq:exact_scheduling_st} is the optimal solution of \eqref{eq:min_variance_st} and achieves the optimal value\footnote{Observe that $\pinti $ is the cumulative arrival rate.}
\begin{align}
\label{eq:1_stationary_problem1}
&\var(P) = \pint   \E\left[ \frac{ \e^2 }{ \st } \right] .
\end{align}
\end{theorem}

\re{Theorem \ref{thm:1_stationary_problem1} shows that the optimal policy for minimizing variance is to keep a constant service rate at all times. Therefore, when considering strict demands an deadlines, the optimal policy is to have a flat service rate across its sojourn time $\st_k$.} Additionally, Theorem \ref{thm:1_stationary_problem1} shows the achievable performance improvement by controlling the service capacity using distributed algorithms. If no control is applied, then $r_k(t) = \mathbf{1}\{t \in [\at_k,\at_k+\sigma_k)\}$, and the stationary mean and variance of $P(t)$ is $\E(P) =\var(P)  = \pinti  \E[\e]$ By performing a distributed service capacity control, the stationary variance can be reduced by 
\begin{align}
\pinti  \E\left[ \frac{\e (\st - \e) }{\st} \right] \in \big[ 0, \pinti  \E[ \e]\big] 
\end{align}
where $\st - \e$ is a slack time (the amount of time left at job completion if a job is served at its maximum service rate).

Next, we present the proof of Theorem \ref{thm:1_stationary_problem1}. To circumvent the complex constraints of \eqref{eq:min_variance_st}, we first provide a lower bound on its optimal solution by relaxing the class of control policies into 
\begin{equation}
\label{eq:u_stationary_complex}
\rate_k(t) = v(   \e_k ,   \st_k , \x_k (t)   )  \quad k\in \hV,
\end{equation}
and solve the optimization problem 
\begin{align}
\label{eq:min_variance_st_relaxed}
\underset{v: \eqref{eq:rate_constraints} \eqref{eq:demand_constraints}\eqref{eq:deadline_constraints}\eqref{eq:u_stationary_complex}}{\text{minimize}}& \;\;\;\var (P ) .
\end{align}
\re{ Notice that any policy that can be realized by $u$ in \eqref{eq:u_stationary_simple} can also be realized by $v$ in \eqref{eq:u_stationary_complex}, but a policy that can be realized by $v$ may not necessarily be realized by $u$. Thus, policy $v$ is more general than $u$, and the constraint set of \eqref{eq:min_variance_st} is contained in the constraint set of \eqref{eq:min_variance_st_relaxed}.} Consequently, the optimal value of \eqref{eq:min_variance_st_relaxed} lower-bounds that of \eqref{eq:min_variance_st}.
Therefore, given the optimal solution of \eqref{eq:min_variance_st_relaxed}, if the solution of \eqref{eq:min_variance_st_relaxed} (given in the next lemma) is also achievable by a control policy $u$ of the form \eqref{eq:u_stationary_simple}, it must be the optimal solution of \eqref{eq:min_variance_st} as well. 

\begin{lemma}
\label{thm:static_problem}
The optimal solution of \eqref{eq:min_variance_st_relaxed} is  
\begin{align}
\label{eq:exact_policy_2}
 v(   \e,   \st , \x   )   = \frac{ \e }{   \st } \mathbf{1}\{ \x > 0\},
\end{align}
and it yields the optimal value
\begin{align}
\var (P(t) ) =\Lambda \E\left[ \frac{ \e^2 }{ \st } \right] .
\end{align}
\end{lemma}

To show Lemma \ref{thm:static_problem}, we use the following property of the system: since the service rate of a job only depends on the property of that job, its impact on $\var(P)$ can be computed by integrating along the trajectory of a job over its distribution $\pinti f(\e,\st)$~\cite{baccelli}.\footnote{This is a restatement of Brumelle's formula \cite{brumelle} from queueing theory for systems with infinitely many servers with time-varying rates.} In particular, the following relation holds. 
 
\begin{lemma}
\label{prop:poisson}
The mean and variance of $P(t)$ under the policy \eqref{eq:u_stationary_complex} are given by 
\begin{align}
\E[ P ] &= \int_{(\e,\st)\in S }  \int_{0}^\st    v(\e,\st,\x) \pinti f(\e,\st) d\x d\e  d\st  \\
\var( P ) &= \int_{(\e,\st)\in S }  \int_{0}^\st    v(\e,\st,\x)^2 \pinti f(\e,\st) d\x d\e d\st   .
\end{align}
\end{lemma}

Lemma \ref{prop:poisson} can be obtained from the Campbell's theorem (see Appendix \ref{sec:lem2}). Now we are ready to prove Lemma \ref{thm:static_problem}.

\begin{proof}{Proof of Lemma \ref{thm:static_problem}.}
The demand constraints \eqref{eq:demand_constraints} and the deadline constraints \eqref{eq:deadline_constraints} leads to
\begin{equation}
\label{eq:eq:ed1}
\int_0^\st v (\e , \st , \x ) d\x = \e , \quad (\e , \st) \in S .
\end{equation}
The objective function \eqref{eq:min_variance_st_relaxed} satisfies
\begin{align}
 \label{eq:thm1_1}
\var(P) &= \int_{(\e,\st)\in S }  \int_0^\st    v(\e,\st,\x)^2 \pinti f (\e,\st) d\x d\e d\st\\
 \label{eq:thm1_11}
&=  \int_{(\e,\st)\in S }  \left\{  \int_0^\st    v(\e,\st,\x)^2  d\x  \right\} \pinti f (\e,\st) d\e d\st \\
 \label{eq:thm1_2}
&\geq \int_{(\e,\st)\in S }  \left\{ \frac{\e^2}{\st}  \right\} \pinti f (\e,\st) d\e d\st .
\end{align}
Here, equality \eqref{eq:thm1_1} is due to Lemma \ref{prop:poisson}. Inequality \eqref{eq:thm1_2} is due to \eqref{eq:eq:ed1} and the Holder's inequality, \ie for any fixed $(\e , \st)$, 
 \begin{align}
 \label{eq:thm1_10}
\left(   \int_0^\st   v(\e,\st,\x)^2  d\x  \right)^{1/2} \left(   \int_0^\st   1  d\x  \right)^{1/2} \geq \int_0^\st    v(\e,\st,\x) d\x = \e , 
\end{align}
where $v(\e,\st,\x) \geq 0$. Alternatively, it can be verified that \eqref{eq:thm1_2} can be attained with equality when $v$ equals \eqref{eq:exact_policy_2}. Therefore, \eqref{eq:exact_policy_2} is the optimal solution of \eqref{eq:min_variance_st_relaxed}.   \hfill \qed

\end{proof}

Lemma \ref{thm:static_problem} considers the optimal scheduler among policies of the form $v( \e ,   \st , \x )$, which takes a more general form than $u(\y , \x )$ with the same objective function. Interestingly, the optimal scheduler in Lemma \ref{thm:static_problem} does not use the additional freedom given by $v(  \e ,   \st , \x   )$ and can be represented by the form $u(\y , \x )$. This indicates that considering accounting for job arrival times by consider a more complex form of scheduler $v( \e ,   \st , \x )$ does not increase the system performance. Now, we can prove Theorem \ref{thm:1_stationary_problem1} using Lemma \ref{thm:static_problem}.

\begin{proof}{Proof of Theorem \ref{thm:1_stationary_problem1}.}
Recall that the optimal solution of \eqref{eq:min_variance_st_relaxed} is the policy \eqref{eq:exact_policy_2}. Under the policy \eqref{eq:exact_policy_2}, the ratio between its remaining demand $\y(t)$ and remaining time $\x(t)$ are constant for any $t \in [\at,\at + \st]$. Therefore, \eqref{eq:exact_policy_2} can be realized using policies of the form \eqref{eq:u_stationary_simple}. Because the optimal value of \eqref{eq:min_variance_st_relaxed} is a lower bound on that of \eqref{eq:min_variance_st}, the optimal solution of \eqref{eq:min_variance_st_relaxed}---Exact Scheduling---is also the optimal solution of \eqref{eq:min_variance_st}.  \hfill \qed
\end{proof}

In fact, the same property also holds for the optimization problems \eqref{eq:min_variance+softdemand_st}, \eqref{eq:min_variance+softdeadline_st}, and \eqref{eq:min_variance+softdemand+softdeadline_st}. This property allows us to derive their closed-form solutions.

\subsection{Soft demand requirements}
\label{sec:soft_demand}

The previous section shows the optimal algorithm under strict service constraints. In this section, we relax the assumption of strict service constraints and characterize the optimal algorithm under soft demand constraints. Specifically, we consider the setting of \eqref{eq:min_variance+softdemand_st}, where the system does not need to satisfy all demands but needs to pay penalty $\C$ for each unit of unsatisfied demands. The resulting optimal algorithm is a variation of Exact Scheduling with an additional rate upper-bound: 
\begin{align}
\label{eq:exact_softdemand_st}	
& u( \y, \x )  = 
\begin{dcases}
\frac{ \y }{ \x } , & \text{if } \frac{ \y }{ \x } \leq \frac{ \C }{ 2 }   \text{ and }   \x > 0 ,\\
\frac{ \C }{ 2 }  , & \text{if } \frac{ \y }{ \x } > \frac{ \C }{ 2 }   \text{ and }   \x > 0,\\
0 , & \text{otherwise.}
\end{dcases}
\end{align}
This policy essentially imposes a threshold (an upper bound) of $\C/2$ on the service rate: jobs whose ratio $\e/\st$ is above threshold $\C/2$ are served at a constant rate $\C/2$ \textit{until its deadline}; jobs whose ratio $\e/\st$ is below this threshold are served according to Exact Scheduling. In other words, a job $k$ receives its full service demand only if $\e_k /\st_k   \leq \C /2$. 

\begin{theorem}
\label{thm:stationary_softdemand}
The policy \eqref{eq:exact_softdemand_st} is the optimal solution of \eqref{eq:min_variance+softdemand_st} and achieves the optimal value 
\begin{align}
\label{eq:min_variance+softdemand_st_cost}
  \var(P) +  \E[ \C \U ]  
= \E\left[   \frac{ \e^2 }{ \st }  \mathbf{1} \left\{  \frac{ \e }{ \st } \leq \frac{ \C }{ 2 }   \right\}  + \C \left( \e - \frac{ \C \st  }{4} \right) \mathbf{1} \left\{  \frac{ \e }{ \st } > \frac{ \C }{ 2 }   \right\} \right]   \pinti  .
\end{align}
\end{theorem}

\re{Theorem \ref{thm:stationary_softdemand} shows the performance improvement gained by relaxing the demand requirements. 
Recall from Theorem \ref{thm:1_stationary_problem1} that the average cost per unit job arrival is $ \E\left[ \e^2 /  \st   \right] $ if all demands must be satisfied. 
If the system does not need to satisfy all demand requests, then the average cost for jobs satisfying $ \e / \st  >  \C / 2$ can be reduced from $\E\left[ \e^2 /  \st   \right] $ to $\E\left[ \C \left( \e - ( \C \st / 4 ) \right) \right]$. And the portion of such jobs are given by $\E [ \mathbf{1} \left\{  \e / \st  >  \C / 2    \right\}  ] $.} The optimal policy \eqref{eq:exact_softdemand_st} is also simple and easy to implement. Despite the convenience and wide use of simple thresholding policies in practice, to the best of our knowledge, its optimality results and the optimal choice of thresholding values on rate have not been proposed in the existing literature.

\subsection{Soft deadline requirements}
\label{sec:soft_deadline}

The previous section shows the optimal algorithm under soft demand requirements. In this section, we relax the deadline requirements instead and characterize the optimal distributed algorithm. Specifically, we consider the setting of  \eqref{eq:min_variance+softdeadline_st}, where the system needs to pay penalty $\D$ for each unit of deadline extensions. Although scheduling problems with deadline extension (tardiness) often leads to an NP-hard problem \cite{du1990minimizing,baker1990sequencing}, by taking a probabilistic approach aimed at finding the best remaining job distribution, we obtain the optimal algorithm in closed-form below. The resulting optimal algorithm is a variation of Exact Scheduling with deadline extensions:
\begin{align}
\label{eq:exact_softdeadline_st}
& u( \y , \x )  = \begin{dcases}
 \frac{ \y }{ \x } &  \text{ if }  \frac{\y}{\x} \leq \sqrt{\D}  \text{ and } \x > 0\\
	\sqrt{\D}\;\mathbf{1} \{ \y > 0 \}    &  \text{ otherwise}. 
 \end{dcases} 
\end{align}
Similarly to \eqref{eq:exact_softdemand_st}, this policy essentially sets a threshold (an upper bond) $\sqrt{ \D }$ on the service rate: jobs with the ratio above  threshold $\sqrt{ \D }$ is served according to Equal Service of rate $\sqrt{\D}$ \textit{until it finishes}, jobs with the ratio below threshold $\sqrt{ \D }$ is served according to Exact Scheduling. In other words, the deadline of job $k$ is extended only if $\e_k/  \st_k > \sqrt{ \D } \st_k$.

\begin{theorem}
\label{thm:stationary_softdeadline}
The policy \eqref{eq:exact_softdeadline_st} is the optimal solution of \eqref{eq:min_variance+softdeadline_st} and achieves the optimal value 
\begin{align}
\label{eq:min_variance+softdeadline_st_cost}
  \var(P)+  \E[ \D \Q ]  
= \E\left[  \frac{ \e^2 }{ \st }  \mathbf{1} \left\{ \frac{\e}{\st} \leq \sqrt{\D}     \right\}  + \left( 2  \sqrt{\D}  \e - \D \st \right) \mathbf{1}  \left\{ \frac{\e}{\st} > \sqrt{\D} \right\}  \right]    \pinti  .
\end{align}
\end{theorem}
Theorem \ref{thm:stationary_softdeadline} shows the performance improvement by relaxing the deadline requirements. Theorem \ref{thm:1_stationary_problem1} states that, if all deadlines must be satisfied, then the average cost per unit job arrival is $ \E\left[ \e^2 /  \st   \right] $. By allowing deadline extensions, the average cost of jobs satisfying $\e/\st > \sqrt{\D}$ can be reduced from $ \E\left[ \e^2 /  \st   \right] $ to $\E \left[ \left( 2  \sqrt{\D}  \e - \D \st \right) \right]$. And the portion of such jobs are given by $\E \left[ \mathbf{1} \{ \e/\st > \sqrt{\D}  \} \right] $. Moreover, service capacity variance and penalties for deadline extension is optimally balanced when jobs whose deadline extension penalties are smaller than $\e/\st$, \ie $\e/\st > \sqrt{\D}$, are served with deadline extension.

\subsection{Soft demand and deadline requirements}
\label{sec:soft_demand_deadline}

The previous sections show the optimal algorithms under soft demand requirements and soft deadline requirements. In this section, we relax both demand and deadline requirements and characterize the optimal distributed algorithm. Specifically, we consider the setting of \eqref{eq:min_variance+softdemand+softdeadline_st} where the system needs to pay penalty $\C$ for each unit of unsatisfied demands and penalty $\D$ for each unit of deadline extensions. This setting recovers all previous settings as special cases.\footnote{For sufficiently large $\C$, this setting recovers the case of strict demand requirements. For sufficiently large $\D$, this setting recovers the case of strict deadline requirements. For sufficiently large $\C$ and $\D$, this setting recovers the case of strict demand and deadline requirements.} 

Recall from previous sections that, under soft demand requirements, the optimal policy uses a constant service rate and reject partial demand requests only if $\e/\st > \C/2$. Meanwhile, under soft deadline requirements, the optimal policy uses a constant service rate and extends the deadline only if $\e/\st > \sqrt{\D}$. These two special cases motivate us to combine the policies~\eqref{eq:exact_scheduling_st}, \eqref{eq:exact_softdemand_st}, and \eqref{eq:exact_softdeadline_st} as follows:
\begin{align}
\label{eq:exact_softdemand+deadline_st}	
& u( \y , \x )  = \begin{dcases}
\frac{\y }{\x}  	& \text{if } \x > 0 \text{ and } \frac{\y }{\x} \leq \min\left\{  \frac{\C }{2} ,  \sqrt{\D} \right\}   \\
\frac{\C }{2} 	& \text{if } \x > 0 \text{ and } \frac{\y }{\x} > \frac{\C}{2} \text{ and } \frac{\C}{2} \leq \sqrt{\D}  \\ 
\sqrt{\D} \; \mathbf{1}\{\y > 0\}&  \text{otherwise} 
\end{dcases},
\end{align}
The policy uses three strategies depending on different regimes of job states and penalties: high penalties regime, low demand penalty regime, and low deadline penalty regime. These regimes are illustrated in Figure~\ref{fig:optimal policy_softdemand&deadline} as the white, light gray, and dark gray regions, respectively. 
\begin{itemize}
\item \textit{High penalties regime.} When $\min( \C / 2,\sqrt{\D} )  > \e / \st$, it is less costly to satisfy the service requirements than paying penalties for unsatisfied demands or deadlines. So, the best strategy is to satisfy both demands and deadlines optimally using Exact Scheduling \eqref{eq:exact_scheduling_st}.  
\item \textit{Low demand penalty regime.} When $ \C/2 \leq \sqrt{\D}$, the penalties for unsatisfied demands is smaller than that of deadline extensions, so the best strategy is to meet all deadlines optimally with potentially unsatisfied demands using the policy \eqref{eq:exact_softdemand_st}.
\item \textit{Low deadline penalty regime.} When $ \C/2> \sqrt{\D}$, the penalties for deadline extension is smaller than that of unsatisfied demands, so the best strategy is to satisfy demands optimally with potential deadline extensions using the policy \eqref{eq:exact_softdeadline_st}. 
\end{itemize}
From above, the policy \eqref{eq:exact_softdemand+deadline_st} generalizes the optimal algorithms in Section \ref{sec:hard}-\ref{sec:soft_deadline}, and we term it \textit{Generalized Exact Scheduling}.  The following theorem states its optimality condition.  

\begin{theorem}
\label{thm:stationary_softdemand+deadline}
The policy \eqref{eq:exact_softdemand+deadline_st}	 is the optimal solution of \eqref{eq:min_variance+softdemand+softdeadline_st} and achieves the optimal value 
\begin{align}
\label{eq:min_variance+softdemand&deadline_st_cost}
  &\var(P)+  \E[ \C \U ] + \E[ \D \Q ]  =\\
  &\E\left[  \frac{ \e^2 }{ \st }  \mathbf{1} \left\{ \frac{\e}{\st} \leq \min \left \{ \frac{\C }{2} ,  \sqrt{\D} \right\}    \right\}  
+ \C  \left( \e   - \frac{ \C \st }{ 4 } \right) \mathbf{1}  \left\{ \frac{\e}{\st} > \frac{\C }{2} \geq \sqrt{\D}  \right\} 
+ \left( 2  \sqrt{\D}  \e - \D \st \right) \mathbf{1}  \left\{ \frac{\e}{\st} > \sqrt{\D}> \frac{\C }{2} \right\} 
 \right]    \pinti  .
  \nonumber 
\end{align}
\end{theorem}

\definecolor{green}{gray}{0.8}
\definecolor{blue}{gray}{0.9}
\definecolor{yellow}{gray}{0.6}
\definecolor{violet}{gray}{0.5}
\definecolor{orange}{gray}{0.4}

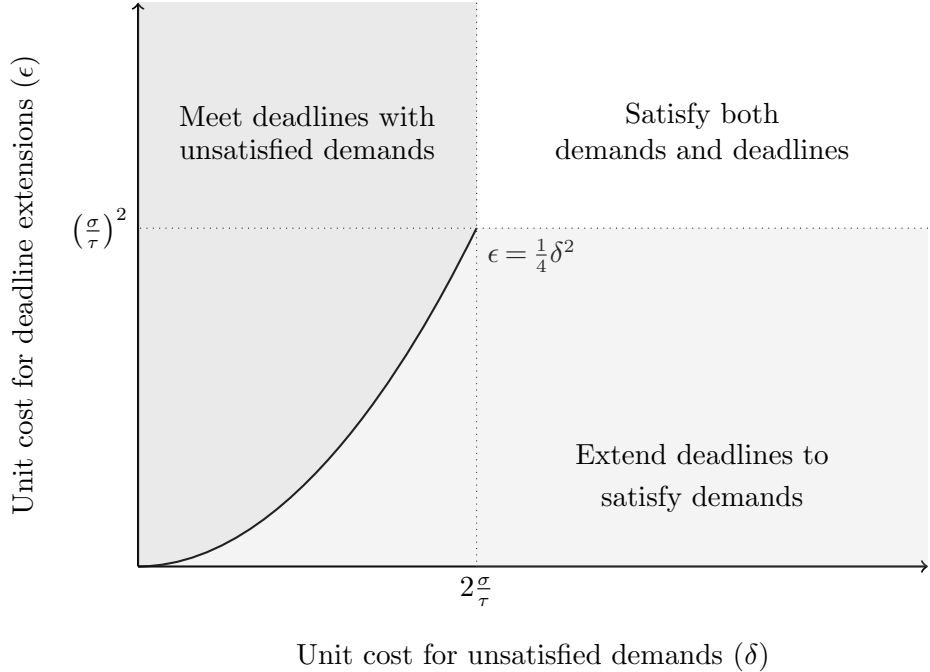
\begin{figure}
	\centering
	
	\begin{tikzpicture}[scale = 1.5]
\draw [thick,->] (0,0) -- (7,0);
\draw [thick,->] (0,0) -- (0,5);
\node [align=center] at (3.5,-.8) {Unit cost for unsatisfied demands ($\C$)};
\node [align=center,rotate=90] at (-1,2.5) {Unit cost for deadline extensions ($\D$)};
\draw [black, thick] (0,0) parabola bend (0,0) (3,3) node[black, below right] {$\D=\frac{1}{4}\C^2$};
\draw [dotted] (3,5)--(3,0) node[below] {$2\frac{\e}{\st}$};
\draw [dotted] (7,3)--(0,3) node[left] {$\left(\frac{\e}{\st}\right)^2$};
\path [fill=green, opacity=0,fill opacity = 0.2] (3,0) rectangle (7,3);
\path [fill=yellow, opacity=0,fill opacity = 0.2] (0,3) rectangle (3,5);
\draw [draw=black, fill=green, opacity=0,fill opacity = 0.2] (0,0) parabola bend (0,0) (3,3) -- (3,0) -- cycle;
\draw [draw=black, fill=yellow, opacity=0,fill opacity = 0.2] (0,0) parabola bend (0,0) (3,3) -- (0,3) -- cycle;
\node at (5,1) {Extend deadlines to};
\node at (5,0.6) {satisfy demands};
\node at (5,4) {Satisfy both};
\node at (5,3.7) {demands and deadlines};
\node at (1.5,4) {Meet deadlines with };
\node at (1.5,3.7) {unsatisfied demands};
\end{tikzpicture}
	\caption{The decision space of the optimal policy for \eqref{eq:min_variance+softdemand+softdeadline_st}. For job profiles with a service demand $\e$, a sojourn time $\st$, and costs $(\C,\D)$, the optimal policy performs either one of the following using constant service rates: satisfy both demands and deadlines (white region), meet deadlines with unsatisfied demand (dark gray region), or satisfy the demand by extending the deadline (light gray region). }
	\label{fig:optimal policy_softdemand&deadline}
\end{figure}

Theorem \ref{thm:stationary_softdemand+deadline} shows when one should extend the deadline to satisfy the demand or let the job depart at its deadline with unsatisfied demands. Moreover, Generalized Exact Scheduling is also optimal for a more general problem \eqref{eq:min_variance+softdemand+softdeadline_st_marked}, when the unit costs for unsatisfied demands and deadline extensions are allowed to be heterogeneous.

\begin{corollary}
\label{lem:stationary_softdemand+deadline}
The optimal solution of \eqref{eq:min_variance+softdemand+softdeadline_st_marked} is 
\begin{align}
\label{eq:exact_softdemand+deadline_st_marked}    
& \uind( \y , \x  ,  \C ,\D)  = \begin{dcases}
\frac{\y }{\x}      & \text{if } \x > 0 \text{ and } \frac{\y }{\x} \leq \min\left\{  \frac{\C }{2} ,  \sqrt{\D} \right\}   \\
\frac{\C }{2}     & \text{if } \x > 0 \text{ and } \frac{\y }{\x} > \frac{\C}{2} \text{ and } \frac{\C}{2} \leq \sqrt{\D}  \\ 
\sqrt{\D} \; \mathbf{1}\{\y > 0\}&  \text{otherwise} 
\end{dcases}.
\end{align}
\end{corollary}
Corollary \ref{lem:stationary_softdemand+deadline} is an immediate consequence of Theorem \ref{thm:stationary_softdemand+deadline}. 


\section{Performance degradation inherent to availability in job information}
\label{sec:CentralizedvsDistributed}

{\color{black}

Given our focus on distributed algorithms, we should investigate how much performance degrades in comparison to centralized algorithms. This investigation also includes a practically important question: if there is any middle ground between centralized and distributed algorithms in which scalability and close-to-centralized performance can be achieved simultaneously. Moreover, another practically relevant question is: what can be done if the information on the service requirement (demands and/or deadlines) are missing, and how much does the performance degrade due to the missing information?  In this section, we answer these questions using both experiments and theory. Specifically, we compare the performance of optimal offline algorithms, centralized online optimization, particularly-centralized algorithms, online distributed algorithms using actual electric vehicle charging data from the Caltech Testbed \cite{testbed} and synthetic data drawn from varying arrival distribution (Section \ref{sec:emp}). Then, we derive bounds on the cost of the optimal centralized policy and use these bounds to characterize the performance degradation of the optimal distributed algorithm (Section \ref{sec:lower-bound}). When it comes to deriving performance bounds, there is no standard technique in queueing to derive the performance limits of centralized policies in this setting. Instead, we borrow tools from optimal control and provide an upper bound on the performance. Finally, we present a proof of the upper bound, which is potentially useful for providing performance degradation bounds for policies in other settings as well (Section \ref{sec:lower-bound-proof}).

\subsection{Empirical evaluation}
\label{sec:emp}

To evaluate the performance of the proposed algorithm, we compare its performance and existing scheduling algorithms in an electric vehicle charging testbed and using synthetic data of varying arrival distribution. 

\subsubsection{System and data}
\label{sec:data-explain}

We employed a trace-driven simulation on real data from an electric vehicle charging testbed~\cite{testbed}, and a synthetic data set randomly drawn from a set of arrival distribution with varying parameters. The real data contain the arrival profiles of 92 days in 2016. A charging instance contains service requests from each electric vehicle arriving in one day. A service request of a job is defined by its arrival time, energy demand, and sojourn time. The statistics of the service requests are summarized in Table \ref{table:instance_stat}. The synthetic data are generated from the following set of arrival distribution. The time is discretized into the sampled time $t_1 = 0, t_2, \cdots , t_n$. Given a vector $b(i)$ of i.i.d. Bernoulli random variables with mean $p_B (\ll 1)$, the sampled time $t_i$ is considered to have one arrival if $b(i) = 1$; and zero arrival if $b(i) = 0$. For each arrival, its service demand $\e$ is uniformly distributed in $[\underline \e, \bar \e]$. Its sojourn time is generated from two different cases, defined by two types of arrival distributions (I and II). In distribution I, the sojourn time is given by $\e + \ell$ (additive), where $\ell$ is i.i.d. exponentially distributed with mean $\bar \ell$. In distribution II, the sojourn time is given by $\gamma \e$ (multiplicative), where $\gamma$ is uniformly chosen from the interval $[1, \bar \gamma]$. In both distributions I and II, all jobs are feasible, \ie $\e_k \leq \st_k$ for all $k \in \hV$. The parameters for the arrival distributions are chosen to be $p_B \in [0.1, 0.3], \underline \e= 10, \bar \e = 20, \bar \ell \in [10, 50], \bar \gamma = 3$.

\begin{table}
\begin{subtable}{\textwidth}
\caption{Job profiles}
\begin{tabularx}{\textwidth}{  L{0.2} | C{0.6} | C{0.6} | C{0.6} | }

				& Demand  & Sojourn time     	\\
& $ \e_k $ (kW $\times$ minutes)	& $ \st_k $ (minutes)	 	\\
\hline 
Mean 			&  $ 5.1 \cdot 10^2$  	& $4.5\cdot 10^3$ 	 		\\
Variance 			&  $ 3.1 \cdot 10^5$	& $7.7\cdot 10^6$ 				\\
\end{tabularx}
\end{subtable}

\begin{subtable}{\textwidth}

\caption{Instance profiles }
\begin{tabularx}{\textwidth}{  L{0.2} | C{0.6} | C{0.6} | C{0.6} | }
				&Total demand   & Time horizon  &Number of jobs 		\\
				&$\sum_{k \in \hV} \e_k $ (kW $\times$ minutes) & $T$ (minutes)  & $|\hV|$ 		\\
\hline
Mean 			 			&  	$8.4 \cdot 10^3$    & 	$6.5\cdot 10^2$ & $14.2$ \\
Variance 					& 	$2.0 \cdot 10^7$ &   $1.4 \cdot 10^4$ & $48.2$ 		\\
\end{tabularx}
\end{subtable}

\caption{Statistics of the electric vehicle charging instances at the testbed~\cite{testbed}.}
\label{table:instance_stat}
\end{table}

\subsubsection{Algorithms tested}
\label{sec:testalgorithm}

We compared a few standard schedulers and Generalized Exact Scheduling. The schedulers range from optimal offline policy and online fully-centralized policies to online partially-centralized policies and online distributed policies. The detail of these schedulers is defined below.

\vspace{3mm}
\noindent \emph{Offline optimal algorithms (centralized).}

To understand the best possible performance, we compare with the  optimal offline algorithms. The optimal offline algorithm tells the best performance achievable given the centralized information of \textit{all} jobs arriving in the future. The offline policy takes the form 
\begin{align}
\label{eq:offline}
r_k(t) = \off ( k, t, \{ A_t \}_{t \in [0 , T] } )  , \quad \forall k \in \hV, 
\end{align}
where the service rate at time $t$ is allowed to use the information of all future arrivals. The optimal offline policy can be computed from the following optimization problem: 
\begin{align}
\label{eq:offline-optimization}
\underset{  \eqref{eq:rate_constraints}\eqref{eq:demand_constraints}\eqref{eq:deadline_constraints}}{\text{minimize}}  \;\;\;\;\;\;\;\; \frac{1}{T} \int_{0}^{T} ( P(t) - \bar P )^2 dt, 
\end{align}
where the optimization variable is $\off$ in \eqref{eq:offline}, and $\bar P = (1/T) \int_{0}^T P(t) dt $ be the time average of $P(t)$. We denote the solution of problem \eqref{eq:offline-optimization} as Optimal Offline.

This assumption on Optimal Offline is often too strong in practice: offline algorithms cannot be used when future information is hard to obtain. However, as any distributed or online algorithms can perform no better than the optimal offline algorithm, it is still useful to have its performance as a baseline. 
Specifically, we quantify the relative cost of any online algorithm and Optimal Offline using the ratio of between the cost of the algorithm and that of Offline Optimal (with a slight abuse of notation,\footnote{Competitive-ratio typically refers to the worst-case ratio among all possible instances, but here, we use empirical competitive-ratio to refer to the empirically realized ratio in one instance.\label{fn:competitive-ratio}} we denote this ratio as the \textit{empirical competitive-ratio}). This quantify is used in Figure \ref{fig:ES_strict}, \ref{fig:EShist}, \ref{fig:ES_strict_app} to evaluate the performance of different algorithms under varying arrival distribution. 

\vspace{3mm}
\noindent \emph{Fully-centralized online algorithms.}

We consider centralized (online) scheduling policies of the form 
\begin{equation}
\label{eq:u_cent1}
\rate_k(t) = \uc( k, t , A_t ) , \quad \forall k \in \hV,
\end{equation}
where $A_t = \{ (\at_k , \e_k ,\st_k ,   \y_k(t) , \x_k(t) )  : \at_k \leq t \} $ is the set that contains the information of jobs arriving no later than $t$, and $\uc(k,t. \cdot)$ is a deterministic mapping from $A_t$ to a service rate $\rate_k(t)$. When deciding the service rate at each time, the policy can use the information of jobs that arrived before that time. We list below a few centralized online policies tested in this paper. 
\begin{itemize}
\item Online Optimization MPC. This policy performs Model Predictive Control (MPC) on the objective function \eqref{eq:offline-optimization}. Specifically, at each time $t$, this policy solves the optimization problem \eqref{eq:offline-optimization} with optimization variable \eqref{eq:u_cent1} to find the service rates from $t$ to the $T$, but the scheduler put only $r(t)$ into action. It recompute problem \eqref{eq:offline-optimization} to obtain the service rates from $t+1$ to the $T$ and put only $r(t+1)$ into action. The computed service rates are optimal at $t$ if no jobs arrive in the future but have no global optimality guarantees otherwise. 
\item Earliest Deadline First (EDF). This policy allocates a fixed capacity $p_{\text{EDF}}$ to jobs in ascending order of their deadlines. Under soft demand constraints, jobs are served until their deadline. Under soft deadline constraints, jobs are served until their completion. 
\item Least Laxity First (LLF). Recall from \eqref{fn:laxity} that $\ell_k(t) = \y_k(t) - \x_k(t)$ is the laxity of job $k$ at time $t$. This policy allocates a fixed capacity $p_{\text{LLF}}$ to jobs in ascending order of their laxity. Under soft demand constraints, jobs are served until their deadlines. Under soft deadline constraints, jobs are served until their completion.
\item Fair Sharing (FS). This policy equally distributes a fixed capacity among jobs. Under soft demand constraints, jobs are served until their deadlines according to $r_k(t) = \min\{  p_{\text{FS}} /n(t) , 1 \} \mathbf{1}\{ \x_k(t) > 0 \}$, where $n(t)$ the is number of unfinished jobs at time $t$. Under soft deadline constraints, jobs are served to their completion according to $r_k(t) = \min\{  p_{\text{FS}}' /n(t) , 1 \}  \mathbf{1}\{ \y_k(t) > 0  \text{ and } \x_k(t) > 0  \}$. Intuitively, $p_{\text{FS}}$ or $p_{\text{FS}}'$ is the service capacity the system is willing to provide in their respective settings, and this capacity is shared among all unfinished jobs. Here $p_{\text{FS}}$ and $p_{\text{FS}}'$ are chosen to be the optimal offline values.\footnote{\label{ft:alg_para}Since the offline optimal parameters are unknown in practice, the test results obtained here are optimistic.}
\end{itemize}

\vspace{3mm}
\noindent \emph{Distributed online algorithms.}

Recall from Section \ref{sec:stationary_model} that a distributed online policy takes the form \eqref{eq:u_stationary_simple}. Below lists the centralized online policy tested in our experiments. 
\begin{itemize}
\item Generalized Exact Scheduling \eqref{eq:exact_softdemand+deadline_st}. This policy recovers Exact Scheduling \eqref{eq:exact_scheduling_st} under strict service requirements, the policy \eqref{eq:exact_softdemand_st} under soft demands, and the policy \eqref{eq:exact_softdeadline_st} under soft deadlines.
\item Immediate Scheduling. This policy schedule jobs at its maximum rate upon arrival, \ie $r_k(t) = \mathbf{1}\{\y_k(t) > 0\}$ for any $k \in \hV$. 
\end{itemize}

\vspace{3mm}
\noindent \emph{Partially-centralized algorithms.}

The centralized algorithms listed above require the scheduler to access all service requirement information (demands and deadlines) for all jobs present in the system. In contrast, the above distributed algorithms use no centralized information (\ie the service rate of each job is only determined using its own information). Beyond the two extreme cases of totally centralized versus totally distributed, there is the middle ground of \textit{partially centralized} algorithms where service rates are determined mostly using the local information of each job but are allowed to access some global state variables. The design choices for partially centralized algorithms are vast, yet their potential is under-explored in the existing literature. Although a comprehensive study of such design space is beyond the scope of this paper, we have explored a few such options empirically to facilitate future discussion. As the major goal of our paper is to the design of scalable algorithms, we focus on near-distributed policies that only use a limited number of global variables. Such policies take the form 
\begin{align}
\label{eq:pc_stationary_simple}
&r_k(t) = pc(  \y_k(t)  , \x_k (t) , z(t)  )  \geq 0 , &k  \in \hV
\end{align}
where $z(t)$ is a low-dimensional vector that is shared among the local schedulers for each job. The policy \eqref{eq:pc_stationary_simple} requires much less resource in computation and communication compared with the centralized algorithms listed above. For example, Online Optimization MPC solves at each time step a quadratic program; EDF and LLF require jobs to be sorted. In contrast, \eqref{eq:pc_stationary_simple} only evaluates a closed-form function, and  $z(t)$ through to local schedulers with minimum communication resources. 
 
We tested many algorithms, where $z(t)$ contains the service capacity, total remaining demands, total remaining time, number of jobs, and combinations of these quantities. Some has better performance than others, and we present below two of such algorithms. 
\begin{itemize}
\item Exact Scheduling PC. This policy performs Exact Scheduling plus minor adjustment using partially centralized (PC) variable $P(t)$. It operates as Exact Scheduling when the service capacity is higher than average but adds additional boosts in service rate otherwise, \ie
\begin{align}
\label{eq:pc_ES}
pc(  \y_k(t)  , \x_k (t) , z(t)  ) =  \begin{dcases}
\mu \frac{ \y_k(t) }{ \x_k (t)  }  & P(t-dt) < \bar P \\ 
\frac{  \y_k(t) }{ \x_k (t) }  & \text{otherwise}
\end{dcases} 
\end{align}
where $dt$ capture the latency in communicating $P(t)$, and $\mu \geq 1$ is the factor that boosts service rate at low service capacity regime. Empirically, the values for $\mu$ that work well for the scheduling instances tested range from $1.2$ to $1.6$. 
\item Equal Service. This policy offers a homogeneous service rate to all unfinished jobs. Under strict service requirements, it serves jobs with positive laxity at a homogeneous service rate $c_{\text{ES}} $ and jobs with zero laxity at its maximum rate $1$. Specifically, the laxity of a job at time $t$ is defined as the remaining time before deadline if the job is to be served at its maximum rate from time $t$ until job completion. In this context, it can be computed as the remaining demand minus the remaining time: 
\begin{align}
\label{fn:laxity} 
   \ell_k(t) := \y_k(t) - \x_k(t) . 
\end{align}
The service rate under this policy is given by $r_k(t) = c_{\text{ES}} \mathbf{1}\{ \ell_k(t) > 0 \text{ and } \y_k(t) > 0\} + \mathbf{1}\{ \ell_k(t) \leq 0 \text{ and } \y_k(t) >0 \}$. Under soft demand constraints, it serves jobs at a homogeneous service rate $c_{\text{ES}}'$ before its deadline, and the service rate is given by $r_k(t) = c_{\text{ES}} ' \mathbf{1}\{ \y_k(t) > 0 \text{ and } \x_k(t) \geq 0 \}$. Note that this policy may not fulfill the demand of all jobs but does satisfy all deadlines. 
Under soft deadline constraints, it serves jobs at a homogeneous service rate until its completion, and the service rate is given by $r_k(t) = c_{\text{ES}} '' \mathbf{1}\{\y_k(t) > 0 \}$. Note that this policy satisfies all demands but may extend the deadlines of some jobs. Insights from Section \ref{sec:stationary_systems} suggest that Equal Service may perform well when the values of $c_{\text{ES}} , c_{\text{ES}}', c_{\text{ES}} ''$ are closed to $\E[ v^*(   \e_k ,   \st_k , \x_k (t)   ) ]$, where $v^*$ is the optimal solutions for problem \eqref{eq:min_variance_st_relaxed}, \eqref{eq:cost2}, and \eqref{eq:min_variance+softdeadline_st1}.\footnote{Problem \eqref{eq:cost2} and \eqref{eq:min_variance+softdeadline_st1} are defined in the Appendix.} Indeed, we observed this behavior empirically. Moreover, we noticed that Equal Service has robust behavior and performance to small perturbation in its rate away from these values as well. 
\end{itemize}

The algorithms listed above are compared in the settings of strict service constraints in Figure~\ref{fig:ES_strict}, soft demand constraints in Figure~\ref{fig:ESsoftdemand}, soft deadline constraints in Figure~\ref{fig:ESsoftdeadline}. 

\subsubsection{Fully-centralized vs partially-centralized vs distributed algorithms}

In the setting of strict service constraints, we can only use Exact Scheduling, Immediate Scheduling, and Equal Service, Online Optimization (MPC), and Offline Optimal because other algorithms cannot guarantee to satisfy the service requirement strictly.  For the instances in the testbed, we observed a significant performance degradation from Offline Optimal to Online algorithms: online algorithms experience 1.5 times more cost due to the lack of future information. However, among online algorithms, Exact Scheduling only degrades from Online Optimization (MPC) by an average of $20\%$ in cost. With a minor adjustment in Exact Scheduling using a globally shared variable $P(t)$, Exact Scheduling PC (partially-centralized) in \eqref{eq:pc_ES} can reduce the cost for about $10\%$. This cost reduction leads to less than $10 \%$ of difference in cost between Exact Scheduling PC and the fully-centralized Online Optimization (MPC), which requires much more computational and communication resources (Figure~\ref{fig:ES_strict-1}). For synthetic instances, the performance degradation from Offline Optimal to online algorithms is much less than in the testbed, which may be attributable to the fact that the synthetic data's arrival distribution is closer to our assumptions in the arrival distribution of the optimization problems. On the other hand, the performance degradation from fully-centralized to partially-centralized to fully-distributed remains similar. Exact Scheduling PC (partially centralized) only degrades from fully-centralized Online Optimization (MPC) by about $10\%$ on average, and Exact Scheduling (fully-distributed) only degrades from Exact Scheduling PC by another $10 \%$ on average. This relation holds beyond the specific arrival distribution considered in Figure~\ref{fig:ES_strict-2}-\ref{fig:ES_strict-3} (see Figure \ref{fig:ES_strict_app} in the Appendix for the performance comparison in varying arrival distributions).

The relative performance/cost of Exact Scheduling and others depends on the characteristics of the charging instance. To further investigate this dependency, we grouped instances according to its empirical competitive-ratio of performing Exact Scheduling and computed the average arrival rate and demand and sojourn time ratio for each group. In general, empirical competitive-ratio (the comparative performance) improves as the number of arrivals decrease (Figure~\ref{fig:hist2}) and also as average demand to sojourn time ratio increases in size (Figure~\ref{fig:hist3}). These points can also be seen in Figure~\ref{fig:example-best-worst}, which compares the service rates for the instance in which the Exact Scheduling performed equally well with Offline Optimal (Figure~\ref{fig:ESbest}) and those for the instance in which Exact Scheduling performed much worse (Figure~\ref{fig:ESworst}). Intuitively, sparser arrivals would require less coordination between scheduling different jobs, which in turn reduces the advantages of being able to use centralized information. Moreover, when the deadline is tight, the service requirement does not allow much flexibility in varying the service rate over time, and the offline (centralized) algorithm may not be able to use the future arrival information to its full advantage.

When the demands do not need to be strictly satisfied, the scheduler can exploit this flexibility to reduce the overall cost by balancing the service capacity variance and the penalties for unsatisfied demands. In this setting, the behavior of a few distributed algorithms (Generalized Exact Scheduling, Immediate Scheduling, Equal Service) and centralized algorithms (Earliest Deadline First, Least Laxity First, Fair Sharing) is compared in Figure~\ref{fig:ESsoftdemand}.\footnote{\label{ft:cost-name} In Figure \ref{fig:ESsoftdemand}-\ref{fig:ESsoftdeadline}, we use ``normalized cost'' instead of empirical competitive-ratio. This is because, unlike the case of empirical competitive-ratio, the cost under soft service requirement is compared here with that of offline optimal for hard service requirement.} As the unit penalty for unmet demands $\C$ grows, all algorithms inevitably suffer from increased costs as well. However, the cost of Generalized Exact Scheduling plateau out at relatively small $\C$, which results in a lower cost compared with other centralized and distributed algorithms (Figure~\ref{fig:ESsoftdemand}a). This quick plateau is achieved by a highly adaptive reduction in the total amount of unsatisfied demands. For small unit penalty $\C$, Generalized Exact Scheduling is among the algorithms with the largest amount of unmet demand to exploit the flexibility in being able to miss some demands. For large unit penalty $\C$, it has the smallest amount of unmet demand in order to minimize its high penalty associated with not meeting demands (Figure~\ref{fig:ESsoftdemand}b). This dynamic and optimal adjustment is obtained as the solution of the optimization problem \eqref{eq:min_variance+softdemand_st}, which balances the service capacity variance and unmet demand. Thus, its design process is systematic and does not require tedious manual adjustments.  

When job deadlines do not need to be strictly enforced, the scheduler can exploit this flexibility to reduce the overall cost by balancing the service capacity variance and the penalties for deadline extensions. In this setting, the behavior of a few distributed algorithms and centralized algorithms is compared in Figure~\ref{fig:ESsoftdeadline}.\textsuperscript{\ref{ft:cost-name}} Similar to the setting of soft demand, Generalized Exact Scheduling achieves a lower cost than other distributed algorithms (Figure~\ref{fig:ESsoftdeadline}a). It also has a comparable performance with the centralized algorithms when the unit penalty for unsatisfied deadline $\D$ is large. Such performance is achieved by drastically reducing the total amount of unsatisfied deadline as $\D$ increases to avoid the high penalty associated with deadline extension (Figure~\ref{fig:ESsoftdeadline}b). This adjustment is obtained as the solution of the optimization problem \eqref{eq:min_variance+softdeadline_st}, which systematically balances the service capacity variance and deadline extension. 

Generalized Exact Scheduling is the outcome of systematically optimizing service capacity to find the right balance between service capacity variance and the penalties for unsatisfied demands or deadlines. Its excellent performance compared to other distributed algorithms is not surprising because those algorithms are not optimized for dynamic service capacity nor designed to systematically trade-off service capacity variance and unmet demands or deadlines. Moreover, the results in the testbed also suggest that Generalized Exact Scheduling can perform better than other distributed algorithms (such as in the charging of electric vehicles) beyond the case of Poisson arrivals under which Exact Scheduling is optimal.

\subsubsection{Dealing with demand and deadline uncertainties} 

Most algorithms discussed in this section require knowledge about the demands and deadlines of all jobs. This condition is valid for certain applications~\cite{testbed,gan2013optimal,yao1995scheduling,bansal2007speed}. For example, in the electric vehicle charging testbed~\cite{testbed}, the system receives user input about the energy demand and departure time of each vehicle. On the other hand, there are other situations where the information on service requirements (demands and/or deadlines) can be missing for all or a subset of jobs. 
Recall from Section \ref{sec:stationary_systems} that the optimal distributed policy is to have a flat and low service rate in the service rate trajectory. This intuition motivates us to consider a mixture of Exact Scheduling and Equal Service: serve according to Exact Scheduling if the demands and deadlines are known; otherwise, make the best guess about a good service rate and apply that rate to all jobs with unknown demand and/or deadlines. This policy reduces to Exact Scheduling if the service requirement for all jobs are known, whereas it reduces to Equal Service when the service requirement for none of the jobs are known. With a slight abuse of notation, we denote this extension for case of potentially unknown service requirement as Generalized Exact Scheduling as well. 

Now we look into how much system performance may degrade in $Var(P)$ if the demands and deadlines are unknown.  Let $s_k$ be a binary random variable taking the value of $1$ if the system has access to the demand and deadline of job $k$ and $0$ otherwise. We assume that the probability of $s= i \in \{ 1, 0\}$ is $p( s = i )$, and $s$ is independent with $\at$ and $(\e, \st)$. Following the argument of \eqref{eq:thm1_1}-\eqref{eq:thm1_11}, we have  
\begin{align}
\label{eq:uncertainty-1}
\var(P) =& \E \left[   \int_0^\st    v(\e,\st,\x)^2  d\x \right]  \\ 
\label{eq:uncertainty-3}
=& \; p( s = 1 ) \E \left[  \frac{ \e^2 } { \st }  \Big| s = 1 \right] + p( s = 0 )  \E \left[   \int_{-\infty}^\st    v(\e,\st,\x)^2  d \Big| s = 0 \right] \\
\label{eq:uncertainty-4}
=& \; p( s = 1 ) \E \left[  \frac{ \e^2 } { \st }  \Big| s = 1 \right] +  p( s = 0 ) \E \left[  \frac{ \e^2 } { \st }  \Big| s = 0 \right] \\
\label{eq:uncertainty-5}
& \; + p( s = 0 )  \E \left[   \int_{-\infty}^\st    v(\e,\st,\x)^2  d \Big| s = 0 \right]  - p( s = 0 ) \E \left[  \frac{ \e^2 } { \st }  \Big| s = 0 \right]\\
\label{eq:uncertainty-6}
= & \; \E \left[  \frac{ \e^2 } { \st }  \right] + p( s = 0 ) \left\{    \E \left[   \int_{-\infty}^\st    v(\e,\st,\x)^2  d \Big| s = 0 \right]  - \E \left[  \frac{ \e^2 } { \st }  \Big| s = 0 \right] \right\} \\
\label{eq:uncertainty-7}
= & \;  q_{ES}   + p( s = 0 ) \left\{    \E \left[   \int_{-\infty}^\st    v(\e,\st,\x)^2  d\x   - (\e^2 / \st) \Big| s = 0 \right]     \right\} \\
\end{align} 
where $q_{ES}$ is the optimal cost of Exact Scheduling. Equality \eqref{eq:uncertainty-3} and \eqref{eq:uncertainty-5} use the Law of Total Expectation. From \eqref{eq:uncertainty-7}, the performance degradation due to unknown demands and deadlines is computed to be\footnote{No formula is given for the case of strict demand and deadline because demand and deadline satisfaction cannot be guaranteed without the information of demands and deadlines. }
\begin{align} 
\label{eq:uncertainty-8}
&\begin{dcases}
p( s = 0 )  \E\left[ {c'_{ES}}^2 \min\{ \st, \e/{c'_{ES}} \} -  (\e^2 / \st)  \right] &  \text{in case of soft demand} \\
p( s = 0 ) \E\left[  { c''_{ES} } \e   - (\e^2 / \st)  \right] &  \text{in case of soft deadline} 
\end{dcases}
\end{align}
For the instances in the testbed and the synthetic data (from Section \ref{sec:data-explain}), the performance degradation is estimated to be $15 \sim 40\%$ of the cost of Exact Scheduling multiplied by the ratio of jobs with unknown service requirements (the overall cost is $100 + 15 p( s = 0 ) \sim 100 + 40 p( s = 0 ) \%$ of that of Exact Scheduling. 
This estimation is obtained by realizing that the value in \eqref{eq:uncertainty-8} is upper-bounded by that of Equal Service for strict demand and deadlines requirement. So the performance difference between Exact Scheduling and Equal Service in Figure \ref{fig:ES_strict} can be used to estimate this value.

}

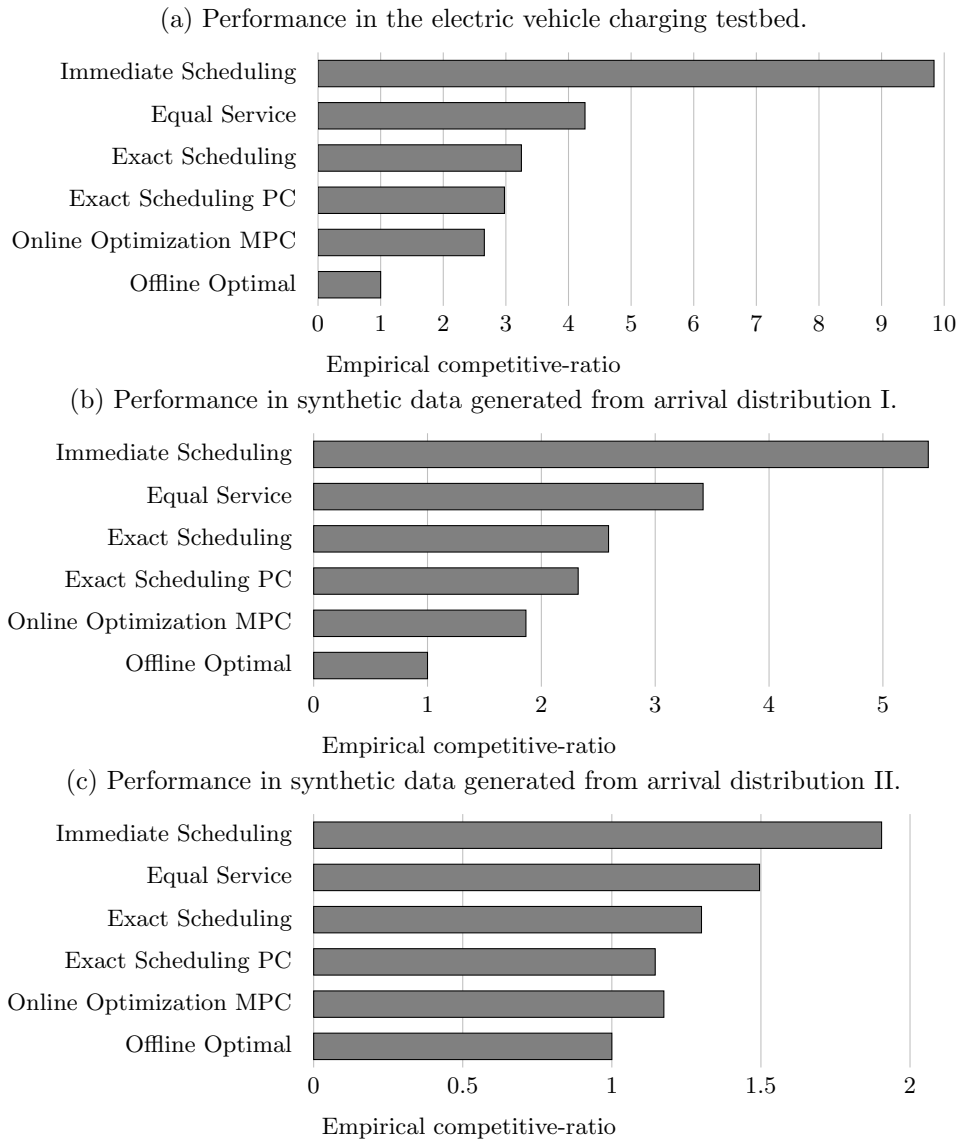
\begin{figure}
    \centering

\begin{subfigure}[b]{\textwidth}
	\centering
	\caption{Performance in the electric vehicle charging testbed.}
	       \begin{tikzpicture}
	
	\begin{axis}[width=0.6\columnwidth,height=0.3\columnwidth,xbar, xmin=0, xmax=10, xlabel={\footnotesize Empirical competitive-ratio}, symbolic y coords={
		{Offline Optimal},
		{Online Optimization MPC},
		{Exact Scheduling PC},
		{Exact Scheduling},	
		{Equal Service},
		{Immediate Scheduling}
		},
	ytick=data,
	 axis line style={draw=none},
	 tick style={draw=none},
	xmajorgrids=true,
	ticklabel style={font=\footnotesize},
	xtick = {0,1,2,3,4,5,6,7,8,9,10},
    xlabel style = {anchor=north east}
	]
	\addplot[fill=gray] coordinates {
		(1,Offline Optimal)
		(2.6559,Online Optimization MPC) 
		(3.2481,Exact Scheduling)
		(2.9772,Exact Scheduling PC) 
		(4.2637,Equal Service) 
		(9.8375,Immediate Scheduling)
	};
	\end{axis}
       \end{tikzpicture}
       
    \label{fig:ES_strict-1}
\end{subfigure}
\begin{subfigure}[b]{\textwidth}
	\centering
	\caption{Performance in synthetic data generated from arrival distribution I.}
	    \begin{tikzpicture}
	
	\begin{axis}[width=0.6\columnwidth,height=0.3\columnwidth,xbar, xmin=0, xmax=5.5, xlabel={\footnotesize Empirical competitive-ratio}, symbolic y coords={%
		{Offline Optimal},
		{Online Optimization MPC},
		{Exact Scheduling PC},
		{Exact Scheduling},	
		{Equal Service},
		{Immediate Scheduling}
		},
	ytick=data,
	 axis line style={draw=none},
	 tick style={draw=none},
	xmajorgrids=true,
	ticklabel style={font=\footnotesize},
	xtick = {0,1,2,3,4,5},
    xlabel style = {anchor=north east}
	]
	\addplot[fill=gray] coordinates {
		(1,{Offline Optimal})
		(1.86459256742504,{Online Optimization MPC}) 
		(2.59018732364249,{Exact Scheduling})
		(2.32346230692476,{Exact Scheduling PC}) 
		(3.42076597449134,{Equal Service}) 
		(5.39872216149675,{Immediate Scheduling})		
	};
	\end{axis}
       \end{tikzpicture}
    \label{fig:ES_strict-2}
\end{subfigure}
    \begin{subfigure}[b]{\textwidth}
	\centering
	\caption{Performance in synthetic data generated from arrival distribution II.}
	       \begin{tikzpicture}
	
	\begin{axis}[width=0.6\columnwidth,height=0.3\columnwidth,xbar, xmin=0, xmax=2.1, xlabel={\footnotesize Empirical competitive-ratio}, symbolic y coords={%
		{Offline Optimal},
		{Online Optimization MPC},
		{Exact Scheduling PC},
		{Exact Scheduling},	
		{Equal Service},
		{Immediate Scheduling}
		},
	ytick=data,
	 axis line style={draw=none},
	 tick style={draw=none},
	xmajorgrids=true,
	ticklabel style={font=\footnotesize},
	xtick = {0,0.5,1,1.5,2},
    xlabel style = {anchor=north east}
	]
	\addplot[fill=gray] coordinates {
		(1,{Offline Optimal})
		(1.1745,{Online Optimization MPC}) 
		(1.3007,{Exact Scheduling})
		(1.1455,{Exact Scheduling PC}) 
		(1.4954,{Equal Service}) 
		(1.9047,{Immediate Scheduling})		
	};
	\end{axis}
       \end{tikzpicture}
    \label{fig:ES_strict-3}
\end{subfigure}
    \caption{Performance comparison of algorithms under strict demand and deadline constraints in the testbed~\cite{testbed}. The ratio of each algorithm's empirical variance to the Offline Optimal is averaged over all scheduling instances. The number of instances averaged are $92$ in plot (a) and $500$ in plot (b) and plot (c). The instances used here are described in Section \ref{sec:data-explain}. In plot (b), the arrival distribution I is set to have parameter $\bar \ell = 15$. In plot (c), the arrival distribution II is set to have parameter $\bar \gamma = 2$. For arrival distribution with different parameters from (b) and (c), the performance is shown in Figure \ref{fig:ES_strict_app} in the Appendix \ref{sec:app-simulation}. 
    }
    \label{fig:ES_strict}
\end{figure}

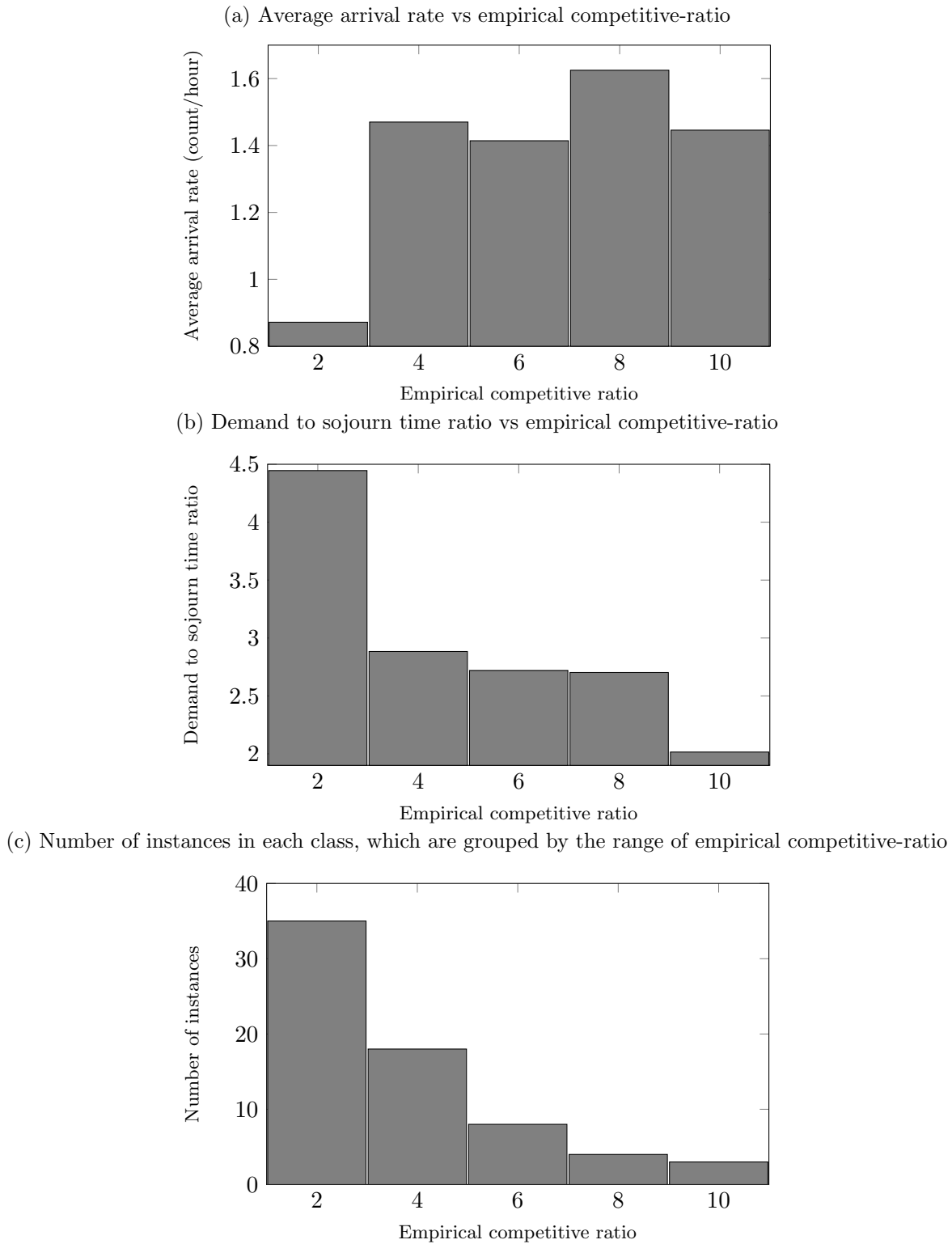
\begin{figure}
\begin{subfigure}[b]{\textwidth}
	\centering
	\caption{Average arrival rate vs empirical competitive-ratio}
	\begin{tikzpicture}

\begin{axis}[%
width=0.5\columnwidth,
height=0.3\columnwidth,
scale only axis,
bar shift auto,
log origin=infty,
xmin=1,
xmax=11,
xtick={ 2,  4,  6,  8, 10},
xlabel style={font=\footnotesize},
xlabel={Empirical competitive ratio},
ymin=0.8,
ymax=1.7,
ylabel style={font=\footnotesize},
ylabel={Average arrival rate (count/hour)},
axis background/.style={fill=white},
legend style={legend cell align=left, align=left, draw=white!15!black}
]
\addplot[ybar, fill=gray,  bar width=46] table[row sep=crcr] {%
2	0.871791170975163\\
4	1.47009135645499\\
6	1.41424798792757\\
8	1.625\\
10	1.44590062111801\\
};
\end{axis}

\end{tikzpicture}%
    \label{fig:hist2}
\end{subfigure}
\begin{subfigure}[b]{\textwidth}
	\centering
	\caption{Demand to sojourn time ratio vs empirical competitive-ratio}
%
%
\definecolor{mycolor1}{rgb}{0.00000,0.44700,0.74100}%
\begin{tikzpicture}

\begin{axis}[%
width=0.5\columnwidth,
height=0.3\columnwidth,
at={(1.011in,0.642in)},
scale only axis,
bar shift auto,
log origin=infty,
xmin=1,
xmax=11,
xtick={ 2,  4,  6,  8, 10},
xlabel style={font=\footnotesize},
xlabel={Empirical competitive ratio},
ymin=1.9,
ymax=4.5,
ylabel style={font=\footnotesize},
ylabel={Demand to sojourn time ratio},
axis background/.style={fill=white},
legend style={legend cell align=left, align=left, draw=white!15!black}
]
\addplot[ybar, fill=gray,  bar width=46] table[row sep=crcr] {%
2	4.44552945110297\\
4	2.88343124468765\\
6	2.72008697715668\\
8	2.70179364013137\\
10	2.01520190443431\\
};

\end{axis}
\end{tikzpicture}%
    \label{fig:hist3}
\end{subfigure}
\begin{subfigure}[b]{\textwidth}
	\centering
	\caption{Number of instances in each class, which are grouped by the range of empirical competitive-ratio}
	\begin{tikzpicture}

\begin{axis}[%
width=0.5\columnwidth,
height=0.3\columnwidth,
at={(1.011in,0.642in)},
scale only axis,
bar shift auto,
xmin=1,
xmax=11,
xtick={ 2,  4,  6,  8, 10},
xlabel style={font=\footnotesize},
xlabel={Empirical competitive ratio},
ymin=0,
ymax=40,
ylabel style={font=\footnotesize},
ylabel={Number of instances},
axis background/.style={fill=white},
legend style={legend cell align=left, align=left, draw=white!15!black},
]
\addplot[ybar, fill=gray,  bar width=46] table[row sep=crcr] {%
2	35\\
4	18\\
6	8\\
8	4\\
10	3\\
};
\addplot[forget plot, color=white!15!black] table[row sep=crcr] {%
1	0\\
11	0\\
};

\end{axis}
\end{tikzpicture}%
    \label{fig:hist1}
\end{subfigure}
\caption{Instance characteristics that allow Exact Scheduling to have  comparable performance with the Offline Optimal. For each instance, the ratio between the cost of Exact Scheduling and that of Offline Optimal (denoted as the empirical competitive-ratio with a slight abuse of notation\textsuperscript{\ref{fn:competitive-ratio}}) is computed. Based on this ratio, instances are grouped into 5 classes, each containing instances for which the empirical competitive-ratio ranges between $[1,3), [3, 5), [5, 7), [7, 9), [9, 11]$. For each group, the average arrival rate and average ratio of demand to sojourn time $\e/\st$ for jobs in each class are shown in (a) and (b), and the number of instances (days) in each class are shown in (c).}
\label{fig:EShist}
\end{figure}

\begin{figure}

\begin{subfigure}[b]{\textwidth}
	\centering
	\caption{Normalized cost for varying unit penalty of unmet demand.}
    \input{fig6a.tex}
	\label{fig:ESsoftdemand1} 
\end{subfigure}
       
\begin{subfigure}[b]{\textwidth}
	\centering
	\caption{Average amount of unsatisfied demands per instance for varying unit penalty of unmet demand.}
	\input{fig6b.tex}
	\label{fig:ESsoftdemand2} 
\end{subfigure}       

\caption{Generalized Exact Scheduling compared to existing algorithms in the case of soft demand constraints. For varying unit penalty $\C$, the empirical costs in all instances are shown. The top plot (a) compares the average empirical costs of all instances for varying values unit penalty $\C$. The cost is normalized by the cost of Offline Optimal (for strict service requirements).\textsuperscript{\ref{ft:cost-name}} The bottom plot (b) shows the average amount of unmet demands in one instance for varying $\C$. The parameters $c_{\text{ES}}', p_{\text{EDF}}, p_{\text{LLF}}$, and $p_{\text{FS}}$ used in Equal Service, Earliest Deadline First, Least Laxity First, and Fair Sharing are set to be the optimal offline values that minimize the average empirical costs. Note that such optimal offline values require the information of all future instances to be computed, so the estimated performance of these algorithms is optimistic.}
\label{fig:ESsoftdemand} 
\end{figure}

%
%

\begin{figure}
        
\begin{subfigure}[b]{\textwidth}
    \centering
    \caption{Normalized cost for varying unit penalty of deadline extension.}
    \input{fig7a.tex}

    \label{fig:ESsoftdeadline1} 
\end{subfigure}

\begin{subfigure}[b]{\textwidth}
\caption{Average amount of deadline extensions per instance for varying unit penalty of deadline extension.}
    \centering
    \input{fig7b.tex}
    \label{fig:ESsoftdeadline2}
\end{subfigure}       
        
\caption{Generalized Exact Scheduling compared to existing algorithms in the case of soft deadline constraints. For varying unit penalty of unmet demand $\D$, the empirical costs in all instances are shown. The top plot (a) compares the average empirical costs of all instances for varying values unit penalty $\D$. The cost is normalized by the cost of Offline Optimal (for strict service requirement).\textsuperscript{\ref{ft:cost-name}} The bottom plot (b) shows the average amount of deadline extension in one instance for varying $\D$. 
The parameters $p_{\text{EDF}}, p_{\text{LLF}}$, and $p_{\text{FS}}$ used in Earliest Deadline First, Least Laxity First, and Fair Sharing are set to be the optimal offline values that minimize the average empirical costs. Note that such optimal offline values require the information of all future instances to be computed, so the estimated performance of these algorithms is optimistic. }
\label{fig:ESsoftdeadline} 
\end{figure}

\subsection{Theoretical analysis}
\label{sec:lower-bound}

{\color{black}
In this section, we compare the performance of online distributed policies and that of online centralized policies. The design problem of centralized scheduler is typically formulated as a Markov decision process, whose optimal solution can only be approximated or computed numerically. To obtain analytic bounds, we formulate it as a constrained functional optimization problem instead. Recall from Section \ref{sec:testalgorithm} that a centralized scheduling policies has the form $\uc$ in \eqref{eq:u_cent1}. It can use all available information of jobs arriving prior to time $t$ in decision making.
The minimum-variance centralized policy can then be obtained as the solution of the following constrained functional optimization problem
\begin{align}\label{eq:centralized_min_var}
& \underset{u : \eqref{eq:rate_constraints} \eqref{eq:demand_constraints}\eqref{eq:deadline_constraints}\eqref{eq:u_cent1}}{\text{minimize}} \;\;\;\;\;  \var(P).
\end{align}
where the optimization variable is the scheduling policy of the form \eqref{eq:u_cent1}.

In order to bound the performance degradation from centralized to distributed algorithms, we first obtain the performance limits for centralized algorithms. Let $X(t)$ be the total remaining demands of jobs arriving before $t$. Let $D$ be a value that satisfies
\begin{align}
\label{eq:variance_x_constraint}
\var( X ) \leq D 
\end{align}
where $\var( X )$ is the stationary variance of $X(t)$.


\begin{lemma}\label{thm:lower_bound}
Under any centralized policy of the form \eqref{eq:u_cent1}, the stationary variance of $P^\dagger(t)$ is lower-bounded by 
\begin{align}
\label{eq:performance_lowerbound}
\var(P) \geq  \frac{1}{4 D}  \pinti ^2 \E[ \e^2  ]^2 .
\end{align}
\end{lemma}

\begin{corollary}
 \label{thm:lower_bound2}
Let $\var(P)$ be the stationary variance of $P(t)$ attained by Exact Scheduling \eqref{eq:exact_scheduling_st}. Let $ \var (  P^\dagger ) $ be the optimal performance among the centralized scheduling policies of the form \eqref{eq:u_cent1} that satisfy the rate, demand, and deadline constraints \eqref{eq:rate_constraints}--\eqref{eq:deadline_constraints}. Then, the following condition holds: 
\begin{align}
\label{eq:solution-Q-final}
\var(P)\leq 
\frac{  4  \E \left[  \e^2 /\st \right]  (  \E[  \st \e^2  ] +  \pinti  \E\left[ \st  \e  \right] ^2 ) }{  \E[\e^2 ]^2    } \var( P^\dagger ) .
\end{align}
\end{corollary}

Corollary \ref{thm:lower_bound2} bounds the ratio of stationary variance achievable by the optimal distributed algorithm to that achievable by any centralized algorithms. Here, both the optimal distributed algorithm and the optimal centralized algorithm are subject to strict constraints on demands \eqref{eq:demand_constraints} and deadlines \eqref{eq:deadline_constraints}. To evaluate this bound, consider a special case where both $\e$ and $\st$ are deterministic and $\e = a \st$ for some scalar $a>1$. Then bound \eqref{eq:solution-Q-final} reduces to $\var(P)\leq 
4 (1 +  \pint \st ) \var( P^\dagger ) $. This formula suggests that Exact Scheduling becomes more competitive to Centralized Optimal algorithms when arrival rate is small. This observation is consistent with the observation that large performance difference between Exact Scheduling and Offline Optimal mostly happens at instances with large arrival rate in the testbed (Figure \ref{fig:hist2}). As $ \pint \st \rightarrow 0$, the bound suggests that the cost of Exact Scheduling remains within approximately $4$ times of the cost of optimal centralized algorithm. Recall from Figure \ref{fig:ES_strict-1} that the cost of Exact Scheduling performs approximately $3$ times of that of Offline Optimal, and 1.2 times of that of Online Optimization MPC. This data suggests Exact Scheduling may perform much better than this performance lower-bound suggests. The pessimistic estimate of bound \eqref{eq:solution-Q-final} may be due to the fact that the proof of Corollary \ref{thm:lower_bound2} uses a loose bound for \eqref{eq:variance_x_constraint} (see \eqref{eq:lower-bound-0}--\eqref{eq:lower-bound-2} in the Appendix \ref{app:corollary2}). Alternatively, a tighter bound can be obtained as 
\begin{align}
\var(P)\leq 
\frac{  4 D \E \left[  \e^2 /\st \right]   }{  \E[\e^2 ]^2    } \var( P^\dagger ) .
\end{align}
where an estimate of $D$ can be computed numerically given the arrival distribution.


\subsubsection{Proof of Lemma \ref{thm:lower_bound}}
\label{sec:lower-bound-proof}

In this section, we present the proof of Lemma \ref{thm:lower_bound}. Let the stationary variance of $X(t)$ be bounded as in \eqref{eq:variance_x_constraint}. We consider the following problem:
\begin{equation*}
\label{eq:problem-offline}
\mathcal Q_{\text{on}} = \underset{\uc: \eqref{eq:u_cent1}\eqref{eq:variance_x_constraint}}{\text{minimize}} \;
 \lim_{T \rightarrow \infty} \frac{1}{T} \int_{0}^{T} \var (  \p(t) )  dt,
\end{equation*}
where the optimization is taken over all centralized policies of the form \eqref{eq:u_cent1} satisfying \eqref{eq:variance_x_constraint}.
The Lagrangian of $\mathcal Q_{\text{on}} $ is 
\begin{align}
L (\uc ; \gamma ) = \lim_{T \rightarrow \infty} \frac{1}{T} \int_{0}^{T} \var (  \p(t) )  + \gamma  ( \var(X(t)) -  D ) dt  , 
\end{align}
where $\gamma \geq 0$ is the Lagrangian multiplier associated with the constraint \eqref{eq:variance_x_constraint}. Observe that 
\begin{align}
\label{eq:lower-bound-formula}
\inf_{\uc: \eqref{eq:u_cent1}} L (\uc; \gamma) \leq \mathcal Q_{\text{on}} \leq  \var(P),
\end{align}
where $\var(P)$ is the stationary service capacity variance of any policy. 
Then, we can derive a lower bound of $\var(P)$ via solving $\inf_{\uc: \eqref{eq:u_cent1}} L (\uc; \gamma)$ as follows.
\begin{lemma}
\label{thm:lower_bound-subproblem1}
Let $\bar X $ and $\bar P$ be defined as the stationary mean of $X(t)$ and $P(t)$, respectively.  The infimum in $\inf_{\uc: \eqref{eq:u_cent1}} L (\uc ; r)$ is attained when $P(t)$ is set to be 
\begin{align}
\label{eq:lower-bound-opt-solution2}
&P(t)  = \sqrt{\gamma} ( X(t) - \bar X ) + \bar P
\end{align}
at all time $t$, and the infimum value is given by 
\begin{align}
 \label{eq:lowerbound_opt_cost3}
 \inf_{\uc: \eqref{eq:u_cent1}} L (\uc ; \gamma)  &=\sqrt{ \gamma }\pinti  \E[ \e^2  ] - \gamma D.
\end{align}
\end{lemma}
Lemma \ref{thm:lower_bound-subproblem1} is proven in Appendix \ref{sec:lemma4}. From \eqref{eq:lower-bound-opt-solution2}, the optimal solution of $\inf_{\uc: \eqref{eq:u_cent1}} L (\uc; \gamma)$ satisfies
\begin{align}
\label{eq:lowerbound_opt_cost1}
 \var (P) + \gamma  \var(X) =  2 \gamma \var(X) .
\end{align}
Combining \eqref{eq:lowerbound_opt_cost3} and \eqref{eq:lowerbound_opt_cost1} leads to \begin{align}
\label{eq:lower-bound-8}
&\var( X) = \frac{1}{2 \sqrt{ \gamma} }\pinti  \E[ \e^2  ]  .
\end{align}
Since $X(t)$ also satisfies the constraint \eqref{eq:variance_x_constraint}, the Lagrangian multiplier $\gamma$ is lower-bounded by 
\begin{align}
\label{eq:lower-bound-f} 
&\frac{1}{2  D} \pinti \E[ \e^2  ] \leq \sqrt{ \gamma} . 
\end{align}
Therefore, we obtain 
\begin{align}
\label{eq:lowerbound-final}
\var ( P )  &\geq \inf_{\uc: \eqref{eq:u_cent1}} L (\uc ; \gamma) \\
\label{eq:lowerbound-final2}
&\geq \frac{\sqrt{ \gamma} }{2}  \pinti \E[ \e^2  ] \\
\label{eq:lowerbound-final3}
&\geq  \frac{ 1}{4 D} \pinti ^2 \E[ \e^2  ]^2  .
\end{align}
where \eqref{eq:lowerbound-final} is due to \eqref{eq:lower-bound-formula}; \eqref{eq:lowerbound-final2} is due to \eqref{eq:lowerbound_opt_cost3} and \eqref{eq:lower-bound-8}; and \eqref{eq:lowerbound-final3} is due to \eqref{eq:lower-bound-f}.

}

\section{Balancing predictability and stability under non-stationary job arrivals}
\label{sec:nonstationary} 

{\color{black}
Building upon the results of stationary job arrivals, we consider a more general setting of non-stationary job arrivals in this section. The non-stationary setting is particularly appealing for practical applications since dynamic capacity management is most crucial when the workload is not stationary. 

In contrast to the stationary setting, there exists a tradeoff between maximizing the stability and predictability of the service capacity in the non-stationary setting. We characterize this tradeoff and introduce a Pareto-optimal distributed algorithm that balances stability and predictability. Below, we first formally define the notion of Pareto-optimality, which recovers maximum predictability and maximum stability as two special cases (Section \ref{sec:varying-rate-problem}). Then, at one extreme case of maximizing predictability, we show that Generalized Exact Scheduling is the optimal algorithm (Section \ref{sec:varying-rate-problem1}). In the other extreme case of maximizing stability, we characterize the optimal algorithm and notice an interesting connection to the well-known YDS algorithm \cite{yao1995scheduling}, which is optimal in a related, deterministic worst-case setting (Section \ref{sec:varying-rate-problem2}). Generalizing the two extreme cases, we describe the Pareto-optimal algorithm that balances predictability and stability (Section \ref{sec:varying-rate-problem3}). 

\subsection{Problem formulation}
\label{sec:varying-rate-problem} 

In this section, we relax our previous stationary assumptions on the arrival process. We assume that the arrival distribution is a non-stationary independently marked Poisson process with the intensity function $\tilde \pint (\at)$ and a mark joint density measure $f_\at(\e, \st) g_\at( \C ) h_\at (\D)$ (see Section \ref{sec:optimization_problems}).
We consider the following three types of policies:
\begin{align}
\label{eq:u_nonstationary_simple}
r_k(t) &= u(\at_k , \y_k(t) ,  \x_k(t) )   \geq 0  \quad k \in \hV\\
\label{eq:u_nonstationary_simple_marked}
r_k(t) &= \bar u (\at_k , \y_k(t) ,  \x_k(t) , \C_k , \D_k )   \geq 0  \quad k \in \hV\\
\label{eq:u_stat4}
r_k(t) &= v(\at_k , \e_k , \st_k,  \x_k(t) )   \geq 0  \quad k \in \hV,
\end{align}
In these forms, the scheduling policies to change over time, which allows us account for the changing arrival rate over time. They are also online and distributed in the sense that the service rate of each job is determined using only the information of the same job but not other jobs. 

We seek to design policies that balance three important performance criteria: the quality of service, the service capacity variance associated with the predictability, and its mean square associated with the stability. In the most basic settings involving the first two criteria, we consider the optimization problem
\begin{align}
\label{eq:min_variance+nonst}
& \underset{u : \eqref{eq:rate_constraints}\eqref{eq:demand_constraints}\eqref{eq:deadline_constraints}\eqref{eq:u_nonstationary_simple} }{\text{minimize}} \;\;\; \lim_{T \rightarrow \infty} \frac{1}{T} \int_{0}^{T} \var (P(t) ) dt 
\end{align}
for the case of strict service requirement and the optimization problem  
\begin{align}
\label{eq:min_variance+softdemand+softdeadline_nonst_marked}
 &\underset{\bar u:\eqref{eq:rate_constraints}\eqref{eq:u_nonstationary_simple_marked} }{\text{minimize}} \;\;\; \lim_{T \rightarrow \infty} \frac{1}{T} \int_{0}^{T}\Big( \var (P(t) ) + \E[ \U(t) ] + \E[ \Q(t) ] \Big)  dt 
\end{align}
for the case of soft service requirement. 
In a more advanced settings involving all three criteria, we consider
\begin{align}
\label{eq:cost12}
\underset{ v:\eqref{eq:rate_constraints} \eqref{eq:demand_constraints}\eqref{eq:deadline_constraints}\eqref{eq:u_stat4} }{\text{minimize}} \;\;\;&\lim_{T \rightarrow \infty} \frac{1}{T} \int_{0}^{T} \alpha \E [ P(t) ]^2 +\beta \var (P(t) )   dt
\end{align}
for the case of strict service requirement. This case also includes 
\begin{align}
\label{eq:cost11}
 \underset{v :\eqref{eq:rate_constraints} \eqref{eq:demand_constraints}\eqref{eq:deadline_constraints}\eqref{eq:u_stat4} }{ \text{minimize} } \;\;\;&\lim_{T \rightarrow \infty}\frac{1}{T} \int_{0}^{T} \E [ P(t) ]^2 dt
\end{align}
as an important special case of maximizing stability. More optimization problems can be formulated by combining the terms in~\eqref{eq:min_variance+nonst}--\eqref{eq:cost11}. Although such optimization problems are beyond the scope of this paper, our techniques can be used to analyze such problems as well. 

\subsection{Maximizing predictability}
\label{sec:varying-rate-problem1}

In this section, we consider the optimization problem of maximizing the predictability in service capacity (\ie minimizing the service capacity variance) in the settings of strict service requirement and soft service requirement. The problem allows us to systematically balance the service quality and service capacity variance and also admits insightful closed-form solutions. Recall from Section \ref{sec:hard} that $\var(P(t))$ is minimized at a flat service rate because having peaks and fluctuations in the service rate within a job's sojourn time amplifies the uncertainties of the future arrivals to cause large $\var(P(t))$. In fact, this intuition holds beyond stationary arrivals, and so does for the optimal algorithm for non-stationary arrival distribution.

\begin{theorem}
\label{thm:nonstationary_maxpred_hard}
The optimal solution of \eqref{eq:min_variance+nonst} is Exact Scheduling, defined by 
\begin{align}
\label{eq:exact_timevarying}
&u( \at, \y ,  \x  )  = 
\begin{dcases}
\frac{\y }{  \x  }  & \x > 0 \\ 
0 & \text{otherwise}
\end{dcases}.
\end{align}
\end{theorem}

Furthermore, Generalized Exact Scheduling is also optimal under soft demand and deadline constraints, despite the non-stationary arrival distribution. 
\begin{corollary}
\label{cor:nonstationary_maxpred_soft}
The optimal solution of \eqref{eq:min_variance+softdemand+softdeadline_nonst_marked} is 
\begin{align}
\label{eq:exact_softdemand+deadline_nonst}    
& \bar u (\at,  \y , \x , \C,\D )  = \begin{dcases}
\frac{\y }{\x}      & \text{if } \x > 0 \text{ and } \frac{\y }{\x} \leq \min\left\{  \frac{\C }{2} ,  \sqrt{\D} \right\}   \\
\frac{\C }{2}     & \text{if } \x > 0 \text{ and } \frac{\y }{\x} > \frac{\C}{2} \text{ and } \frac{\C}{2} \leq \sqrt{\D}  \\ 
\sqrt{\D} \; 1\{\y > 0\}&  \text{otherwise} 
\end{dcases},
\end{align}
\end{corollary}
Analogously to its stationary-case counterpart in Section \ref{sec:soft_demand_deadline}, unit costs for unmet demands and deadlines $(\C_k, \D_k)$ determines the tradeoffs between reducing service capacity variance versus allowing unmet demands or deadline extension as Figure~\ref{fig:optimal policy_softdemand&deadline}. The policy operates in three regimes: 
\begin{itemize}
\item \textit{High penalties regime.} Both demands and deadlines are satisfied. The service rates are identical to the ones produced by Exact Scheduling \eqref{eq:exact_scheduling_st}.  
\item \textit{Low demand penalty regime.} All deadlines are strictly enforced with potentially unsatisfied demands. The service rates are identical to the ones produced by policy \eqref{eq:exact_softdemand_st}.
\item \textit{Low deadline penalty regime.} Demands are strictly satisfied with potential deadline extensions. The service rates are identical to the ones produced by policy \eqref{eq:exact_softdeadline_st}. 
\end{itemize}
Thus, Corollary~\ref{cor:nonstationary_maxpred_soft} also recovers the results from previous sections as the special cases: it becomes Corollary \ref{lem:stationary_softdemand+deadline} when the job arrival distribution is stationary; it becomes Theorem~\ref{thm:nonstationary_maxpred_hard} as the unit costs for unmet demands and deadlines $(\C_k, \D_k)$ approach infinity; it becomes Theorem~\ref{thm:nonstationary_maxpred_hard}.

To prove Theorem~\ref{thm:nonstationary_maxpred_hard}, we take analogous steps to Theorem \ref{thm:1_stationary_problem1}. We consider the optimization problem that relaxes the form of the scheduling policy from \eqref{eq:u_nonstationary_simple} to \eqref{eq:u_stat4}: 
\begin{align}
\label{eq:cost10_1}
 \underset{v: \eqref{eq:rate_constraints}\eqref{eq:demand_constraints}\eqref{eq:deadline_constraints}\eqref{eq:u_stat4} }{\text{minimize}} \;\;\;&  \int_{0}^{T} \var (P(t) ) dt,
\end{align}
where the time horizon $T$ is assumed to be finite. Compared with the stationary setting, obtaining a closed-form solution of \eqref{eq:cost10_1} requires additional treatment to account for the non-stationarity in arrival distribution.

\begin{lemma}
\label{thm:dynamic_problem2}
Exact Scheduling $v( \at, \e , \st,  \x )  = (\e/ \st)1\{ \x > 0 \}$ is the optimal solution of \eqref{eq:cost10_1}.
\end{lemma}

The proof of Lemma \ref{thm:dynamic_problem2} uses the following lemma. 

\begin{lemma}
\label{prop:poisson-nonstationary}
The mean and variance of $P(t)$ under the policy \eqref{eq:u_stat4} is given by 
\begin{align}
\label{eq:P-mean-v-2}
&\E[ P(t) ] = \int_{(\e,\st)\in S }  \int_{0}^\st    v(t+\x-\st,\e,\st,\x) \pinti (t+\x-\st, \e,\st) d\x d\e  d\st  \\
\label{eq:P-variance-v-2}
&\var( P(t) ) = \int_{(\e,\st)\in S }  \int_{0}^\st  v(t+\x-\st,\e,\st,\x)^2 \pinti (t+\x-\st, \e,\st) d\x d\e d\st   ,
\end{align}
\end{lemma}

Lemma \ref{prop:poisson-nonstationary} can be proved following similar steps to the prove of Lemma \ref{prop:poisson} (see Appendix \ref{sec:lem2} for more detail). Now we are ready to prove Lemma \ref{thm:dynamic_problem2}.

\begin{proof}{Proof (Lemma \ref{thm:dynamic_problem2}).}

From Lemma \ref{prop:poisson-nonstationary}, the objective function of \eqref{eq:cost10_1} satisfies 
\begin{align}
\int_{0}^{T} \var ( P(t) ) dt  &= \int_{t=0}^{T}  \int_{(\e,\st)\in S }  \int_{\x=0}^\st  v(t+\x-\st, \e,\st,\x)^2 \pint(t+\x-\st,\e,\st) d\x d\e d\st dt \\
&= \int_{(\e,\st)\in S }  \left\{  \int_{\x=0}^\st \int_{t=0}^{T}  v(t+\x-\st, \e,\st,\x)^2 \pint(t+\x-\st,\e,\st) dt d\x  \right\} d\e d\st  .
\end{align}
Moreover, the constraints of \eqref{eq:cost10_1} can be rewritten into 
\begin{align}
\label{eq:sec5p-x}
&\int_{\x = 0}^{\st} v(\at, \e,\st,\x)  d\x = \e & \at \in \hT , (\e,\st) \in S\\
\label{eq:sec5p-y}
&  0 \leq v(\at, \e,\st,\x)  \leq 1 & \at \in \hT , (\e,\st) \in S, \x \in [ 0, \st ] 
\end{align}
For any $(\e,\st) \in S$, the optimal solution of \eqref{eq:cost10_1} is attained at the minimum of the following optimization problem:
\begin{align}
\label{eq:thm3_obj}
\underset{v : \eqref{eq:sec5p-x} \eqref{eq:sec5p-y} }{\text{minimize}} 	\;\;	& \int_{\x=0}^\st \int_{t=0}^{T}  v(t+\x-\st, \e,\st,\x)^2 \pint(t+\x-\st,\e,\st)  dt d\x 
\end{align}
From integration by substitution, the objection function of \eqref{eq:thm3_obj} satisfies  
\begin{align}
 \int_{\x=0}^\st \int_{t=0}^{T}  v(t+\x-\st, \e,\st,\x)^2 \pint(t+\x-\st,\e,\st)  dt d\x   &=  \int_{\x = 0}^\st  \int_{\at = \x-\st}^{T+\x-\st}  v(\at, \e,\st,\x)^2 \pint(\at,\e,\st)  d\at d\x \\
 &= \int_{\x = 0}^\st  \int_{\at = 0}^{T}  v(\at, \e,\st,\x)^2 \pint(\at,\e,\st)  d\at d\x, 
\end{align}
where the last equality is due to the assumption that $\pint(\at,\e,\st) = 0$ if $\at \notin [0,T-\st]$.
The Lagrangian of \eqref{eq:thm3_obj} is 
 \begin{align}
 \nonumber
L(v; \mu, \nu)   =& \int_{\x = 0}^\st \int_{\at = \x-\st}^{T+\x-\st} v(\at, \e,\st,\x)^2 \pint(\at,\e,\st) d\at d\x - \int_{\at = 0}^{T} \mu_{\e,\st}(\at)  \int_{\x = 0}^{\st}  v(\at, \e,\st,\x)  d\x  d\at
\\
& + \int_{\at = 0}^{T}  \int_{\x = 0}^{\st}  ( \bar \nu_{\e,\st}(\at,\x) -   \underline \nu_{\e,\st}(\at,\x) ) v(\at, \e,\st,\x)   d\x d\at ,
\end{align}
where $\mu_{\e,\st}(\at) $ is the Lagrange multiplier associated with constraint \eqref{eq:sec5p-x}; $\underline \nu_{\e,\st}(\at,\x)\geq 0$ is the Lagrange multiplier associated with the constraint $v(\at, \e,\st,\x)\geq 0$, and $\underline \nu_{\e,\st}(\at,\x)$ is the Lagrange multiplier associated with the constraint $\bar v(\at, \e,\st,\x)\leq 1$. A necessary condition for $v^*$ to be the optimal scheduling policy is that $L(v; \mu, \nu)$ is stationary at $v = v^*$. 
After some tedious manipulation, the stationary condition can be computed as follows: 
\begin{align}
v^*(\at ,  \e,\st,\x) = \frac{ \mu_{\e,\st}(\at) + \underline \nu_{\e,\st}(\at,\x) - \bar \nu_{\e,\st}(\at,\x) }{\pint(\at,\e,\st)    }. 
\end{align}
We observe that $\underline \nu_{\e,\st}(\at,\x) = 0$ when $v^*(\at, \e,\st,\x) > 0$. Combining this condition with \eqref{eq:sec5p-x} and \eqref{eq:sec5p-y} leads to
\begin{align}
\frac{ \mu_{\e,\st}(\at) - \bar \nu_{\e,\st}(\at,\x) }{\pint(\at,\e,\st)    } > 0. 
\end{align}
We first suppose $v^*(\at, \e,\st,\x) = 0$ at some $\x \in [0, \st)$. Then for that $\x$, we have
\begin{align}
v^*(\at, \e,\st,\x) = 0 &= \frac{ \mu_{\e,\st}(\at) - \bar \nu_{\e,\st}(\at,\x) }{\pint(\at,\e,\st)    } + \frac{ \underline \nu_{\e,\st}(\at,\x)  }{\pint(\at,\e,\st)    }  \\
&> \frac{ \mu_{\e,\st}(\at) - \bar \nu_{\e,\st}(\at,\x) }{\pint(\at,\e,\st)    } .
\end{align}
The last inequality cannot hold because the left hand size equals zero while the right hand side is strictly positive. So there is a contradiction, Therefore, $v^*(\at, \e,\st,\x)$ must take non-zero values at all $\x \in [0, \st)$. We then suppose that $v^*(\at, \e,\st,\x_1)  <  v^*(\at, \e,\st,\x_2)  = 1 $ for some $\x_1, \x_2 \in [0, \st)$. Then 
\begin{align}
v^*(\at, \e,\st,\x_1) = \frac{ \mu_{\e,\st}(\at)  }{\pint(\at,\e,\st)    }  > \frac{ \mu_{\e,\st}(\at) - \bar \nu_{\e,\st}(\at,\x) }{\pint(\at,\e,\st)    } = v^*(\at, \e,\st,\x_2) . 
\end{align}
This is also a contradiction, so $v^*(\at, \e,\st,\x_1)$ takes a constant value at all $\x \in [0, \st)$. Therefore, the optimal solution of \eqref{eq:cost10_1} is $v( \at, \e , \st,  \x )  = (\e/ \st)1\{ \x > 0 \}$, which is Exact Scheduling. \hfill \qed
\end{proof}

It shall be noted that the optimal value of Exact Scheduling can be represented by control policy of the form \eqref{eq:u_nonstationary_simple}, and the optimal value of \eqref{eq:min_variance+nonst} is lower bounded by that of \eqref{eq:cost10_1}. Therefore, Exact Scheduling is also optimal for \eqref{eq:min_variance+nonst}, yielding Theorem~\ref{thm:nonstationary_maxpred_hard}. 

The optimization problem \eqref{eq:min_variance+softdemand+softdeadline_st} can also be solved in a similar manner. The optimal solution of \eqref{eq:min_variance+nonst} is also the point-wise minimum of 
 \begin{align}
\int_{\at = 0}^{T} \int_{\R_+}   \int_{\R_+} \left\{ \int_{\x = 0}^\st  v(\at, \e,\st,\x)^2 +  \C v(\at, \e,\st,\x)   d\x + \D ( \hst(\at, \e, \st) - \st )  \right\}   \pint(\at,\e,\st) f(\C) f(\D)  d\C d\D d\at  .
\end{align}
From this observation, we can get Corollary~\ref{cor:nonstationary_maxpred_soft} by computing 
\begin{align}
v* = \arg \min_{v} \left\{ \int_{\x = 0}^\st  v(\at, \e,\st,\x)^2 +  \C v(\at, \e,\st,\x)   d\x + \D ( \hst(\at, \e, \st) - \st )  \right\}  
\end{align}
and converting $v^*$ to take the form \eqref{eq:u_stat4}.

\subsection{Maximizing stability}
\label{sec:varying-rate-problem2}

In this section, we consider the optimization problem of maximizing the stability in service capacity (\ie minimizing the service capacity mean square) in the settings of strict service requirements. This problem yields a scheduler that has a striking analogy to the YDS algorithm~\cite{yao1995scheduling}, an optimal scheduler in a deterministic problem. 

Recall from Section \ref{sec:stationary_systems} that achieving stability and predictability are not mutually conflicting goals when it comes to the design of distributed algorithms given a stationary arrival distribution. However, when the arrival process is non-stationary, there is a tradeoff between these two goals, and the minimum-service-capacity-variance algorithm (Exact Scheduling) does not minimize the service capacity mean square. This fact can be easily seen in the example instance of Figure~\ref{fig:ES_ex}. In this instance, the arrival rate increases over time, and Exact Scheduling is likely to incur a substantial cost at a later time (Figure~\ref{fig:ES_ex1}). Meanwhile, an ideal algorithm should account for the increment in future arrivals by serving previous jobs more aggressively than Exact Scheduling (Figure~\ref{fig:ES_ex}).

\begin{figure}
    \begin{subfigure}[b]{\textwidth}
              \centering
        \includegraphics[width=80mm]{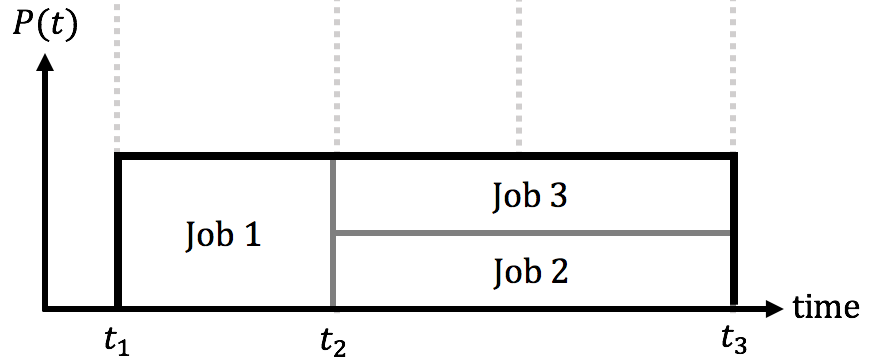} 
        \caption{The behavior of Exact Scheduling. Exact Scheduling is likely to incur a substantial cost at a later time. }
        \label{fig:ES_ex1}
    \end{subfigure}
    \\
    \begin{subfigure}[b]{\textwidth}
              \centering
        \includegraphics[width=80mm]{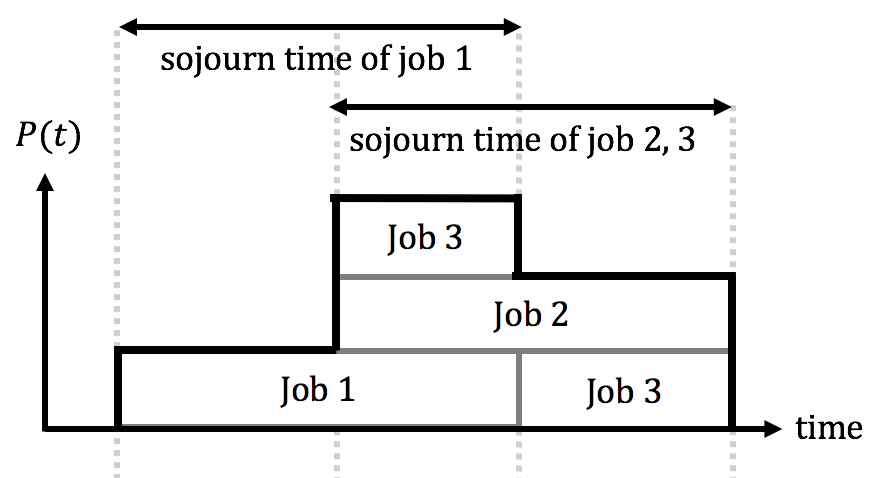}
        \caption{Ideal behavior. The service rate of job 1 is increased to account for potentially large arrivals in the future. }
        \label{fig:ES_ex2}
    \end{subfigure}
    \caption{This example demonstrates why Exact Scheduling does not maximize stability. This instance has a small arrival rate initially but higher arrival rate at a later time. }    \label{fig:ES_ex}
\end{figure}

Formally, the optimal algorithm must satisfy the condition stated below. 
\begin{corollary}
\label{lem:nonstationary_maxstab_hard}
The optimal solution of \eqref{eq:cost11} has the following properties: for each job profiles $(\at , \e,\st) \in ( \hT , S)$, 
\begin{itemize}
\item[(i)] $\E[ P(h)]$ takes a constant value for any time $h$ at which $v(\at ,\e, \st ,\at+\st-h) > 0$.
\item[(ii)] If $v(\at ,\e, \st ,\at+\st-h) > 0$ for some $h$ and $v(\at ,\e, \st ,\at+\st-h') = 0$ for some $h'$, then $\E[ P(h') ] \geq \E[ P(h)]$.
\end{itemize}
\end{corollary}

Corollary \ref{lem:nonstationary_maxstab_hard} states that a job receives non-zero service rate only during the period at which the service capacity is expected to be low. Specifically, if the service capacity is expected to be lower at $h$ than $h'$, no service should be provided at $h'$ without first providing service at $h$ (without exploiting the expected low service capacity level at $h'$). This desired property is formalized into conditions (i) and (ii). Condition (i) flattens out the service capacity for times at which service rate is non-zero because fluctuations of $\E[ P(h)]$ during these intervals compromise stability. Condition (ii) picks the period of lowest expected service capacity to serve. In the example of Figure~\ref{fig:ES_ex}, condition (i) forces $\E [P(t)]$ to be constant during the time intervals $[t_1, t_2]$ and $[t_2, t_3]$. Condition (ii) constrains $\E[ P(h') ] \geq \E[ P(h)]$ for any $h \in [t_1, t_2),  h' \in [t_2, t_3)$ so that, during interval $[t_1, t_2]$, job $1$ is not served beyond an extent that makes $\E[ P(h) ]$ in $[t_1, t_2]$ higher than $\E[ P(h') ]$ in $[t_2, t_3]$. Consequently, the resulting service capacity has less fluctuations (Figure~\ref{fig:ES_ex2} for an intuitive illustration). 

The above property is commonly observed in a well-known class of algorithm: `Valley Filling' or `Water Filling'. This class of algorithms was proposed for many budget allocation problems such as CPU scheduling~\cite{yao1995scheduling}, temperature and energy control~\cite{bansal2007speed}, electric vehicle charging in deterministic settings~\cite{gan2013optimal,gan2013real}, Parallel Gaussian Channels~\cite[Chapter 9]{cover2012elements}, optimal packet scheduling~\cite{6253062}. Furthermore, the optimal policy in Corollary \ref{lem:nonstationary_maxstab_hard} also has an interesting similarity to the YDS algorithm~\cite{yao1995scheduling,bansal2007speed}. Specifically, the YDS algorithm is the solution of 
\begin{equation}
\label{eq:YDS_optimization_problem}
\begin{aligned}
\; \underset{r \geq 0: \eqref{eq:demand_constraints}\eqref{eq:deadline_constraints}}{ \text{minimize} } 
& &\frac{1}{T} \sum_{t = 0}^T P(t)^\alpha  
\end{aligned}
\end{equation}
where $\alpha > 1$ is some constant. The optimal solution of $\YDS$ satisfies the following conditions: for any job $k \in \hV$, 
\begin{itemize}
\item[(iii)] $P(h)$ takes a constant value at any time $h$ at which $r_k(h) > 0$.
\item[(iv)] If $r_k(h) > 0$ for some $h$ and $r_k(h') = 0$ for some $h'$, then $P(h') \geq P(h)$.
\end{itemize}
When we replace $\E[P(t)]$ with $P(t)$ and $u$ with $r$, condition (i)--(ii) in Corollary \ref{lem:nonstationary_maxstab_hard} become condition (iii)--(iv) above. This relationship allows us to adapt the computational tool of the YDS algorithm to find the optimal distributed policy in our setting. 

Algorithm \ref{alg:max_stability} finds the optimal distributed policy that maximizes stability.
Let $\hV(t_1,t_2) = \{ (\at,\e,\st) :  \at \geq t_1 , \at+\st \leq t_2 , (\e,\st) \in S \}$ be the set of job profiles that have an arrive time after $t_1$ and a deadline before $t_2$. When $(\at,\e,\st) \in \hV( t_1, t_2)$, we say that jobs with $(\at,\e,\st)$ present in the interval $[t_1,t_2]$. Let $w(t_1, t_2)$ denote the expected cumulative demand of jobs present in the interval $[t_1,t_2]$, \ie 
\begin{align}
w(t_1, t_2) = \int_{\at \geq t_1} \int_{\at+\st \leq t_2, (\e,\st) \in S}\e \pint(\at,\e,\st) d\e d\st d\at . 
\end{align}
Intuitively, $w(t_1, t_2) $ is the minimum expected demand that must be supplied during a time interval $[t_1, t_2]$ to satisfy the demand requirements. We further define the intensity of an interval $[t_1, t_2]$ as 
\begin{align}
I (t_1, t_2) = \frac{ w(t_1, t_2) } { t_2 - t_1} .
\end{align}
The algorithm finds the service rate $v^*(\at,\e,\st,\x)$ in descending order of the intensity $I (t_1, t_2) $ in which a job present.
Specifically, it iterates the following procedures. At each step, it finds a maximum intensity interval 
\begin{align}
\label{alg:step1}
[t_1, t_2] = \arg \max_{[t_1, t_2]} I (t_1, t_2)
\end{align}
is computed (line 3). Jobs present in the maximum intensity interval are served subject to 
\begin{align}
\label{eq:e-const}
&\E[ P(h) ] = \frac{ I (t_1, t_2) }{  t_2 - t_1 } \;\;\;\; h \in [t_1, t_2)\\
\label{eq:e-const2}
&v^*(\at,\e,\st,\x) \leq \bar v (\at,\e,\st,\x).
\end{align}
in ascending order of their deadlines (line 4). As jobs not present in $\hV(t_1,t_2)$ are assigned zero service capacity during $[t_1, t_2]$, the maximum service rates of jobs not present in $\hV(t_1,t_2)$ are set to be zero during $[t_1,t_2]$ (line 6). Because jobs in $\hV(t_1,t_2)$ are already scheduled, before the next iteration, they are removed from the arrival statistics $\pint(\at,\e,\st)$ as if the arrival probability of jobs present in $\hV(t_1,t_2)$ is zero (line 7). Using the modified arrival statistics, the algorithm repeats the same process of finding a new maximum intensity interval, computing the service rates of jobs present during this interval, and modifying job statistics. 

\begin{algorithm}
\SetKwInOut{Input}{input}\SetKwInOut{Output}{output}
\Input{$\pint(\at,\e,\st)$ } 
\Output{$v^*(\at,\e,\st,\x)$ } 
initialize $ \bar v (\at,\e,\st,\x) \leftarrow \infty$\;
\While{ $\pint(\at,\e,\st)  > 0$ for some $(\at,\e,\st)$}{
identify the maximum intensity interval $[t_1, t_2]$ by solving \eqref{alg:step1} \; 
compute $v^*(\at,\e,\st,\x)$ for job profiles in $\hV(t_1,t_2)$ s.t. \eqref{eq:e-const} and \eqref{eq:e-const2}  \;
\For{ $(\at,\e,\st) \notin  \hV(t_1,t_2)$ }{ 
set $\bar v(\at,\e,\st,\x)  \leftarrow 0$ for any $\at + \st - \x \in [t_1,t_2]$ \;
set $\pint(\at,\e,\st)  \leftarrow 0$\;
}}
\caption{Computing the optimal distributed policy that maximizes stability}
\label{alg:max_stability}
\end{algorithm}

We can observe that Algorithm 1 also minimizes $\max_{t \in \hT} \E[ P(t) ]$. This property is derived from the fact that the value of $\max_{t \in \hT} \E[ P(t) ]$ cannot be smaller than what is required to schedule jobs present in $\hV(t_1,t_2)$ in the first iteration because the interval $[t_1,t_2]$ found by the first iteration is the most intensive interval of an instance. Moreover, Corollary \ref{lem:nonstationary_maxstab_hard} and the optimal algorithm can be generalized to account for service capacity variance. We discuss this generalization in the next section.

\subsection{Balancing stability and predictability}
\label{sec:varying-rate-problem3}

In the previous two sections, we show the optimal policies that maximizes predictability and stability separately. Beyond the two special cases, however, balancing stability and predictability is a much more complex problem, and it is too ambitious to seek a purely analytic solution. Instead, we characterize the Pareto-optimality condition for the distributed algorithm that balances predictability and stability in this section.  

Recall that, with regard to maximizing predictability, it is favorable to have a fixed service rate over time (see Section \ref{sec:varying-rate-problem1}). Meanwhile, with regard to maximizing stability, it is desirable to have a fixed service capacity over time (see Section \ref{sec:varying-rate-problem2}). These special cases provide us with the intuition that the evenness of $r_k(t)$ and $\E[P(t)]$ may be used to balance predictability and stability. We formalize this intuition in the following theorem, which generalizes the result in Theorem~\ref{thm:nonstationary_maxpred_hard} and Corollary \ref{lem:nonstationary_maxstab_hard}.

\begin{theorem}
\label{thm:nonstationary_balance_hard}
The optimal solution of \eqref{eq:cost12} has the following properties: for each job profiles $(\at, \e,\st) \in (\hT, S)$, 
\begin{itemize}
\item[(i)] $\alpha \E[ P(h)] + \beta v(\at ,\e, \st ,\at+\st-h)$ takes a constant value for any time $h$ at which $v(\at ,\e, \st ,\at+\st-h) > 0$,
\item[(ii)] $\alpha \E[ P(h') ] \geq \alpha \E[ P(h)] + \beta v(\at ,\e, \st ,\at+\st-h) $ for any time $h' \in [\at , \at+\st]$ at which $v(\at ,\e, \st ,\at+\st-h') = 0$.
\end{itemize}
\end{theorem}

When $\alpha = 0$, Theorem \ref{thm:nonstationary_balance_hard} essentially states that Exact Scheduling maximizes predictability. This is because condition (ii) cannot happen when $\alpha = 0$, so the optimality condition reduces to the case when $v(\at ,\e, \st ,\x)$ is constant at all $\x \in [ 0 , \st]$. When $\beta = 0$, the conditions for $\alpha \E[ P(h)] + \beta v(\at ,\e, \st ,\at+\st-h)$ reduces to the conditions stated in Corollary \ref{lem:nonstationary_maxstab_hard}.

\begin{proof}{Proof (Theorem \ref{thm:nonstationary_balance_hard})}
From Lemma \ref{prop:poisson}, the objective function of \eqref{eq:cost11} is equivalent to 
\begin{align}
& \int_{0}^{T} \alpha \E [ P(t) ]^2 + \beta \var ( P(t) ) dt \\
&=  \alpha   \int_{0}^{T}  \left \{ \int_{(\e,\st) \in S} \int_{\x = 0}^{\st} v( t + \x - \st ,\e,\st, \x ) \pint(t + \x - \st,\e,\st)  d\x d\e d\st \right \}^2 dt  \\
&\;\;\;\;\;+ \beta  \int_{0}^{T}  \int_{(\e,\st) \in S} \int_{\x = 0}^{\st} v( t + \x - \st ,\e,\st, \x )^2 \pint(t + \x - \st,\e,\st)  d\x d\e d\st dt  
\end{align}
Moreover, the constraints of \eqref{eq:cost12} are equivalent to 
\begin{align}
\label{eq:sec5p3-x}
&\int_{\x = 0}^{\st}  v ( \at ,\e,\st, \x ) d\x = \e , & (\e, \st) \in S , \;\; \at  \in \hT \\
\label{eq:sec5p3-u}
&v ( \at ,\e,\st, \x ) \geq 0 , &(\e, \st) \in S , \;\; \at \in \hT , \;\; \x \in [ 0 , \st ] .
\end{align}
The Lagrangian associated with problem \eqref{eq:cost11} is 
\begin{align*}
L( v ; \mu, \nu) 
=& \alpha \int_{0}^{T}  \left \{ \int_{(\e,\st) \in S} \int_{\x = 0}^{\st} v( t + \x - \st ,\e,\st, \x ) \pint(t + \x - \st,\e,\st)  d\x d\e d\st \right \}^2 dt  \\
 & + \beta \int_{(\e,\st) \in S} \int_{\x = 0}^{\st}  \int_{\at = 0}^{T} v( \at ,\e,\st, \x )^2 \pint(\at,\e,\st)  d \at d\x d\e d\st   \\
&	- \int_{\at = 0}^T \int_{(\e,\st) \in S}  \mu(\e,\st,\at)   \int_{\x = 0}^{\st}  v ( \x,\e,\st,\at )  d\x  d\e d\st d\at \\
&	 -  \int_{\at = 0}^T \int_{(\e,\st) \in S} \int_{\x = 0}^{\st}  \nu(\at, \e,\st,\x) v(\x,\e,\st,\at) d\x   d\e d\st  d\at  , 
\end{align*}
where $\mu(\e,\st,a)$ is the Lagrange multiplier associated with \eqref{eq:sec5p3-x}, and $\nu(\at, \e,\st,\x)  \geq 0$ is the Lagrange multiplier associated with \eqref{eq:sec5p3-u}. We can alternatively consider $L: U \rightarrow \real$ as a functional defined on the function space $U$ of policies. Let $U_f \subset U$ be the space of feasible scheduling policies, \ie 
\begin{align}
U_f = \{ v  :  v   \text{ satisfies }  \eqref{eq:sec5p3-x} \&\eqref{eq:sec5p3-u} \}.
\end{align}
Now we impose an infinitesimal perturbation to $v$ such that
\begin{align}
v' = v + \epsilon \hp  \in U_f.
\end{align}

Let $G: (\hT, U) \rightarrow \R$ be the following functional:
\begin{align}
G(t; v) = \E[ P(t) ]  = \int_{(\e,\st) \in S} \int_{\at = t-\st}^{t} v(\e,\st,\at,\x) \pint(\at,\e,\st) d\at d\e d\st  .
\end{align}
The difference in Lagrangian can be written as 
\begin{align}
&L( v' ; \mu, \nu) - L( v  ; \mu, \nu) \\
=&\alpha  \int_{0}^{T} 2 \;G(t; u) \left \{ \int_{(\e,\st) \in S} \int_{\at = t-\st}^{t} \epsilon \hp(\at, \e, \st, \x) \pint(\at,\e,\st)  d\at d\e d\st \right \} dt  \\ \nonumber 
&+\beta \int_{(\e,\st) \in S} \int_{\x = 0}^{\st}  \int_{\at = 0}^{T} 2 \epsilon \hp(\at, \e, \st, \x)  v( \at ,\e,\st, \x ) \pint(\at,\e,\st)  d \at d\x d\e d\st  \\		\nonumber 
&- \int_{(\e,\st) \in S} \int_{\at = 0}^T  \mu(\e,\st,\at)    \int_{\x = 0}^{\st}  \epsilon \hp(\at, \e, \st, \x)  d\x d\at d\e d\st \\	\nonumber 
&-   \int_{\x = 0}^{\st}  \nu(\e,\st,\at,\x)  \epsilon \hp (\e,\st,\at,\x)  d\x  d\at d\e d\st + O(  \epsilon^2 ) \\	
\label{eq:dym2_lem1}
=& \alpha  \int_{(\e,\st) \in S} \int_{\x = 0}^{\st} \int_{\at = 0}^{T} 2  \epsilon G(t; u) \hp(\at, \e, \st, \x)  \pint(\at,e,\st)  d\at   d\x  d\e d\st \\	\nonumber 
&+ \beta \int_{(\e,\st) \in S} \int_{\x = 0}^{\st}  \int_{\at = 0}^{T} 2 \epsilon \hp(\at, \e, \st, \x)  v( \at ,\e,\st, \x ) \pint(\at,\e,\st)  d \at d\x d\e d\st  \\	\nonumber 
&- \int_{(\e,\st) \in S} \int_{\at = 0}^T  \int_{\x = 0}^{\st}  \epsilon ( \mu(\e,\st,\at) \hp(\at, \e, \st, \x)  +  \nu(\e,\st,\at,\x)  \hp (\e,\st,\at,\x) ) d\x  d\at d\e d\st + O(  \epsilon^2 ) \\
\nonumber
=& \epsilon \int_{(\e,\st) \in S} \int_{\x = 0}^{\st} \int_{\at = 0}^{T}  \big\{ 2 ( \alpha G(t; u) +  \beta v( \at ,\e,\st, \x ) ) \pint(\at,\e,\st) -  \mu(\e,\st,\at) -  \nu(\e,\st,\at,\x) \big\} \hp(\at, \e, \st, \x)  d\at   d\x  d\e d\st \\
\label{eq:dym2_lem2} 
&+ O(  \epsilon^2 ),
\end{align}
where \eqref{eq:dym2_lem1} is obtained using integration by substitution.
For a functional $L$ to be stationary at some $v \in U_f$, the first term should be zero for any $\hp(\at, \e, \st, \x) $ satisfying the constraint $v' \in U_f$. From \eqref{eq:dym2_lem2}, the stationary point of $L$ satisfies 
\begin{align}
\label{eq:dym2_lem3} 
\alpha \E[ P(t) ] + \beta v (\at, \e, \st, \at+\st-h_1) &= \alpha G(t; u) + \beta v (\at, \e, \st, \at+\st-h_1) \\
&=  \frac{ \mu(\e,\st,\at) +  \nu(\e,\st,\at,\x) } { 2 \pint(\at,\e,\st) }  
\end{align}
for any $ \x \in [ 0 ,\st]$, $(\e, \st) \in S$, $\at\in \hT$.
Since the optimal solution of \eqref{eq:cost11} is a stationary point of $L$, \eqref{eq:dym2_lem3} is the necessary condition for optimality: 
\begin{itemize}
\item[(i)] For any job profiles $(\at, \e, \st)$, if the service rate is strictly positive at $h_1, h_2$ such that $h_1 \neq h_1$ and $h_1 , h_1 \in [\at, \at + \st]$, then
\begin{align}
\label{eq:dym2_lem4} 
v(\at-h_1+\st,\e,\st,\at) = v (\at-h_2+\st,\e,\st,\at) = 0.
\end{align}
Combining \eqref{eq:dym2_lem3} and \eqref{eq:dym2_lem4} leads to
\begin{align}
\alpha \E[ P(h_1) ]  + \beta v (\at, \e, \st, \at+\st-h_1) =  \frac{ \mu(\e,\st,\at) }{ 2 \pint(\at,\e,\st) } = \alpha \E[ P(h_2) ] + \beta v (\at, \e, \st, \at+\st-h_2) .
\end{align}
\item[(ii)] For any job profiles $(\at, \e, \st)$, if its service rate is strictly positive at $h \in [\at, \at+\st]$ and zero at $h' \in [\at, \at+\st]$, then 
 \begin{align}
\alpha \E[ P(h) ]  +\beta  v (\at, \e, \st, \at+\st-h_1) = \frac{ \mu(\e,\st,\at)  } { 2 \pint(\at,\e,\st)}  \leq  \frac{ \mu(\e,\st,\at) + \nu(\e,\st,\at,\x) } { 2 \pint(\at,\e,\st) }  = \alpha \E[ P(h') ].
\end{align}
\end{itemize}
\hfill \qed
\end{proof}
}

\section{Conclusion}

{\color{black}
As it becomes more common for service systems to have a dynamic capacity that instantaneously adapts to demand, the goal of providing a high quality of service (\eg meeting deadlines) while minimizing the variance of service capacity has received increasing attention.  Though there exists an extensive literature analyzing existing algorithms, few analytic results characterizing optimal policies were known in such settings. 

In this paper, we characterize the optimal policies in many common scenarios, stationary and non-stationary arrivals, strict or soft demands, with or without deadline extensions, and a variety of objective functions. The results highlight that novel generalizations of Exact Scheduling maximize the predictability (\ie minimizes the service capacity variance) under both stationary and non-stationary Poisson arrival processes. 

When the goal is to balance the stability and predictability, more complex policies turn out to be optimal. Such optimal policies include a novel variation of the YDS algorithm, which is the optimal algorithm for a related worst-case framework. This connection and the proof of optimality suggest a new bridge between the stochastic and worst-case scheduling communities, and this connection will be interesting to be explore in future work. 

In addition to characterizing optimal distributed policies, we also bound the gap between the performance of distributed policies and centralized policies using both theory and experiment. To derive theoretical bounds, we adapt optimal control techniques, which can also be extended to derive performance bonds for related problems. In an experiment conducted in the electrical vehicle charging testbed, we observe that the optimal distributed policies could nearly match the performance of centralized policies.  

Going forward, we note two interesting future directions of this paper. 

\textit{Non-asynmtotic optimization in related problems} Typically, the analysis of scheduling policies for non-stationary settings with deadlines has been done using asymptotic regimes, \eg heavy-traffic regimes. However, the techniques we develop in this paper do not require asymptotic approximations.  Thus, in addition to the results we have proven, our techniques are also an important contribution. We hope these techniques will inspire the discovery of other optimality results in the context of deadline 

\textit{The design space of partially-centralized schedulers} In Section \ref{sec:emp}, we made an initial attempt to explore the middle ground between fully-centralized algorithm and fully-distributed algorithms. We demonstrated numerically that Exact Scheduling PC can have a competitive performance to fully-centralized algorithms while preserving the scalability of Exact Scheduling, which is fully-distributed. Beyond this algorithms, the vast design space of partially-centralized algorithms may have great potential in balancing performance versus scalability but is scarcely explored. Therefore, we note this as an important future direction.
}

\bibliographystyle{unsrt}
\bibliography{ref1}

\newpage

\begin{APPENDICES}

\section{Proof of Lemma \ref{prop:poisson}}
\label{sec:lem2}

In this section, we present results that are useful for proving our main theorems. First, we restate one part of the Campbell's theorem, which is relevant to our proofs.
\begin{theorem}[Campbell formula for marked Poisson processes \cite{baccelliBremaud}]
\label{thm:Campbell}
Consider an independently marked Poisson point process $\{x_k\}\subset \real^d$ with intensity measure $\pinti: \real^d \rightarrow \real_+$ and marks in $\real^p$ with distribution $F(dz)$. Let $g: \real^d\times\real^p \rightarrow \real$ be a measurable function satisfying 
\begin{align}
 \label{eq:Campbell-assumption}
 \int_{ \real^d}\int_{\real^p}g(x,z)^2 \pinti (dx) F(dz)< \infty. 
\end{align}
Then, the random sum 
\begin{align}
G = \sum_{k\in\mathbb{Z}} g(x_k,z_k) 
\end{align}
is absolutely convergent with probability one and satisfies 
\begin{align}
&\E[  G ] =  \int_{\real^d}\int_{\real^p} g(x,z) \pinti (dx)F(dz) \\
&\var( G ) =  \int_{\real^d}\int_{\real^p} g(x,z)^2 \pinti (dx)F(dz).
\end{align}
\end{theorem}

Throughout, we consider a scheduling policy \eqref{eq:u_stationary_complex}, which is defined by a function $v: S  \times \real \rightarrow \real_+$ as follows:
\begin{align}
&\rate_k(t) = v(   \e_k ,   \st_k , \x_k (t)   )  &\; k\in \hV.
\end{align}
The function $v$ satisfies 
\begin{align}
\int_{\at_k}^{\infty} v(\e_k, \st_k,   \at_k  +   \st_k - t   )^2  dt 
= \int_{\at_k}^{\infty} r_k (t)^2 dt 
\leq \e_k 
\end{align}
where the maximum value of the integral is attained by Immediate Scheduling $r_k(t) = 1 \{ t \in [\at_k, \at_k + \e_k)  \}$ subject to constraint \eqref{eq:rate_constraints}. Therefore, we have
\begin{align}
 \int_{(\e,\st)\in S } \int_{\real} v(\e,\st, \x) ^2   d\x \pinti f(\e,\st)  d\e  d\st 
\label{eq:Campbell-assumption2}
\leq \pinti \int_{(\e,\st)\in S } \e    f(\e,\st)  d\e  d\st = \E[ \e ]  \pinti < \infty
\end{align}
Combining \eqref{eq:Campbell-assumption2} and Theorem \ref{thm:Campbell}, we obtain \eqref{eq:P-mean-v-2} and \eqref{eq:P-variance-v-2}:
\begin{align*}
&\E[ P(t) ] = \int_{(\e,\st)\in S }  \int_{0}^\st    v(\e,\st,\x) \pinti f(\e,\st) d\x d\e  d\st  \\
&\var( P(t) ) = \int_{(\e,\st)\in S }  \int_{0}^\st    v(\e,\st,\x)^2 \pinti f(\e,\st) d\x d\e d\st   ,
\end{align*}
which yields Lemma \ref{prop:poisson}.

\section{Proof of Proposition \ref{prop:poisson_P}} 
\label{sec:poisson_P_proof}

We observe that 
\begin{align}
\label{eq:st_meanP_fact1}
\int_{-\infty}^{\infty}\int_{0}^{\infty}  \frac{\partial}{d\x} \fint (\y,\x) \y d\y d\x 
&=\int_{0}^{\infty} \y  \left\{ \int_{-\infty}^{\infty} \frac{\partial}{d\x} \fint (\y,\x)   d\x \right\} d\y  \\
\label{eq:st_meanP_fact2}
&= - \int_{0}^{\infty} \y \lim_{L\rightarrow \infty} \fint (\y,L)  d\y  \\
\label{eq:st_meanP_fact3}
&= -  \pint  \E[  \e -  \he(\e,\st) ]
\end{align}
where \eqref{eq:st_meanP_fact1} holds because bounded $S$ implies that $\fint (\y,\infty) = 0$. 
Therefore, the stationary mean of the service capacity satisfies  
\begin{align}
\E[ P(t) ] &= \E\left[ \sum_{k \in \hV} u  (\y_k(t) , \x_k(t)) \right] \\ 
&=\int_{-\infty}^{\infty}  \int_{0}^{\infty}  \fint (\y,\x) u( \y,\x)    d\y d\x \\ 
\label{eq:st_meanP_3}
&= - \int_{-\infty}^{\infty}  \int_{0}^{\infty}\frac{d} {d \y} \left( \fint (\y,\x) u( \y,\x)  \right) \y  d\y d\x \\ 
\label{eq:st_meanP_4}
&=\int_{-\infty}^{\infty}\int_{0}^{\infty} \left( \frac{\partial }{d\x}  \fint (\y,\x) + \pinti f(\y,\x) \right) \y  d\y d\x \\
\label{eq:st_meanP_5}
&=  -  \pint \E[  \e -  \he(\e,\st) ] +  \pint  \E[ \e ] \\
&= \pint \E[  \he(\e,\st) ]
\end{align}
Here, \eqref{eq:st_meanP_3} is due to Integration by Parts, \eqref{eq:st_meanP_4} is due to \eqref{eq:conservation}, \eqref{eq:st_meanP_5} is due to \eqref{eq:st_meanP_fact1}--\eqref{eq:st_meanP_fact3}.

\section{Proof of Theorem \ref{thm:stationary_softdemand}} 

Since the constraints of \eqref{eq:min_variance+softdemand_st} is hard to solve, we first consider providing a lower bound on its optimal solution. Again, we consider the class of control policies representable by \eqref{eq:u_stationary_complex} and the optimization problem 
\begin{align}
\label{eq:cost2}
\underset{v: \eqref{eq:rate_constraints} \eqref{eq:deadline_constraints}\eqref{eq:u_stationary_complex}}{\text{minimize}}& \;\;\;\var (P ) + \E[ \U ]  .
\end{align}
Because the constraint set of \eqref{eq:cost2} contains that of \eqref{eq:min_variance+softdemand_st}, the optimal value of \eqref{eq:cost2} lower-bounds the optimal value of \eqref{eq:min_variance+softdemand_st}.
Therefore, to prove Theorem \ref{thm:stationary_softdemand}, it suffices to solve \eqref{eq:cost2} (in the next lemma) and observe that its optimal solution is representable by a control policy of the form \eqref{eq:u_stationary_simple}.

\begin{lemma}
\label{lem:stationary_softdemand}
The optimal solution of \eqref{eq:cost2} is 
\begin{align}
\label{eq:exact_policy_2_soft}
v(\e,\st,\x)  =  \min\left\{ \;\frac{ \C }{ 2 } , \;\frac{ \e }{ \st } \;\right\} 1 \left \{   \x > 0 \right \} , \end{align}
and it achieves the optimal value \eqref{eq:min_variance+softdemand_st_cost}.

\end{lemma}

\begin{proof}{Proof.}
First, we derive an analytical formula for $\E[ \U ]$ as a function of the scheduling policy $v$. Let 
\begin{align}
\label{eq:he_def}
\he(\e,\st) = \int^\st_{-\infty} v(\e,\st,\x)  d\x, 
\end{align}
be the actual amount of service received by a job with demand $\e$ and sojourn time $\st$. The amount of unsatisfied demand for this job is $\e - \he(\e,\st)$. Additionally, $\he(\e,\st)$ satisfies 
\begin{align}
\label{eq:ehat_constraint}
&0 \leq\he(\e,\st) \leq \e, &\forall (\e,\st) \in S.
\end{align}
Consequently, the stationary mean of $U$ satisfies 
\begin{align}
\E[ \U ] &= \lim_{t \rightarrow \infty} \E \left[ \sum_{k \in \hV: \at_k + \st_k = t} ( \e_k - \he(\e_k,\st_k) ) \right] \\
\label{eq:U_value}
&=   \int_{(\e,\st) \in S}( \e- \he(\e,\st) ) \pinti f(\e,\st) d\e d\st 
\end{align}
Then, we use \eqref{eq:U_value} to rewrite \eqref{eq:cost2} as follows
\begin{align}
\label{eq:thm2_cost1}
&  \inf_{v: \eqref{eq:rate_constraints} \eqref{eq:deadline_constraints}\eqref{eq:u_stationary_complex}} \;\; \var(P) + \; \E[ \U ]  \\
\label{eq:thm2_cost2}
&= \inf_{\substack{\he: \eqref{eq:ehat_constraint} }} \left[\inf_{v: \eqref{eq:rate_constraints} \eqref{eq:deadline_constraints}\eqref{eq:u_stationary_complex}\eqref{eq:he_def}} \var(P) + \C\int_{(\e,\st) \in S}( \e- \he(\e,\st) )  \pinti f(\e,\st) d\e d\st \right]\\
\label{eq:thm2_cost3}
&= \inf_{\substack{\he:\eqref{eq:ehat_constraint} }} \left[ \left\{ \inf_{v: \eqref{eq:rate_constraints} \eqref{eq:deadline_constraints}\eqref{eq:u_stationary_complex}\eqref{eq:he_def}}\var(P) \right\}+ \C\int_{(\e,\st) \in S}( \e- \he(\e,\st) )  \pinti f(\e,\st)  d\e d\st \right].
\end{align}
Equality \eqref{eq:thm2_cost3} holds because, constrained on $\he(\e,\st) = \int^\st_0 v(  \e ,\x, \st ) d\x$ for some fixed $\he$, the second term of \eqref{eq:thm2_cost2} is not a function of $v$. From Lemma \ref{thm:static_problem}, the first term of \eqref{eq:thm2_cost3} admits the closed-form expression 
\begin{align}
\label{eq:thm2_cost4}
\inf_{v : \eqref{eq:rate_constraints} \eqref{eq:deadline_constraints}\eqref{eq:u_stationary_complex}} \var(P)  = \int_{(\e,s) \in S} \frac{ \he(\e,\st)^2 }{\st} \pinti f(\e,\st) d\e d\st ,
\end{align}
which is attained by 
\begin{align}
\label{eq:exact_policy_s}
& v(\e,\st,\x)  = \frac{ \he(\e,\st) }{ \st } .  
\end{align}
Substitute \eqref{eq:thm2_cost4} into \eqref{eq:thm2_cost3} yields
\begin{align}
\label{eq:thm2_cost5}
 \inf_{\substack{\he:\eqref{eq:ehat_constraint} }} \int_{(\e,\st) \in S} \left\{ \frac{ \he(\e,\st)^2 }{\st} + \C(\e- \he(\e,\st)) \right\}\pinti f(\e,\st)  d\e d\st ,
\end{align}
where the optimization variable is now $\he$ instead of $v$. To derive a closed-form solution of \eqref{eq:cost2}, we can minimize the integrand of \eqref{eq:thm2_cost5} point-wisely. By doing so, we observe that, for each $(\e,\st) \in S$, a necessary and sufficient condition for optimality is 
\begin{align}
\label{eq:exact_policy_s1} 
 \he(\e,\st) = \text{arg} \inf_{ \he:\eqref{eq:ehat_constraint} } \frac{ \he(\e,\st)^2 }{\st} + \C(\e- \he(\e,\st)) = \min \left\{\frac{ \C\st}{ 2} , \e \right\} . 
\end{align}
Combining \eqref{eq:exact_policy_s} and \eqref{eq:exact_policy_s1}, we obtain that \eqref{eq:exact_policy_2_soft} is the optimal solution of \eqref{eq:cost2}. Substitute \eqref{eq:exact_policy_2_soft} into \eqref{eq:thm2_cost5}, we obtain its optimal value \eqref{eq:min_variance+softdemand_st_cost}. \hfill \qed

\end{proof}

Given Lemma \ref{lem:stationary_softdemand}, Theorem \ref{thm:stationary_softdemand} can be derived as follows. It can be verified that scheduler \eqref{eq:exact_policy_2_soft} can be realized as \eqref{eq:exact_softdemand_st} using a scheduling policy of the form \eqref{eq:exact_softdemand_st}. This implies that the optimal solution of problem \eqref{eq:cost2} also lies within the constraint set of problem \eqref{eq:min_variance+softdemand_st}. Because the cost attained by scheduler \eqref{eq:exact_policy_2_soft} is a lower bound on the optimal value of problem \eqref{eq:min_variance+softdemand_st}, the optimal solution of problem \eqref{eq:min_variance+softdemand_st} is scheduler \eqref{eq:exact_softdemand_st}.

\section{Proof of Theorem \ref{thm:stationary_softdeadline}}

Since the constraints of \eqref{eq:min_variance+softdeadline_st} is hard to solve, we first consider providing a lower bound on its optimal solution. Again, we consider the class of control policies representable by \eqref{eq:u_stationary_complex} and the optimization problem 
\begin{align}
\label{eq:min_variance+softdeadline_st1}
\underset{v: \eqref{eq:rate_constraints} \eqref{eq:demand_constraints}\eqref{eq:u_stationary_complex} }{\text{minimize}}& \;\;\;\var (P) + \E[ \Q ]  .
\end{align}
Because the optimal value of \eqref{eq:min_variance+softdeadline_st1} lower-bounds that of \eqref{eq:min_variance+softdemand_st}, to prove Theorem \ref{thm:stationary_softdeadline}, we can solve \eqref{eq:min_variance+softdeadline_st1} (in the next lemma) and observe that its optimal solution is representable by a control policy of the form \eqref{eq:u_stationary_simple}. 
\begin{lemma}
\label{lem:stationary_softdeadline}
The optimal solution of \eqref{eq:min_variance+softdeadline_st1} is 
\begin{align}
\label{eq:exact_softdeadline_st1}
& v( \e ,\st, \x )  =  \begin{dcases}
\frac{ \e }{ \st } \;1\{ \x > 0 \} & \text{ if } \frac{\e}{\st} \leq  \sqrt{\D}   \\
\sqrt{\D} \;1 \left\{ \x >  \st - \frac{ \e }{ \sqrt{\D} } \right\}  & \text{ otherwise }
 \end{dcases}.
\end{align}
and it achieves the optimal value \eqref{eq:min_variance+softdeadline_st_cost}.
\end{lemma}

\begin{proof}{Proof.}
With a slight abuse of notation, let 
\begin{align}
\label{eq:def-st-hat}
\hst ( \e,\st ) = 
\begin{cases}
\st & \text{if }\;\; v (  \e ,   \st , \x  ) = 0, \;  \forall \x < 0 \\
\st - \min \{ \bar \x : v (  \e ,   \st , \x  ) = 0 , \forall \x \leq \bar \x \}  & \text{otherwise} \\
\end{cases}
\end{align}
denote the actual sojourn time for jobs having a service demand $\e$ and a sojourn time $\st$. Then, the stationary mean of $\Q$ satisfies 
\begin{align}
\E[ \Q ] &=  \D \int_{(\e,\st) \in S}  ( \hst ( \e,\st  ) - \st ) \pinti f(\e,\st) d\e d\st .
\end{align}
The optimization problem \eqref{eq:min_variance+softdeadline_st1} can then be written into
\begin{align}
\label{eq:thm3_cost3}
&  \inf_{v:   \eqref{eq:rate_constraints} \eqref{eq:demand_constraints} \eqref{eq:u_stationary_complex}} \;\; \var(P) +  \E[\D \Q ]   \\
\label{eq:thm3_cost4} 
&= \inf_{\substack{\hst \geq  \st}} \left[ \left\{ \inf_{v:  \eqref{eq:rate_constraints} \eqref{eq:demand_constraints} \eqref{eq:u_stationary_complex} } \var(P) \right\} + \D  \int_{(\e,\st) \in S} ( \hst ( \e,\st  ) - \st ) \pinti f(\e,\st) d\e d\st \right]  \\
\label{eq:thm3_cost5} 
&= \inf_{\substack{\hst \geq  \st}} \int_{(\e,\st) \in S} \left\{ \frac{ \e^2 }{\hst } + \D ( \hst ( \e,\st  ) - \st ) \right\}\pinti f(\e,\st)  d\e d\st ,
\end{align}
where $\inf_{v: \eqref{eq:rate_constraints} \eqref{eq:demand_constraints} \eqref{eq:u_stationary_complex} } \var(P)$ in \eqref{eq:thm3_cost4} is attained by 
\begin{align}
\label{eq:thm3_solution1}
v( \e ,\st, \x ) = \frac{ \e }{ \hst ( \e,\st  ) } .
\end{align}
The optimal choice of deadline extensions $\hst^\star ( \e,\st  )$ is the point-wise maximum of the integrand of \eqref{eq:thm3_cost5}, \ie 
\begin{align}
\label{eq:thm3_solution2}
 \hst^\star(\e,\st) = \text{arg} \inf_{ \he:\eqref{eq:ehat_constraint} } \frac{ \e^2 }{\hst } + \D ( \hst ( \e,\st  ) - \st )= \left\{  \frac{\e}{\sqrt{\D}}, \st \right\} . 
\end{align}
Combining \eqref{eq:thm3_solution1} and \eqref{eq:thm3_solution2}, we obtain \eqref{eq:exact_softdeadline_st1} as the closed-form solution of \eqref{eq:min_variance+softdeadline_st1}.  \hfill \qed

\end{proof}

Given Lemma \ref{lem:stationary_softdeadline}, we are now ready to prove Theorem \ref{thm:stationary_softdeadline}.

\begin{proof}{Proof (Theorem \ref{thm:stationary_softdeadline})}

Recall that the optimal value of problem~\eqref{eq:min_variance+softdeadline_st1} lower-bounds the optimal value of problem~\eqref{eq:min_variance+softdeadline_st}. Therefore, if there is a policy of the form \eqref{eq:u_stationary_simple} that produces identical service rates to \eqref{eq:exact_softdeadline_st1}, it is also optimal for problem~\eqref{eq:min_variance+softdeadline_st}. Next, we show that the policy \eqref{eq:exact_softdeadline_st} satisfies the above description. 

Given any job $k \in \hV$ with $\e \leq \st \sqrt{\D}$, both \eqref{eq:exact_softdeadline_st} and \eqref{eq:exact_softdeadline_st1} produce the service rates $r_k (t) = \e_k/ \st_k$ if $t \in [\at_k , \at_k+\st_k]$ and $r_k (t) = 0$ otherwise. Given any job $k \in \hV$ with $\e > \sqrt{\D}  \st$, \eqref{eq:exact_softdeadline_st1} produces the service rates $r_k (t) = \sqrt{\D}$ if $t \in [\at_k , \at_k+\e/\sqrt{\D}\;]$ and $r_k (t) = 0$ otherwise. Observe that under the policy \eqref{eq:exact_softdeadline_st}, for any $\x(t) > 0$, we have
\begin{align}
\frac{ \y(t) }{ \x(t) }  - \frac{ \e }{ \st } 
= \frac{ \e - \sqrt{\D} (t-\at)  }{ \st - (t-\at) }  - \frac{ \e }{ \st }   
\geq  \frac{( \e/\st- \sqrt{\D} )(t-\at)  }{ \st - (t-\at) } \geq 0,
\end{align}
where the third inequality is due to $\st \geq \sqrt{\D}$. Thus, the policy \eqref{eq:exact_softdeadline_st1} also produce the service rates $r_k (t) = \sqrt{\D}$ if $t \in [\at_k , \at_k+\e/\sqrt{\D}\;]$ and $r_k (t) = 0$ otherwise. 

\end{proof}

\section{Proof of Theorem \ref{thm:stationary_softdemand+deadline}} 

We first consider providing a lower bound of problem~\eqref{eq:min_variance+softdemand+softdeadline_st} by solving the optimization problem 
\begin{align}
\label{eq:min_variance+softdemand+softdeadline_st1}
\underset{v: \eqref{eq:rate_constraints}  \eqref{eq:u_stationary_complex}}{\text{minimize}}& \;\;\;\var (P ) + \E[  \U ] + \E[ \Q ]  .
\end{align}
The solution of problem~\eqref{eq:min_variance+softdemand+softdeadline_st1} is given in the next lemma, which is also a feasible policy for the constraint set of problem~\eqref{eq:min_variance+softdemand+softdeadline_st}. 

\begin{lemma}
\label{lem:stationary_softdemand+deadline_new}
The optimal solution of problem~\eqref{eq:min_variance+softdemand+softdeadline_st1} is 
\begin{align}
\label{eq:exact_softdemand+deadline_st1}
& v( \e ,\st, \x )  = \begin{dcases}
\frac{ \e}{ \st } 1 \left\{ \x > 0 \right\} &  \text{if } \frac{\e}{\st} \leq \min \left \{ \frac{\C }{2} ,  \sqrt{\D} \right\}   \\
\frac{\C}{2} 1 \left\{ \x > 0 \right\} & \text{if } \frac{\e}{\st} >  \frac{\C }{2}  \text{ and }   \frac{\C }{2} \leq \sqrt{\D}  \\
\sqrt{\D} 1 \left\{ \x > \st - \frac{\e}{\sqrt{\D}} \right\} &\text{otherwise}
\end{dcases}.
\end{align}
and it achieves the optimal value \eqref{eq:min_variance+softdemand&deadline_st_cost}. 
\end{lemma}

\begin{proof}{Proof.}
Recall that $\he(\e,\st)$ in \eqref{eq:he_def} denotes the actual service supply for jobs having a service demand $\e$ and a sojourn time $\st$, and $ \hst ( \e,\st )$ in \eqref{eq:def-st-hat} denote the actual sojourn time for such jobs. The optimization problem \eqref{eq:min_variance+softdemand+softdeadline_st1} can be written into
\begin{align}
\label{eq:thm4_cost3}
&  \inf_{v: \eqref{eq:rate_constraints}  \eqref{eq:u_stationary_complex}}\;\;\;\var (P(t) ) + \E[ \C \U ] + \E[\D \Q ]    \\
\label{eq:thm4_cost4} 
&=  \inf_{\substack{\he( \e,\st  ) \geq \e \\ \hst( \e,\st  )  \geq  \st } } \left[ \inf_{\substack{ v:  \eqref{eq:rate_constraints}  \eqref{eq:u_stationary_complex}} } \var(P)  +  \int_{(\e,\st) \in S} \{  \C (  \e - \he ( \e,\st  )  )   + \D ( \hst ( \e,\st  ) - \st ) \} \pinti f(\e,\st) d\e d\st \right]  \\
\label{eq:thm4_cost5} 
&=  \inf_{\substack{\he( \e,\st  ) \geq \e \\ \hst( \e,\st  )  \geq  \st } }  \int_{(\e,\st) \in S}  \left[  \frac{ \he( \e,\st  )^2 }{\hst( \e,\st  ) } +   \C (  \e - \he ( \e,\st  )  ) +   \D ( \hst ( \e,\st  ) - \st ) \right]  \pinti f(\e,\st)  d\e d\st  \\
&=  \inf_{\substack{\he( \e,\st  ) \geq \e \\ \hst( \e,\st  )  \geq  \st } }  \int_{(\e,\st) \in S} C(\e, \hst)    \pinti f(\e,\st)  d\e d\st  ,
\end{align}
where $C(\e, \hst) $ is defined to be
\begin{align}
\label{eq:thm4_solution2}
C(\e, \hst) &:= \frac{ \he(\e,\st)^2 }{\hst( \e,\st  )  } +  \C (  \e - \he ( \e,\st  )  )  +   \D ( \hst ( \e,\st  ) - \st )  \\
&= \begin{dcases}
\frac{ \e^2 }{ \st }  & \text{if }  \hst =  \st  \text{ and } \frac{ \e }{ \st } \leq \frac{ \C }{ 2 }  \\
 \C \left( \e - \frac{ \C \st }{4} \right) & \text{if }  \hst =  \st  \text{ and }  \frac{ \e }{ \st } > \frac{ \C }{ 2 } \\
\frac{ \e^2 }{ \hst }  + \D ( \hst - \st) & \text{if }  \hst > \st  \text{ and } \frac{ \e }{ \hst } \leq \frac{ \C }{ 2 }  \\
 \C \left( \e - \frac{ \C \hst }{4} \right)  + \D ( \hst - \st) & \text{if }  \hst >  \st  \text{ and }  \frac{ \e }{ \hst } > \frac{ \C }{ 2 } \\
\end{dcases}.
\end{align}
Relation \eqref{eq:thm4_cost5} holds because $\inf_{v: \eqref{eq:rate_constraints}  \eqref{eq:u_stationary_complex} } \var(P)$ is attained by 
\begin{align}
\label{eq:thm4_solution1}
v( \e ,\st, \x ) = \frac{ \he( \e,\st  ) }{ \hst ( \e,\st  ) } .
\end{align}
The optimal $\he^*( \e,\st  )$ and $\hst^* ( \e,\st  )$ is the point-wise maximum of the integrand of \eqref{eq:thm4_cost5}.

To derive a closed form expression for $\he^*( \e,\st  )$ and $\hst^* (\e,\st) $, we first show that in the case of $\C^2 / 4 \leq \D$, we have $\hst^* (\e,\st) = \st$. Suppose not and $\hst (\e,\st) = \hst \geq \st$. Then, if $\e \leq  \C \st /  2$, we have 
\begin{align}
C(\e,\hst) - C(\e,\st) &= \frac{ \e^2 }{ \hst }  + \D ( \hst - \st)  - \frac{ \e^2 }{ \st } \\
&= ( \hst - \st) \left( \D - \frac{ \e^2 }{ \st \hst} \right)\\
\label{eq:thm4_eq1}
&\geq ( \hst - \st) \left\{ \D -  \left(\frac{ \C \st}{2} \right)^2 \frac{ 1 }{ \st \hst} \right\}\\
\label{eq:thm4_eq2}
&\geq ( \hst - \st) \left\{ \D - \frac{ \C^2 }{4} \right\}\\
\label{eq:thm4_eq3}
&\geq 0, 
\end{align}
where \eqref{eq:thm4_eq1} is due to $\e \leq  \C \st /  2$; \eqref{eq:thm4_eq2} is due to $\hst > \st$; and \eqref{eq:thm4_eq3} is due to $\C^2 / 4 \leq \D$. 
When $\e \in (  \C \st /  2 ,  \C \hst /  2 ]$, we have 
\begin{align}
C(\e,\hst) - C(\e,\st) &= \frac{ \e^2 }{ \hst }  + \D ( \hst - \st)  -  \C \left( \e - \frac{ \C \st }{4} \right) \\
\label{eq:thm4_eq4}
&\geq \D ( \hst - \st)  +  \left(\frac{ \C \st}{2} \right)^2 \frac{1}{\hst}  - \C \frac{ \C \hst }{2 } +  \frac{ \C^2 \hst }{4} \\
\label{eq:thm4_eq5}
&\geq \D ( \hst - \st)  + \frac{1}{2}   \C^2 ( \st - \hst ) \\
&= ( \hst - \st) \left\{ \D - \frac{ \C^2 }{4} \right\}\\
\label{eq:thm4_eq6}
&\geq 0, 
\end{align}
where \eqref{eq:thm4_eq4} is due to $\e \leq  \C \st /  2$; \eqref{eq:thm4_eq5} is due to $\hst > \st$; and \eqref{eq:thm4_eq6} is due to $\C^2 / 4 \leq \D$. 
When $\e > \C \hst /  2 $, we have 
\begin{align}
C(\e,\hst) - C(\e,\st) &=  \C \left( \e - \frac{ \C \hst }{4} \right)  + \D ( \hst - \st)  -  \C \left( \e - \frac{ \C \st }{4} \right) \\
&=  ( \hst - \st) \left( \D - \frac{ \C^2 }{4} \right) \\
\label{eq:thm4_eq7}
& \geq 0
\end{align}
where \eqref{eq:thm4_eq7} is due to $\C^2 / 4 \leq \D$. Since \eqref{eq:thm4_eq3}, \eqref{eq:thm4_eq6}, and \eqref{eq:thm4_eq7} contradict with the supposition that $\hst(\e,\st) = \hst > \st$ is optimal, we have $\hst^* ( \e,\st  ) = \st$. Then, given $\hst^* ( \e,\st  ) = \st$, the optimal $\he^* ( \e,\st  )$ follows from Lemma \ref{lem:stationary_softdemand}. In a similar manner, we can show that, in the case of $\C^2 / 4 > \D$, the optimal service supply is $\he^* (\e,\st) = \e$. Then, given $\he^* ( \e,\st  ) = \e$, the optimal $\st^* ( \e,\st  )$ follows from Lemma \ref{lem:stationary_softdeadline}. Finally, combining above, we obtain \eqref{eq:exact_softdemand+deadline_st1} as the closed-form solution of \eqref{eq:min_variance+softdemand+softdeadline_st1}.  \hfill \qed

\end{proof}

Theorem \ref{thm:stationary_softdemand+deadline} is an immediate consequence of Lemma \ref{lem:stationary_softdemand+deadline_new}. To see it, recall that the optimal value of problem~\eqref{eq:min_variance+softdemand+softdeadline_st1} lower-bounds that of problem~\eqref{eq:min_variance+softdemand+softdeadline_st}. Moreover, scheduler \eqref{eq:exact_softdemand+deadline_st} of the form \eqref{eq:u_stationary_simple} can produce identical service rates to \eqref{eq:exact_softdemand+deadline_st1}, so it is also optimal for problem~\eqref{eq:min_variance+softdemand+softdeadline_st}.

\section{Proof of Lemma \ref{thm:lower_bound-subproblem1}}
 \label{sec:lemma4}

To solve $\inf_{\uc:  \eqref{eq:u_cent1} } L (\uc; \gamma)$, we first observe that 
\begin{align}
\label{eq:lower-bound-6-0}
 \inf_{\uc:\eqref{eq:u_cent1}} L (\uc; \gamma) 
&=  \inf_{\uc:\eqref{eq:u_cent1}} \lim_{T \rightarrow \infty} \frac{1}{T} \int_{0}^{T} \var (  \p(t) )  + \gamma  ( \var(X(t)) -  D ) dt \\
\label{eq:lower-bound-6}
&\geq   \inf_{\uc:\eqref{eq:u_cent1}}  \lim_{T \rightarrow \infty} \inf_{\uc:\eqref{eq:u_cent1}} \frac{1}{T}   \int_{ 0}^{T} \var (  \p(t) )  + \gamma  ( \var(X(t)) -  D ) dt\\
\label{eq:lower-bound-9}
&=   \lim_{T \rightarrow \infty}\inf_{\uc:\eqref{eq:u_cent1}} \frac{1}{T}   \int_{ 0}^{T} \var (  \p(t) )  + \gamma  ( \var(X(t)) -  D )  dt,
\end{align}
where equality \eqref{eq:lower-bound-6-0} holds by the definition of $L (\uc; \gamma) $, inequality \eqref{eq:lower-bound-6} holds because $(1/T) \int_{0}^{T} \var (  \p(t) )  + \gamma  ( \var(X(t)) -  D ) dt $ is always less than $(1/T) \inf_{\uc:\eqref{eq:u_cent1}}  \int_{ 0}^{T} \var (  \p(t) )  + \gamma  ( \var(X(t)) -  D )dt $.

Now we consider representing the integral of \eqref{eq:lower-bound-9} as  the sum of $\E[ (P(t_n)-\bar P)^2+\gamma   (X(t_n)-\bar X)^2]$ at discrete points in time, where $\{t_{n}\}$ have a fixed sampling interval $\hh  = t_{n+1} -  t_n, \forall n \in \Z$. So, the dynamics of $X(t_{n})$ satisfies 
\begin{align}
\label{eq:x-discrete}
&X(t_{n+1})  = X(t_n )  + A(t_n,\hh) - \hh P(t_n) , 
\end{align}
where $A(t_n,\hh)$ is the demand added to $X$ due to new arrivals in the time interval $[t_n, t_{n+1})$ (the total demands of jobs arriving at this interval), $\hh P(t_n) $ is the total service provided during $[t_n, t_{n+1})$. Here, the service policy $\uc$ is assumed to produce constant values during each sampling intervals, \ie $\uc (k, t_1, A_t) = \uc (k, t_2, A_t) $ for any $t_1, t_2 \in [t_n, t_{n+1}),$\footnote{As the sampling interval goes to zero, $\uc$ can realize any continuous function $\uc (k, t, A_t)$ of $t$.} so the service capacity takes the constant value $P(t_n)$ during this interval. 

Let $L_{\hh,N}(u; \gamma )$ is defined by 
\begin{align}
\label{eq:def_L}
L_{\hh,N}(\uc; \gamma ) :=& \E\left[ \gamma (X(t_{N})-\bar X)^2 \right]  +\sum_{k = 0}^{N- 1} \E \left[   (P(t_{k})  - \bar P)^2   + \gamma (X(t_{k})-\bar X)^2 \right] .
\end{align}
Observe that \eqref{eq:lower-bound-9} satisfies 
\begin{align}
& \lim_{T \rightarrow \infty} \inf_{\uc:\eqref{eq:u_cent1}} \frac{1}{T}   \int_{ 0}^{T}  \E[ (P(t)-\bar P)^2 + \gamma (  (X(t)-\bar X)^2 -D) ] dt \\ 
\label{eq:lower-bound-10-0}
&=  \lim_{T \rightarrow \infty}  \inf_{\uc:\eqref{eq:u_cent1} }  \lim_{\hh \rightarrow 0} \frac{1}{T}   L_{\hh, \lceil T/h \rceil }(\uc; \gamma ) \hh- \gamma D  \\
\label{eq:lower-bound-10}
&= \lim_{T \rightarrow \infty}   \lim_{\hh \rightarrow 0}  \inf_{\uc:\eqref{eq:u_cent1}}  \frac{1}{T}   L_{\hh, \lceil T/h \rceil }(\uc; \gamma ) \hh - \gamma D . 
\end{align}

To solve \eqref{eq:lower-bound-10}, we first consider the cost-to-go $J_n (X(t_{n}))$ for some $\hh>0$ and $N \in \Z$, \ie 
\begin{align}
\label{eq:cost-to-go_optimal-000}
J_n (X(t_{n}))  :=&   \E\left[ \gamma (X(t_{N}) - \bar X )^2 \right]+\sum_{k = n}^{N-1}  \E\left[  (P(t_{k}) - \bar P )^2   +\gamma (X(t_{k}) - \bar X)^2 \right].
\end{align}
Using mathematical induction, we show below that, at the optimal solution $\uc^*$, the cost-to-go takes the form
 \begin{align}
\label{eq:cost-to-go_optimal}
 &J_n (X(t_{n})) =\E \left[ p_{n}  (X(t_{n}) - \bar X )^2 \right]  +\sum_{k = n}^{N-1} \E[ p_{k+1}  (A(t_n,\hh) - \bar A_{\hh} )^2 ],
\end{align}
where $\{p_k\}$ satisfies the Riccati difference equation 
\begin{align}
\label{eq:rec_equation}
&p_k = p_{k+1} - \frac{ \hh^2 p_{k+1}^2 }{ \hh^2 p_{k+1} + 1} + \gamma , &p_N = \gamma .
\end{align}
First, condition \eqref{eq:cost-to-go_optimal} holds for $n = N$ by the construction of \eqref{eq:cost-to-go_optimal-000}. Second, assume that condition \eqref{eq:cost-to-go_optimal} holds for $n + 1$. Let $ \bar A_{\hh}$ be the stationary mean of $A(t_n,\hh)$.
Recall from \eqref{eq:cost-to-go_optimal-000} that $J_n(X(t_n))$ is the sum of term $n$ to term $N$. 
Thus, it can be decomposed into the term of $n$ and the sum of $n+1$ term to $N$ term, which is $J_{n+1}(X(t_{n+1}))$. So, we have 
 \begin{align}
\nonumber
&J_n (X(t_{n}))  \\
\label{eq:cost-to-go_rec1}
&= \inf_{P(t_n)} \E [ (P(t_{n}) - \bar P )^2   +  \gamma  (X(t_{n}) - \bar X)^2 +J_{n+1} ( X( t_{n+1} ) )] \\
\label{eq:cost-to-go_rec1-2}
&=\inf_{P(t_n)} \E [ (P(t_{n}) - \bar P )^2   +   \gamma  (X(t_{n}) - \bar X)^2 + J_{n+1} (  X(t_n )  + A(t_n, \hh) -  \hh  P(t_{n}  ) ] \\
\label{eq:cost-to-go_rec1-3}
&=\inf_{P(t_n)} \E [ (P(t_{n}) - \bar P )^2   +   \gamma  (X(t_{n}) - \bar X)^2 + J_{n+1} (  X(t_n )  + ( A(t_n, \hh) - \bar A_{\hh} ) -  \hh ( P(t_{n} - \bar P )  ) ] \\
\label{eq:cost-to-go_rec1-4}
&= \inf_{P(t_n)} \E [ (P(t_{n}) - \bar P )^2   +   \gamma  (X(t_{n}) - \bar X)^2 +p_{n+1} (  X(t_n )  + (A(t_n, \hh) - \bar A_{\hh} )-  \hh ( P(t_{n})  -  \bar P ) )^2 ] \\
\nonumber
&\;\;\;\;\;\;\;\;\;\;\;\;\;\;\;+ \sum_{k = n+1}^{N-1} \E[ p_{k+1}  (A(t_k,\hh) - \bar A_{\hh} )^2 ]
\end{align}
where \eqref{eq:cost-to-go_rec1-2} uses relation \eqref{eq:x-discrete}; \eqref{eq:cost-to-go_rec1-3} relies on $ \bar A_{\hh} = \hh \bar P$ from Brumelle's formula; \eqref{eq:cost-to-go_rec1-4} uses the induction hypothesis that the cost-to-go at $n+1$ satisfies \eqref{eq:cost-to-go_optimal}. Note that 
\begin{align}
\label{eq:A-X-indep}
\E[(A(t_n,\hh) - \bar A_{\hh})   X(t_n ) ] = \E[A(t_n,\hh) - \bar A_{\hh}  ] \E[ X(t_n ) ] = 0 , 
\end{align}
where the first equality holds because future arrivals in interval $[t_n , t_{n+1})$ does not depend on past arrivals in interval $[0 , t_n)$, and the second equality is due to $\E[ A(t_n,\hh) - \bar A_{\hh} ] = 0$. Expanding the last quadratic term in \eqref{eq:cost-to-go_rec1} and applying $\E[(A(t_n,\hh) - \bar A_{\hh})   X(t_n ) ] = 0$, we can rewrite \eqref{eq:A-X-indep} into 
\begin{align}
\nonumber
J_n (X(t_{n})) =
& (p_{n+1} +\gamma ) (X(t_{n}) - \bar X)^2 +\sum_{ k = n}^N p_{k+1} \E  (A(t_k,\hh) - \bar A_{\hh} )^2 ] \\
\label{eq:pn}
&+ \inf_{P(t_n)} \{ (1+\hh^2 p_{n+1}) (P(t_n)  -  \bar P )^2 - 2 \hh  \gamma p_{n+1} ( X (t_n) - \bar X) ( P (t_n) - \bar P)  ].
\end{align}
The minimum value of \eqref{eq:pn} is attained by 
\begin{align}
 \label{eq:lower-bound-opt-solution}
P(t_n,\hh)  -  \bar P_{\hh}   = \frac{\hh p_{n} }{1 + \hh^2 p_{n}}  (X (t_n) - \bar X),
\end{align}
and the optimal cost-to-go becomes \eqref{eq:cost-to-go_optimal}, where $p_{n}$ is defined by \eqref{eq:rec_equation}. As $N \rightarrow \infty$, $p_k$ converges to a unique positive scalar 
\begin{align}
\label{eq:rec_equation_limit}
p := \lim_{N \rightarrow \infty } p_k = \frac{\hh^2 \gamma +\hh \sqrt{\gamma} \sqrt{\hh^2 \gamma +4}}{2 \hh^2},
\end{align}
which is also a fixed point of \eqref{eq:rec_equation}~\cite{bertsekas1995dynamic}. 
Taking the limit of $N \rightarrow \infty$ and $\hh \rightarrow 0$ for \eqref{eq:lower-bound-opt-solution} and \eqref{eq:rec_equation_limit}, the infimum of \eqref{eq:lower-bound-10} is attained when
\begin{align}
&P(t) - \bar P = \sqrt{\gamma} \; (X (t) - \bar X) .
\end{align}
Finally, the infimum value of \eqref{eq:lower-bound-10} is computed as 
\begin{align}
&\lim_{T \rightarrow \infty}\inf_{\uc:\eqref{eq:u_cent1}} \frac{1}{T}   \int_{ 0}^{T} \var (  \p(t) )  + \gamma  ( \var(X(t)) -  D )  dt \\
\label{eq:lowerbound_opt_cost2}
&=  \lim_{T \rightarrow \infty} \lim_{\hh \rightarrow 0} \frac{h}{T} \sum_{k =0}^{N-1} \E[ p_{k+1}  (A(t_k,\hh) - \bar A_{\hh} )^2 ] - \gamma D   \\
\label{eq:lowerbound_opt_cost2-2}
&=\sqrt{ \gamma }\pinti  \E[ \e^2  ].
\end{align}
where equality \eqref{eq:lowerbound_opt_cost2} is due to \eqref{eq:cost-to-go_optimal}; and equality \eqref{eq:lowerbound_opt_cost2-2} is derived from \eqref{eq:rec_equation_limit}.

\section{Proof of Corollary \ref{thm:lower_bound2}}
\label{app:corollary2}

Recall from Lemma \ref{thm:lower_bound} that $X(t)$ is the total remaining demands of jobs arriving before $t$. For any time interval $\hh > 0$, $X(t)$ satisfies the following dynamics:
\begin{align}
\label{eq:x-dynamics}
&X(t + \hh )  = X(t )  + A(t,\hh) - P(t,\hh) ,
\end{align}
where $A(t,\hh)$ is the total demand of jobs arriving during time interval $[t, t + \hh)$, and $P(t,\hh)$ is the total amount served during this interval, \ie 
\begin{align}
&A(t,\hh) := \sum_{ \{ k \in \hV : \at_k \in [t, t+\hh) \} } \e_k , \\
&P(t,\hh) := \int_t^{t+\hh} \p(\tau ) d\tau .
\end{align}
let $D_t = \{ k \in \hV: \at_k + \st_k \leq t \}$ be the set of jobs that departs by time $t$. As no job receives more service than its demand, $X(t)$ is bounded from above by 
\begin{align}
  X(t)  &=  \sum_{k \in A_t} \e_k -\int_{\tau \geq t} P(\tau) d\tau \\ 
   &\leq \sum_{k \in A_t } \e_k - \sum_{i \in D_t } \e_k \\
   &\leq  \sum_{k \in A_t \setminus D_t  } \e_k  
   \label{eq:lower-bound-11}
   \end{align}
   where $D_t = \{ k \in \hV: \at_k + \st_k \leq t \}$ is the set of jobs that departs by time $t$. From \eqref{eq:lower-bound-11} and $  X(t) \geq 0$, the variance of $X(t)$ is upper-bounded by  
\begin{align}
\var( X(t) ) &\leq 
\label{eq:lower-bound-0}
 \E[ X(t)^2 ] \\
 \label{eq:lower-bound-1}
 &\leq \E \left[ \left(\sum_{k \in A_t  \setminus  D_t  } \e_k \right)^2  \right] \\
 & = \var  \left( \sum_{k \in A_t  \setminus  D_t  } \e_k  \right) +  \E\left[ \sum_{k \in A_t  \setminus  D_t  } \e_k\right]^2\\
&=\int_{(\e,\st) \in S} \st \e^2 \pinti f (\e,\st) d\e d\st + \left ( \int_{(\e,\st) \in S}\st  \e  \pinti f(\e,\st)  d\e d\st \right)^2\\
 \label{eq:lower-bound-2}
&= \pinti  \E  \left(  \st \e^2   \right) + \left(  \pinti  \E\left[ \st  \e  \right] \right) ^2
\end{align} 
Applying $D = \pinti   \E  \left(  \st \e^2   \right) + \left(  \pinti \E\left[ \st  \e  \right] \right) ^2$ to Lemma \ref{thm:lower_bound}, we obtain \eqref{eq:solution-Q-final}.

\section{Additional performance bound.}
\label{sec:add_performancebound}

Lemma \ref{thm:lower_bound} characterizes the tradeoff between achieving a small variance of $X(t)$ and achieving a small variance of $P(t)$. Plugging in Exact Scheduling's stationary variance of $X$,
\begin{align}
\label{eq:ES_x} 
\var( X ) =  \pinti   \E \left[ \frac{1}{3} \e^2 \st \right] ,
\end{align}
we obtain a competitive-ratio like bound for Exact Scheduling \eqref{eq:exact_scheduling_st}.
\begin{corollary}
 \label{cor:lower_bound1}
Let $\var(P)$ be the stationary variance of $P(t)$ that is attained by Exact Scheduling \eqref{eq:exact_scheduling_st}. Let $\var( P^\dagger )$ be the minimum stationary variance attainable by any centralized algorithm \eqref{eq:u_cent1} with the same level of $\var( X )$ as Exact Scheduling. Then, the following condition holds:
\begin{align}
\label{eq:ES_x1} 
 \var(P)  \leq \frac{4}{3} \frac{ \E[ \e^2 / \st ] \E [ \e^2 \st ] } { \E[ \e^2 ]^2 }   \var(P^\dagger) ,
\end{align}
where the expectations on the right hand side are taken over the arrival distribution. 
\end{corollary}

%
%
%

In the setting of soft service requirements, Generalized Exact Scheduling attains 
\begin{align}
\var(X) =
& \pinti \E\left[ \frac{\e^2 \st}{3} \mathbf{1} \left\{ \frac{\e}{\st} \leq \min \left \{ \frac{\C }{2} ,  \sqrt{\D} \right\}    \right\}   \right] 
+ \pinti  \E\left[ \left( \frac{\C^2 \st^3}{12}-\frac{1}{2} \C \e \st^2+\e^2 \st \right) \mathbf{1}  \left\{ \frac{\e}{\st} > \frac{\C }{2} \geq \sqrt{\D}  \right\}  \right] \\
& + \pinti  \E\left[   \left( \frac{\e^3}{3 \sqrt{\D}} \right) \mathbf{1}  \left\{ \frac{\e}{\st} > \sqrt{\D}> \frac{\C }{2} \right\} 
 \right]    . 
\end{align}
Combining above and Lemma \ref{thm:lower_bound}, we obtain the following corollary. 
\begin{corollary}
 \label{cor:lower_bound1_soft}
Let $\var(P)$ be the stationary variance of $P(t)$ that is attained by Generalized Exact Scheduling \eqref{eq:exact_softdemand+deadline_st}. Let $\var( P^* )$ be the minimum stationary variance attainable by any centralized algorithm of the form \eqref{eq:u_cent1} with the same level of $\var( X )$ as Generalized Exact Scheduling. Then, the following condition holds:
\begin{align}
\label{eq:ES_x2} 
 \var(P)  \leq \frac{ \alpha \beta  } { \E[ \e^2 ]^2 }   \var(P^*) ,
\end{align}
where 
\small
\begin{align*}
\alpha &= \E\left[  \frac{ \e^2 }{ \st }  \mathbf{1} \left\{ \frac{\e}{\st} \leq \min \left \{ \frac{\C }{2} ,  \sqrt{\D} \right\}    \right\}  
+ \C  \left( \sqrt{\D}   - \frac{ \C \st }{ 4 } \right) \mathbf{1}  \left\{ \frac{\e}{\st} > \frac{\C }{2} \geq \sqrt{\D}  \right\} 
+ \left( 2  \sqrt{\D}  \e - \D \st \right) \mathbf{1}  \left\{ \frac{\e}{\st} > \sqrt{\D}> \frac{\C }{2} \right\} 
\right] \\
\beta &=  \E\left[ \frac{\e^2 \st}{3} \mathbf{1} \left\{ \frac{\e}{\st} \leq \min \left \{ \frac{\C }{2} ,  \sqrt{\D} \right\}    \right\}  
+  \left( \frac{\C^2 \st^3}{12}-\frac{1}{2} \C \e \st^2+\e^2 \st \right) \mathbf{1}  \left\{ \frac{\e}{\st} > \frac{\C }{2} \geq \sqrt{\D}  \right\} 
+ \left( \frac{\e^3}{3 \sqrt{\D}} \right) \mathbf{1}  \left\{ \frac{\e}{\st} > \sqrt{\D}> \frac{\C }{2} \right\} 
\right].
\end{align*}
\normalsize
\end{corollary}
Corollary \ref{cor:lower_bound1_soft} bounds the ratio of $ \var(P) $ achievable by Generalized Exact Scheduling (the optimal distributed algorithm) to $\var(P^*)$ achievable by any centralized algorithms. Here, the optimal distributed algorithm is subject to soft service constraints, while the optimal centralized algorithm is subject to the same $\var( X )$ with Generalized Exact Scheduling.


\section{Additional numerical results}
\label{sec:app-simulation}
{\color{black}

Section \ref{sec:emp} shows the empirical performance of different algorithms for typical cases. In this section, we provide more detailed experimental data to support the results in Section \ref{sec:emp}. Figure \ref{fig:example-best-worst} compares how Exact Scheduling and Offline Optimal schedule jobs in two instances: one instance in which Exact Scheduling performed competitively, and another instance in which Exact Scheduling performed poorly. Figure \ref{fig:ES_strict_app} provide a more comprehensive view of Figure \ref{fig:ES_strict} by comparing the algorithms' performance for the arrival distribution of a broader range of parameters. }

\begin{figure}[h!]
\begin{subfigure}[H]{\textwidth}
	\centering
	\caption{Example case when Exact Scheduling performed competitively to Offline Optimal}
	\includegraphics[width=12cm]{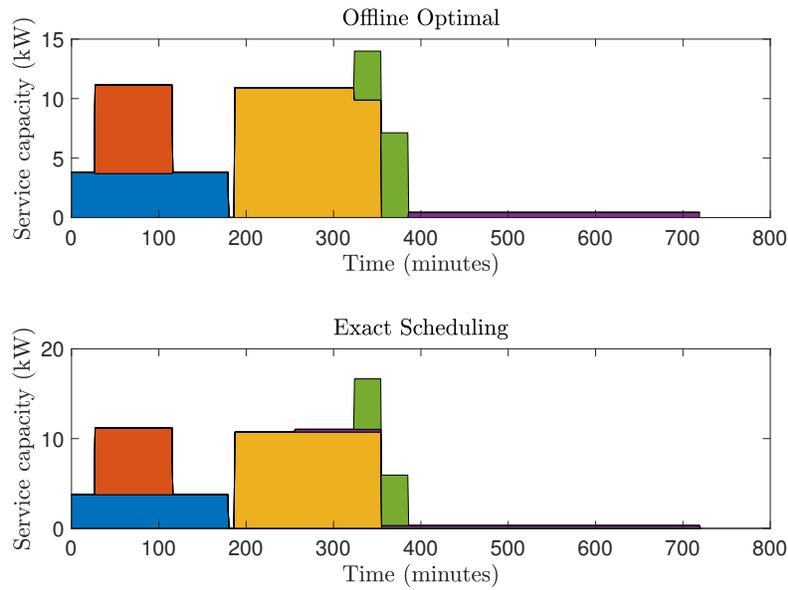}
    \label{fig:ESbest}
\end{subfigure}

\vspace{1cm}
\begin{subfigure}[H]{\textwidth}
	\centering
	\caption{Example case when Exact Scheduling performed poorly to Offline Optimal}
	\includegraphics[width=12cm]{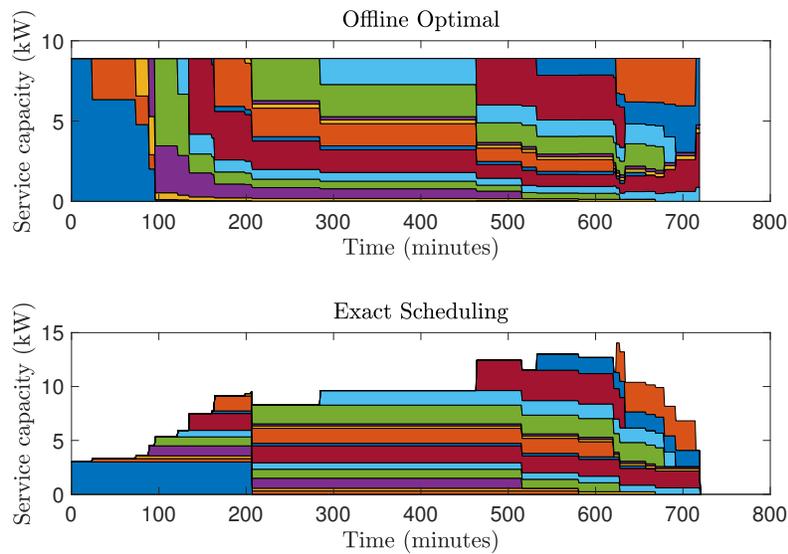}
    \label{fig:ESworst}
\end{subfigure}

\caption{Example cases when Exact Scheduling performs competitively or poorly in comparison to Offline Optimal. Each colored region represents the service rate for one job over its sojourn time, and the height of the colored region (by any color) shows the sum of service rate, \ie the service capacity, at each time. The sum of all service rates at time $t$ is the service capacity $P(t)$. The top plot (a) shows a case when Exact Scheduling performs competitively to Offline Optimal, while the bottom plot (b) shows a case when the Offline Optimal has much better performance than Exact Scheduling. 
}
\label{fig:example-best-worst}

\end{figure}
\newpage

\begin{figure}[h!]
    \centering

\begin{subfigure}[b]{\textwidth}
	\centering
	\caption{Performance in synthetic data generated from arrival distribution I.}

%
%
\definecolor{mycolor1}{rgb}{0.00000,0.44700,0.74100}%
\definecolor{mycolor2}{rgb}{0.85000,0.32500,0.09800}%
\definecolor{mycolor3}{rgb}{0.92900,0.69400,0.12500}%
\definecolor{mycolor4}{rgb}{0.49400,0.18400,0.55600}%
\definecolor{mycolor5}{rgb}{0.46600,0.67400,0.18800}%
\resizebox{.7\columnwidth}{!}{
    \begin{tikzpicture}
    
    \begin{axis}[%
    width=6.049in,
    height=5.094in,
    at={(1.015in,0.688in)},
    scale only axis,
    xtick={10,15,20,25,30},
    xmin=10,
    xmax=30,
    xlabel style={font=\color{white!15!black}},
    xlabel={Mean laxity ($\bar \ell$)},
    ymin=0.5,
    ymax=6,
    ylabel style={font=\color{white!15!black}},
    ylabel={Empirical competitive-ratio},
    axis background/.style={fill=white},
    legend style={at={(0.05,0.95)}, anchor=north west, legend cell align=left, align=left, draw=white!15!black}
    ]
    \addplot [color=mycolor1, line width=1.0pt, mark size=1.7pt, mark=triangle, mark options={solid, rotate=90, black}]
      table[row sep=crcr]{%
    10	2.93932852230498\\
    15	3.8737826475027\\
    20	4.52963202641788\\
    25	5.39872216149675\\
    30	5.53001915983032\\
    };
    \addlegendentry{Immediate Scheduling}
    
    \addplot [color=mycolor2, line width=1.0pt, mark size=4.3pt, mark=diamond, mark options={solid, black}]
      table[row sep=crcr]{%
    10	2.18217234635405\\
    15	2.66739043616621\\
    20	2.99439341789938\\
    25	3.42076597449134\\
    30	3.57139829909578\\
    };
    \addlegendentry{Equal Service}
    
    \addplot [color=mycolor3, line width=1.0pt, mark size=2.5pt, mark=o, mark options={solid, black}]
      table[row sep=crcr]{%
    10	1.86265982839707\\
    15	2.1695762355096\\
    20	2.3473634985732\\
    25	2.59018732364249\\
    30	2.65430588528156\\
    };
    \addlegendentry{Exact Scheduling}
    
    \addplot [color=mycolor4, line width=1.0pt, mark size=2.5pt, mark=x, mark options={solid, black}]
      table[row sep=crcr]{%
    10	1.63834829750703\\
    15	1.91873104947429\\
    20	2.1091224104743\\
    25	2.32346230692476\\
    30	2.41731427400411\\
    };
    \addlegendentry{Exact Scheduling PC}
    
    \addplot [color=mycolor5, line width=1.0pt, mark size=2.5pt, mark=+, mark options={solid, black}]
      table[row sep=crcr]{%
    10	1.54384695227913\\
    15	1.70823191261555\\
    20	1.73474285353245\\
    25	1.86459256742504\\
    30	1.78338760138946\\
    };
    \addlegendentry{Online Optimization MPC}
    
    \end{axis}
    
    \begin{axis}[%
    width=7.806in,
    height=6.25in,
    at={(0in,0in)},
    scale only axis,
    xtick={10,15,20,25,30},
    xmin=0,
    xmax=1,
    ymin=0,
    ymax=1,
    axis line style={draw=none},
    ticks=none,
    axis x line*=bottom,
    axis y line*=left,
    ]
    \end{axis}
\end{tikzpicture}%
}
    \label{fig:ES_strict_app-1}
\end{subfigure}
\begin{subfigure}[b]{\textwidth}
	\centering
	\caption{Performance in synthetic data generated from arrival distribution II.}
%
%
\definecolor{mycolor1}{rgb}{0.00000,0.44700,0.74100}%
\definecolor{mycolor2}{rgb}{0.85000,0.32500,0.09800}%
\definecolor{mycolor3}{rgb}{0.92900,0.69400,0.12500}%
\definecolor{mycolor4}{rgb}{0.49400,0.18400,0.55600}%
\definecolor{mycolor5}{rgb}{0.46600,0.67400,0.18800}%
\resizebox{.7\columnwidth}{!}{
\begin{tikzpicture}

\begin{axis}[%
width=6.028in,
height=4.754in,
at={(1.011in,0.642in)},
scale only axis,
xtick={0,1,2,3,4},
xmin=0,
xmax=4,
xlabel style={font=\color{white!15!black}},
xlabel={Maximum ratio of laxity to demand ($\bar \gamma - 1$)},
ymin=0.5,
ymax=4.5,
ylabel style={font=\color{white!15!black}},
ylabel={Empirical competitive-ratio},
axis background/.style={fill=white},
legend style={at={(0.05,0.95)}, anchor=north west, legend cell align=left, align=left, draw=white!15!black}
]
\addplot [color=mycolor1, line width=1.0pt, mark size=1.7pt, mark=triangle, mark options={solid, rotate=90, black}]
  table[row sep=crcr]{%
0	1.09348788726214\\
1	1.90473167017667\\
2	2.53051555654909\\
3	3.52012366162379\\
4	4.34824578592511\\
};
\addlegendentry{Immediate Scheduling}

\addplot [color=mycolor2, line width=1.0pt, mark size=4.3pt, mark=diamond, mark options={solid, black}]
  table[row sep=crcr]{%
0	1.09412624773304\\
1	1.49538269806261\\
2	1.68233618488351\\
3	2.05072185796516\\
4	2.29423809076993\\
};
\addlegendentry{Equal Service}

\addplot [color=mycolor3, line width=1.0pt, mark size=2.5pt, mark=o, mark options={solid, black}]
  table[row sep=crcr]{%
0	1.0647068232769\\
1	1.30070326857497\\
2	1.46936161886848\\
3	1.67777777038091\\
4	1.84259975906828\\
};
\addlegendentry{Exact Scheduling}

\addplot [color=mycolor4, line width=1.0pt, mark size=2.5pt, mark=x, mark options={solid, black}]
  table[row sep=crcr]{%
0	1.035027628506702\\
1	1.14550533493144\\
2	1.29029325097981\\
3	1.47045077917272\\
4	1.61299601020594\\
};
\addlegendentry{Exact Scheduling PC}

\addplot [color=mycolor5, line width=1.0pt, mark size=2.5pt, mark=+, mark options={solid, black}]
  table[row sep=crcr]{%
0	1.0320016479393\\
1	1.17451169228583\\
2	1.25622810761299\\
3	1.37147594007475\\
4	1.43891958066515\\
};
\addlegendentry{Online Optimization MPC}

\end{axis}

\begin{axis}[%
width=7.778in,
height=5.833in,
at={(0in,0in)},
scale only axis,
xmin=0,
xmax=1,
ymin=0,
ymax=1,
axis line style={draw=none},
ticks=none,
axis x line*=bottom,
axis y line*=left,
legend style={legend cell align=left, align=left, draw=white!15!black}
]
\end{axis}
\end{tikzpicture}%
}
    \label{fig:ES_strict_app-2}
\end{subfigure}
    \caption{Performance comparison of algorithms under strict demand and deadline constraints for varying parameters of arrival distribution. The ratio of each algorithm's empirical variance to the Offline Optimal is averaged over all scheduling instances. The number of instances averaged are $500$ for both plots. In plot (a), the mean laxity refers to parameter $\bar \ell$ in distribution I, and the empirical competitive-ratio for $\bar \ell = 25$ is shown in Figure Figure \ref{fig:ES_strict-2}. In plot (b), the maximum ratio of laxity to demand refers to $\bar \gamma - 1$ in distribution II, and the empirical competitive-ratio for $\bar \gamma - 1 = 1$ is shown in Figure Figure \ref{fig:ES_strict-3}.}
    \label{fig:ES_strict_app}
\end{figure}

\newpage

\end{APPENDICES}

\end{document}